\documentclass[a4paper,10pt,twoside]{scrartcl}

\usepackage{amsmath}
\usepackage{amsthm}
\usepackage{amsxtra}
\usepackage{amssymb}

\usepackage[format=plain, labelfont=bf]{caption}
\usepackage{csquotes}

\usepackage{dsfont}

\usepackage{exscale}

\usepackage{float}

\usepackage{graphicx}
\graphicspath{ {./figures/} }
\usepackage{geometry}
\usepackage{hyperref}

\usepackage[utf8]{inputenc}

\usepackage{latexsym}

\usepackage{mathrsfs}
\usepackage{mathtools}
\usepackage{microtype}

\usepackage[automark]{scrlayer-scrpage}
\usepackage{setspace}
\usepackage{subfigure}

\usepackage{todonotes}

\usepackage{verbatim}

\hypersetup{
	colorlinks=false,
}

\DeclareFontFamily{U}{mathx}{\hyphenchar\font45}
\DeclareFontShape{U}{mathx}{m}{n}{
	<5> <6> <7> <8> <9> <10>
	<10.95> <12> <14.4> <17.28> <20.74> <24.88>
	mathx10
}{}
\DeclareSymbolFont{mathx}{U}{mathx}{m}{n}
\DeclareFontSubstitution{U}{mathx}{m}{n}
\DeclareMathAccent{\widecheck}{0}{mathx}{"71}
\DeclareMathAccent{\wideparen}{0}{mathx}{"75}

\newcommand{\R}{\mathds{R}}

\newcommand{\B}{\mathds{B}}

\newcommand{\Pw}{\mathds{P}}

\newcommand{\Z}{\mathds{Z}}
\newcommand{\E}{\mathds{E}}

\newcommand{\N}{\mathds{N}}

\newcommand{\IS}{\mathbb{S}}

\newcommand{\cA}{\mathcal{A}}
\newcommand{\cB}{\mathcal{B}}

\newcommand{\cD}{\mathcal{D}}

\newcommand{\cF}{\mathcal{F}}
\newcommand{\cG}{\mathcal{G}}

\newcommand{\cM}{\mathcal{M}}

\newcommand{\cN}{\mathcal{N}}

\newcommand{\cP}{\mathcal{P}}
\newcommand{\cS}{\mathcal{S}}
\newcommand{\cR}{\mathcal{R}}

\newcommand{\bfB}{\mathbf{B}}
\newcommand{\bfC}{\mathbf{C}}

\newcommand{\bfH}{\mathbf{H}}
\newcommand{\bfK}{\mathbf{K}}
\newcommand{\bfX}{\mathbf{X}}
\newcommand{\bfY}{\mathbf{Y}}

\newcommand{\symdiff}{\mathbin{\triangle}}

\newcommand{\zero}{\mathbf{0}}
\newcommand{\1}{\mathds{1}}
\DeclareMathOperator{\inte}{Int}

\theoremstyle{definition}
\newtheorem{definition}{Definition}[section]
\newtheorem{example}[definition]{Example}
\newtheorem{remark}[definition]{Remark}
\newtheorem{problem}{Problem}
\newtheorem{conjecture}[problem]{Conjecture}
\theoremstyle{plain}
\newtheorem{theorem}[definition]{Theorem}
\newtheorem{lemma}[definition]{Lemma}
\newtheorem{corollary}[definition]{Corollary}

\newtheorem{proposition}[definition]{Proposition}
\newtheorem{condition}[definition]{Condition}
\newtheorem{assumption}[definition]{Assumption}
\makeatletter
\renewenvironment{proof}[1][\proofname]{%
	\par\pushQED{\qed}\normalfont%
	\topsep6\p@\@plus6\p@\relax
	\trivlist\item[\hskip\labelsep\bfseries#1\@addpunct{.}]%
	\ignorespaces
}{%
	\popQED\endtrivlist\@endpefalse
}
\makeatother

\providecommand{\keywords}[1]
{
	\small	
	\textbf{\textit{Keywords---}} #1
}
\providecommand{\MSC}[1]
{
	\small	
	\textbf{\textit{MSC2020 subject classifications.---}} #1
}

\begin{document}

\title{Contact process in an evolving random environment}	
\author{M. Seiler\footnote{Frankfurt Institute for Advanced Studies, Ruth-Moufang-Straße 1, 60438
		Frankfurt am Main, Germany	E-mail: seiler@fias.uni-frankfurt.de}  \and  A. Sturm\footnote{Institute for Mathematical Stochastics, Georg-August-Universit\"at G\"ottingen, Goldschmidtstr. 7, 37077
		G\"ottingen, Germany E-mail: anja.sturm@mathematik.uni-goettingen.de}}
\maketitle
\begin{abstract}
	In this paper we introduce a contact process in an evolving random environment (CPERE) on a  connected and transitive graph with bounded degree, where we assume that this environment is described through an ergodic spin systems with finite range. We show that under a certain growth condition the phase transition of survival is independent of the initial configuration of the process. We study the invariant laws of the CPERE and show that under aforementioned growth condition the phase transition for survival coincides with the phase transition of non-triviality of the upper invariant law. Furthermore, we prove continuity properties for the survival probability and derive equivalent conditions for complete convergence, in an analogous way as for the classical contact process. We then focus on the special case, where the evolving random environment is described through a dynamical percolation. We show that the contact process on a dynamical percolation on the $d$-dimensional integer dies out almost surely at criticality and complete convergence holds for all parameter choices. In the end we derive some comparison results between a dynamical percolation and ergodic spin systems with finite range such that we get bounds on the survival probability of a contact process in an evolving random environment and we determine in this case that complete convergence holds in a certain parameter regime.
\end{abstract}
\\
\MSC{Primary 60K35; Secondary 05C80, 82C22.}

\section{Introduction}
The contact process (CP) is a particularly simple example of an epidemiological model, which was first introduced by Harris \cite{harris1974contact}. 
This process models the spread of an infection over time in a spatially structured population, whose structure is given through a graph $G=(V,E)$. The vertex set $V$ labels the individuals and two individuals $x,y\in V$ are considered neighbours, i.e.~they have physical contact, if there exists an edge $\{x,y\}\in E$ between them. If an individual is infected it can pass on its infection to its neighbours. In addition, it can recover spontaneously from the infection.

Many variations of the classical contact process exist which try to incorporate more realistic assumptions. One aspect is for example that the underlying spatial structure given by $G$ is in applications often not explicitly known and may also change dynamically in time. This is for example the case in large social networks that may also describe the possibility of physical contact. But changes in $G$ have a drastic impact on the long-term behaviour of the model.

One way to model uncertainty in the spatial structure is to introduce a random environment. Probably the first to consider a contact process in a random environment were Bramson, Durrett and Schonmann \cite{bramson1991contact}. Since then there has been a lot of effort in this direction, see for example \cite{liggett1992survival}, \cite{klein1994extinction}, \cite{yao2012complete}, \cite{xue2014upper} and \cite{garet2012asymptotic}. They all consider contact processes in a \textit{static} random environment, i.e.~the random environment is random but fixed for the whole time horizon.

More recently, the contact processes has also been considered in a dynamical or time evolving random environments. One of the first to explicitly study a contact process with dynamical rates was Broman \cite{broman2007stochastic}. Other works have been for example \cite{steif2007critical}, \cite{linker2019contact} and \cite{hilario2021results}. A related variation are multi type contact process see for example \cite{durrett1991there}, \cite{durrett1991complete}, \cite{remenik2008contact} and \cite{kuoch2016phase}.

The general observation is that a (random dynamical) change of the underlying spatial structure can have a drastic impact on the system's long-term behavior such as survival of the infection. This has also been the observation in applications, be it from real data or simulations, most recently in the global Covid-$19$ pandemic, see for example \cite{dehning2020inferring}.

\smallskip
In this article we study \emph{contact processes in an evolving random environment} (CPERE) in some generality. We assume that the graph $G=(V,E)$ is a connected and transitive graph with bounded degree. The CPERE $(\bfC,\bfB)=(\bfC_t,\bfB_t)_{t\geq 0}$ is a Feller process on $\cP(V)\times \cP(E)$, where $\cP(V)$ and $\cP(E)$ are the power sets of $V$ and $E$. We call the process $\bfC$  with values in $\cP(V)$ the \emph{infection process}.  If $x\in\bfC_{t}$, then we call $x$ \textit{infected} at time $t$. We call the process $\bfB$  with values in $\cP(E)$ the \emph{background process}, since it describes the evolving random environment. If $e\in \bfB_t$ we call $e$ \emph{open} at time $t$ and \emph{closed} otherwise. We assume that $\bfB$ is an autonomous Feller process.
Given $\bfB$ currently in state $B$ the transitions of the infection process $\bfC$ currently in state $C$ are for all $x\in V$, 
\begin{align}\label{InfectionRatesWithBackground}
\begin{aligned}
%\bfC_{t-}=
C&\to C\cup \{x\}	\quad \text{ at rate } \lambda \cdot|\{y\in C:\{x,y\}\in %\bfB_{t-}
B\}| \text{ and }\\
%\bfC_{t-}=
C&\to C\backslash \{x\}	\quad\; \; \text{ at rate } r,
\end{aligned}
\end{align}
where $\lambda>0$ denotes the \emph{infection rate} and $r>0$ the \emph{recovery rate}.

We equip $\cP(V)\times \cP(E)$ with the topology which induces point wise convergence. This means if $\big((C_n,B_n)\big)_{n\in \N}$ is a sequence in  $\cP(V)\times \cP(E)$, then $(C_n,B_n)\to(C,B)$ as $n\to \infty$ if and only if $\1_{\{(x,e)\in (C_n,B_n)\}}\to \1_{\{(x,e)\in (C,B)\}}$ as $n\to \infty$ for every $(x,e)\in V\times E$. Furthermore, we denote by $``\Rightarrow"$ the weak convergence of probability measures on $\cP(V)\times \cP(E)$.

We will indicate the initial configuration $(C,B)$ by adding it as a superscript to the process, i.e.~$(\bfC^{C,B},\bfB^{B})$ or, if more convenient, to the law $\Pw$, i.e.~$\Pw^{(C,B)}$. If the initial configuration is a distribution $\mu$ then we will also write $\Pw^{\mu}$. Note that if $\mu=\delta_{C}\otimes\mu_2$, where $\mu_2$ is a probability measure on $\cP(E)$, then we abuse the notation slightly and write $\Pw^{(C,\mu_2)}$. Furthermore we add the model parameter as subscripts to the law $\Pw$, i.e.~$\Pw_{\lambda,r}$ whenever needed for clarity.

The CP can be constructed via the so called graphical representation, where one draws infection and recovery events according to a Poisson point process, which are respectively depicted by arrows pointing from an individual $x$ to a neighbour $y$ and by crosses at a vertex $x$.
Now if $x$ is infected the arrow causes the infection of $y$. On the other hand a cross at $x$ leads to the recovery of $x$, see Figure~\ref{fig:CPGraphRepVis:a}.  
The CPERE is essentially constructed in the same way as the CP with the difference that we incorporate the background into the graphical representation, see  Figure~\ref{fig:CPGraphRepVis:b}. This means that 
an infection arrow from $x$ to $y$ can only transmit an infection at a time $t$ if the edge is open, i.e.\ $\{x,y\}\in\bfB_t$. 
For the formal description of the graphical representation we refer to Section~\ref{GraphRepresentationCPERE}.
\begin{figure}[t]
	\centering 
	\subfigure[The arrows 	correspond to a possible transmission of an infection %from $x$ to $y$ 
	and the crosses correspond to possible recovery of the respective site. The red lines indicate the infection paths. ]{\label{fig:CPGraphRepVis:a}\includegraphics[width=68.5mm]{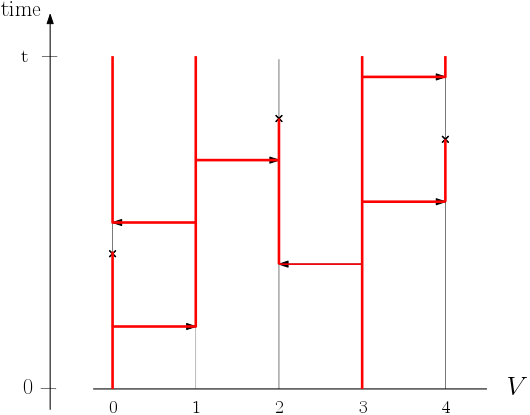}}\hfill
	\subfigure[Grey areas indicate that an edge is closed with respect to the background. Infection arrows in a grey area are ignored. The red lines again indicate the infection paths.  ]{\label{fig:CPGraphRepVis:b}\includegraphics[width=68.5mm]{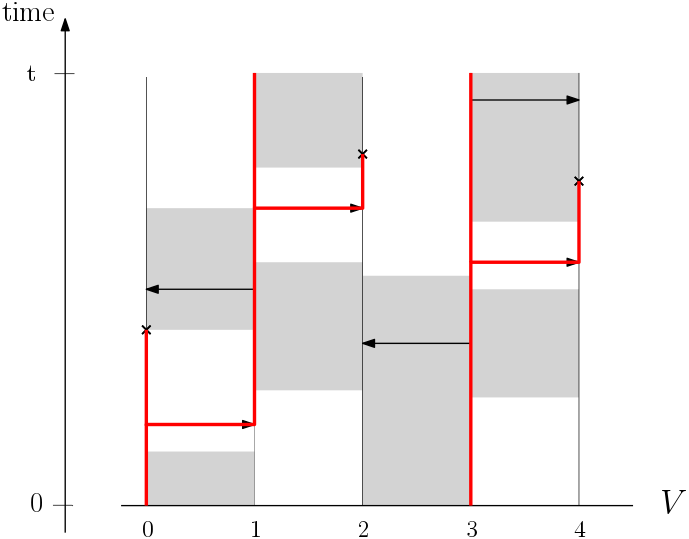}}
	\caption{Graphical representation of the CP in (a) and the CPERE in (b).}
	\label{fig:CPGraphRepVis}
\end{figure}

One of the key quantities for models describing the spread of infections is the \emph{survival probability} of the infection process $\bfC$, which we denote by
\begin{align*}
\theta(C,B):=\theta(\lambda,r,C,B):=\Pw^{(C,B)}_{\lambda,r}\big(\bfC_t\neq \emptyset\,\, \forall t\geq 0\big)
\end{align*}
for $C\subset V$ and $B\subset E$. We are interested in a phase transition dependent on the parameters $(\lambda,r)$ from a zero to positive survival probability. 
Since the survival probability is monotone in  $\lambda$ we can define the \emph{critical infection rate for survival} by
\begin{align*}
\lambda_c(r,C,B)&:=\inf\{\lambda>0:\theta(\lambda,r,C,B)>0\}.
\end{align*}

The CPERE is well defined for a fairly general class of interacting particles systems acting as the background process $\bfB$. We focus on the case where the background is a \emph{spin system} on $\cP(E)$
with generator of the form
\begin{align*}
\cA_{\text{Spin}} f(B)=\sum_{e\in E} q(e,B)\big(f(B\symdiff\{e\})-f(B)\big),
\end{align*} 
where $q(e,B)$ is the flip rate of $e$ with respect to the current configuration $B\subset E$ and $\symdiff$ is the symmetric difference of sets. 
Note that this is a spin system on the edge set $E$. In order to use the usual notation for interacting particle systems in this setting 
we additionally equip  $E$ with a spatial structure by considering the line graph $L(G)$. Here, the original edge set $E$ is considered to be the vertex set and edges $e_1,e_2\in E$ are defined to be adjacent if they have a vertex in common, i.e.~there exists $x\in V$ such that $x\in e_1,e_2$. Along with the original graph the line graphs is also a connected and transitive graph with bounded degree.
Let $\B^L_n(e)$ denote the ball with center $e\in E$ of radius $n$ with respect to the graph distance of $L(G)$. We assume that the spin system satisfies the following three properties.
\begin{enumerate}
	\item It is \emph{attractive}, i.e.~the spin rate $q(\cdot,\cdot)$ satisfies that if $B_1\subset B_2$, then
	\begin{align*}
	q(e,B_1)\leq q(e,B_2) \text{ if }  e\notin B_2 \quad\text{ and }\quad
	q(e,B_1)\geq q(e,B_2)  \text{ if }  e\in B_1.
	\end{align*}
	\item It is \emph{translation invariant}, i.e.~if $\sigma$ is a graph automorphism then
	\begin{align*}
	q(e,B)=q(\sigma(e),\sigma(B))\qquad \text{ for all } B\subset E.
	\end{align*} 
	\item The spin system is of \emph{finite range}, i.e.~there exists a constant $R\in \N$ such that
	\begin{align*}
	q(e,B)=q(e,B\cap \B^L_R(e))
	\end{align*}
	for all $e\in E$ and $B\subset E$. We call such a spin system of range $R$.
\end{enumerate}
In order to formulate further assumption on the background process $\bfB$ we introduce the \emph{exponential growth} of the graph $G$ by
\begin{equation*}
\rho:=\lim_{n\to \infty}\frac{1}{n}\log(|\B_n(x)|),
\end{equation*}
where $\B_n(x)$ denotes the ball of radius $n$ with $x\in V$ as centre with respect to the graph distance $d(\cdot,\cdot)$. Note that the existence of the limit follows by subadditivity and since $G$ is transitive $\rho$ does not depend on the choice of $x$. If $\rho=0$ we call $G$ of \emph{subexponential} growth. Next we define the permanently coupled region of the background $\bfB$ at time $t$ through
\begin{equation}\label{DefinitionPermanetlyCoupledRegion}
\Psi'_{t}:=\{e\in E: e\notin \bfB_s^{B_1}\symdiff\bfB_s^{B_2}\enskip \forall B_1,B_2\subset E,\forall s\geq t\},
\end{equation}
where $t\geq 0$. In order to define $\Psi'_{t}$ the processes $\bfB^{B_1}$ and $\bfB^{B_2}$ need to be defined on the same probability space for all initial backgrounds $B_1,B_2\subset E$. This is indeed possible since all finite range spin systems can be constructed via a graphical representation, see Section~\ref{GraphRepSpinSystemCose}, that couples $\bfB_s^{B}$ for different initial backgrounds $B$.
\begin{assumption}\label{AssumptionBackground}
	The background $\bfB$ satisfies the following assumptions:
	\begin{enumerate}
		\item[$(i)$] $\bfB$ is ergodic, i.e there exists a unique invariant law $\pi$ such that $\bfB^B_t\Rightarrow \pi$ as $t\to \infty$ for all $B\subset E$.
		\item[$(ii)$] There exist constants $T,K,\kappa>0$ such that $\Pw(e\notin \Psi'_t)<K\exp(-\kappa t)$ for every $e\in E$ and for all $t\geq T$.
		\item[$(iii)$] $\bfB$ is a reversible Feller process.
	\end{enumerate}
\end{assumption}
Loosely speaking since we assume that $\bfB$ is ergodic in (i) 
then the expansion speed of the permanently coupled region in $(ii)$ gives us a rough insight on how fast the background process convergences to the invariant law $\pi$.

We list now some examples of spin systems which satisfy these assumptions. Note that with the trivial example $\bfB_t\equiv E$ for all $t\geq 0$ we recover the CP from the CPERE. 
\begin{example}\label{ThreeParticularSpinSystems}
    Let $\cN_e^L$ denote the set of neighbours of $e$ in the line graph $L(G)$.
	\begin{enumerate}
		\item[$(i)$] One 
		of the simplest choices is \emph{dynamical percolation}, which %This system 
		will be our main example. Dynamical percolation updates every edge independently. 
		Hence, the background $\bfB$ is a Feller process with transition
		\begin{align*}
		%\bfB_{t-}=
		B&\to B\cup \{e\}	\quad \text{ at rate } \alpha \text{ and } \\
		%\bfB_{t-}=
		B&\to B\backslash \{e\}	\quad \,\,\,\text{ at rate } \beta,
		\end{align*}
		where $\alpha,\beta>0$.
		\item[$(ii)$] Next we consider a \emph{noisy voter model} on $G=(V,E)$ with
		\begin{align*}
		V=\Z \quad \text{ and } \quad E=\{\{x,y\}\subset \Z: |x-y|=1\}.
		\end{align*}
		In this case $L(G)$ is again a $1$-dimensional nearest neighbour integer lattice just like $G$. The background $\bfB$ has transitions
		\begin{align*}
		%\bfB_{t-}=
		B&\to B\cup \{e\}	\quad \text{ at rate } \frac{\alpha}{2}+\beta|B\cap \cN^L_e| \text{ and } \\
		%\bfB_{t-}=
		B&\to B\backslash \{e\}	\quad \,\,\,\text{ at rate } \frac{\alpha}{2}+\beta|B^{c}\cap \cN^L_e|,
		\end{align*}
		where $\alpha,\beta>0$.
		\item[$(iii)$] The third example is the \emph{ferromagnetic stochastic Ising model} with inverse temperature $\beta$. Here, the transitions of $\bfB$ are
		\begin{align*}
		B&\to B\cup \{e\}	\quad \text{ at rate } 1-\tanh\big(\beta\big( |\cN_e^L|-2|B\cap \cN_e^L|\big)\big) \text{ and }\\
		%\bfB_{t-}=
		B&\to B\backslash \{e\}	\quad \,\,\,\,\text{ at rate } 1-\tanh\big(\beta\big( |\cN_e^L|-2|B^c\cap \cN_e^L|\big)\big),
		\end{align*}
		where $0<\beta<\tfrac{1}{4}\log\big(\frac{|\cN_e^{L}|+2-\rho}{|\cN_e^{L}|-2+\rho})$.
	\end{enumerate}
\end{example}
\begin{remark}\label{SufficientCond}
    It is by no means trivial to determine if a given spin systems satisfies any of the three Assumptions~\ref{AssumptionBackground} $(i)-(iii)$. Define the two constants
    \begin{align*}
	    M:=\sum_{a\in \cN^{L}_e}\sup_{B\subset E}|q(e,B)-q(e,B\symdiff\{a\})|\,\, \text{ and }\,\,		\varepsilon:=\inf_{B \subset E} |q(e,B)+q(e,B\symdiff\{e\})|.
	\end{align*}
	One can show that a sufficient condition for $(i)$ and $(ii)$ to be satisfied is that $\varepsilon-M>\rho$, see \cite[Corollary~1.4.3]{seiler2021}. On the other hand, the class of \emph{stochastic Ising models} is a natural choice in our setting since models of this type satisfy reversibility by definition, see \cite[Section IV.2]{liggett2012interacting} for the general definition of a stochastic Ising model.
\end{remark}

\section{Main results}\label{Sec:MainResults}
Our first aim is to deduce some sufficient conditions such that the critical infection rate for survival does not depend on the initial configuration $(C,B)$ as long as we have finitely many initially infected vertices, i.e. $C$ is finite and non-empty. We often use the case where the background is started stationary, i.e.~$\bfB_0\sim\pi$, as a reference, and thus denote the survival probability in this case by
\begin{equation*}
\theta^{\pi}(\lambda,r,C):=\Pw^{(C,\pi)}_{\lambda,r}\big(\bfC_t\neq \emptyset\,\, \forall t\geq 0\big).
\end{equation*}
%i.e.~$\bfB_0\sim \pi$, where $C\subset V$, $B\subset E$ and $\lambda,r>0$. 
Furthermore we denote the associated critical infection rate by
\begin{equation*}
\lambda_c^{\pi}(r):=\inf\{\lambda> 0:\theta^{\pi}(\lambda,r,\{x\})>0\}.
\end{equation*}
Observe that $\lambda_c^{\pi}(r)$ is independent of the choice of $x\in V$ by translation invariance.
Furthermore, we note that by additivity of the contact process (conditional on any given background) we have that $\theta^{\pi}(\lambda,r,\{x\})>0$ if and only if $\theta^{\pi}(\lambda,r,C)>0$
for any finite and non-empty set $C$. Thus, the critical value  $\lambda_c^{\pi}(r)$ is the same for any such finite and non-empty initial configuration $C$. Note that the same argument applies for any initial background that is translation invariant, in particular when starting the background in $\emptyset$ or $E$.

We denote by $\cN_x$ the set of neighbours of $x\in V$. Let $c_1(\lambda,\rho)$
be the unique solution of
\begin{align}\label{UniqueSolution2}
c\lambda-1-\log(c\lambda |\cN_x|))=\rho
\end{align}
which satisfies $0<c_1(\lambda,\rho)\leq \frac{1}{\lambda}$, where $\lambda>0$ and $x\in V$.
We will see that the constant $c_1(\lambda,\rho)^{-1}$ is an upper bound of the asymptotic expansion speed of the set of all infections by showing (see Lemma \ref{LinearGrowth}) that for all $c<c_1(\lambda,\rho)$ and $x\in V$
		\begin{equation*}
			\Pw(\exists s\geq0: \widetilde{\bfC}^{\{x\}}_{t}\subset \B_{\lfloor \frac{t}{c}\rfloor}(x)\,\, \forall t\geq s)=1,
		\end{equation*}
where $\widetilde{\bfC}$ is the CP without recovery which always upper bounds the (coupled) CPERE ${\bfC}$. Note that $\widetilde{\bfC}$ is also called the SI-Model or alternatively the Richardson-Model, and that the first part of our proof follows the ideas of \cite{durrett1988lecture} for this model.

On the other hand $\kappa\rho^{-1}$ is a lower bound of the asymptotic expansion speed of the permanently coupled region around any edge $e$. Proposition \ref{ExpansionSpeedPerCouRegion} essentially states that
for all $c>\kappa^{-1}\rho$ and  $e\in E$
		\begin{equation*}
			\Pw(\exists s\geq 0: \B^L_{\lfloor \frac{t}{c} \rfloor}(e)\subset\Psi_{t}' \,\,\forall t\geq s)=1.
		\end{equation*}
Intuitively, this can be understood by observing that $\kappa$ is a measure for the speed of coupling while $\rho$ is a measure for the number of edges in a ball (both in an exponential sense). 		
	Thus, the inequality $c_1(\lambda,\rho)>\kappa^{-1}\rho$ implies that asymptotically the growth speed of the infection $\bfC$ is slower than the expansion of the permanently coupled region $\Psi'$
since in this case the two statements above can be combined into
	\begin{equation*}
			 \Pw(\exists s\geq 0: \widetilde{\bfC}^{C}_t\subseteq\Psi'_t\,\, \forall t\geq s )= 1,
		\end{equation*}
for any finite and non-empty $C$, see Theorem \ref{InfectionVSBackground}. We refer to Section \ref{Sec:Independece} for more discussion of the growth condition and also an illustration in Figure \ref{fig:Growth}.

\smallskip
Our result now states that if this growth condition holds the critical infection rate is independent of the initial configuration.

\begin{theorem}\label{IndependencyCirticalRateSurvival}
	Suppose that Assumptions~\ref{AssumptionBackground} $(i)$ and $(ii)$ are satisfied. If there exists a non-empty and finite set $C'\subset V$ and a $B'\subset E$ such that $c_1(\lambda_c(r,C',B'),\rho)>\kappa^{-1}\rho$, then it follows that $\lambda_c(r,C,B)=\lambda^{\pi}_c(r)$ for all non-empty and finite $C\subset V$ and all $B\subset E$. In this case we denote the critical infection rate simply by $\lambda_c(r)$.
\end{theorem}
\begin{remark}
	If we consider graphs with subexponential growth, i.e.~$\rho=0$, the inequality $c_1(\lambda,\rho)>\kappa^{-1}\rho$ is satisfied for all $\lambda>0$. Thus, on these graphs the critical infection rate is always independent of the initial configuration as long as Assumptions~\ref{AssumptionBackground} $(i)$ and $(ii)$ are satisfied.
\end{remark}

Next, we consider the relationship to the critical infection rate for non-triviality of the \emph{upper invariant law}. As in the CP  we can conclude by standard methods the existence of this upper invariant law $\overline{\nu}= \overline{\nu}_{\lambda,r}$ with $(\bfC_t^{V,E},\bfB_t^E)\Rightarrow \overline{\nu}$ as$t \rightarrow \infty.$
Since the upper invariant law is the largest invariant law in the stochastic order 
the equality $\overline{\nu}=\delta_{\emptyset}\otimes \pi$, where the right hand side is the trivial invariant law, is equivalent to ergodicity of the system, i.e.~ that there exists a unique invariant law which is the weak limit of the process. Since the upper invariant law $\overline{\nu}_{\lambda,r}$ is monotone in the stochastic order with respect to $\lambda$ (and also $r$) we can define the critical infection rate for non-triviality of $\overline{\nu}$ by
\begin{equation*}
\lambda'_{c}(r):=\inf\{\lambda>0:\overline{\nu}_{\lambda,r}\neq\delta_{\emptyset}\otimes \pi\}.
\end{equation*}
The next result  connects this phase transition to the already known phase transition between certain extinction and survival (with positive probability) of the infection in the population. 
\begin{theorem}\label{CrticalValuesAgree}
	Suppose that Assumption~\ref{AssumptionBackground} is satisfied. Then $\lambda'_{c}(r)=\lambda^{\pi}_{c}(r)$. If additionally $c_1\big(\lambda^{\pi}_{c}(r),\rho\big)>\kappa^{-1}\rho $, then $\lambda'_{c}(r)=\lambda_{c}(r)$.
\end{theorem}
Next we state results on the continuity of the survival probability with respect to its parameters.
This is important for subsequently studying the behavior at the critical point and (complete) convergence. 

If $C$ is empty then $\theta(C,B)=0$ and if $|C|= \infty$ then $\theta(C,B)=1$ for any $B$.
Thus, the survival probability is obviously continuous. Therefore, we will only consider the case where $C$ is non-empty and finite, and
define for such initial configurations $(C,B)$ 
the region of survival by
\begin{align}\label{SurvivalRegion}
\cS(C,B):=\{(\lambda,r)\in (0,\infty)^2:\theta(\lambda,r,C,B)>0\}.
\end{align}
On the complement $\big(\cS(C,B)\big)^c$ we see that the survival probability is again $0$, and thus obviously continuous. So the only interesting question is if $\theta(\cdot,C,B)$ is continuous on $\cS(C,B)$ and its boundary. Unfortunately, on general graphs we are not able to determine whether the survival probability is continuous on the whole survival region. However, we can assure continuity on the interior of the parameter set
\begin{align}\label{SurvivalRegionWithGrowth}
\cS_{c_1}:=\{(\lambda,r): \exists \lambda'\leq \lambda \text{ s.t. } (\lambda',r)\in\cS(\{x\},\emptyset) \text{ and }  c_1(\lambda',\rho)>\kappa^{-1}\rho\},
\end{align}
which contains all parameters $(\lambda,r)$ such that a $\lambda'\leq \lambda$ exists for which survival is still possible and the aforementioned growth condition is satisfied. 

By Theorem~\ref{IndependencyCirticalRateSurvival} if follows that if the graph $G$ is of subexponential growth, i.e.~$\rho=0$, then $\cS(C,B)=\cS_{c_1}$ for all $(C,B)$ with $C$ finite and non-empty. In this case we drop the initial configuration and write the survival region as $\cS$. We denote by $\inte(U)$ the interior of a set $U\subset \R^d$, i.e.~the largest open set which is contained in $U$.
\begin{proposition}\label{ContinuityTheorem}
	Let $C\subset V$ be finite and non-empty and $B\subset E$. Suppose Assumption~\ref{AssumptionBackground} is satisfied, then 
	\begin{itemize}
	    \item[$(i)$] the survival probability $\theta(\cdot,C,B)$ is continuous on $\inte(\cS_{c_1})$.
	    \item[$(ii)$] If additionally $\rho=0$ and $\theta(\lambda,r,C,B)=0$ for all $(\lambda,r)\in (0,\infty)^2\backslash\inte(\cS)$, then the survival probability $\theta(\cdot,C,B)$ is continuous on $(0,\infty)^2$.
	\end{itemize}
\end{proposition}
With this result we are now able to determine two conditions which are equivalent to complete convergence of the CPERE.
In order to formulate the result  we abuse notation somewhat by writing
\begin{align*}
\{x\in \bfC_t \text{ i.o.}\}=\{x\in \bfC_t \text{ for a sequence of times } t\uparrow \infty\},
\end{align*}
where i.o.~is short for ``infinitely often".
\begin{theorem}\label{CompleteConvergenceComplete}
	Let Assumption~\ref{AssumptionBackground} be satisfied. 
	Suppose there exists a $\lambda'\leq \lambda$ such that $c_1(\lambda',\rho)>\kappa^{-1}\rho$, and that
	\begin{equation}\label{ConvergenceCond1.1}
	\Pw_{\lambda',r}^{(C,B)}(x\in \bfC_t \text{ i.o.})=\theta(\lambda',r,C,B)
	\end{equation}
	for all $x\in V$, $C\subset V$ and $B\subset E$ as well as
	\begin{align}\label{ConvergenceCond2.1}
	\lim_{n\to\infty}\limsup_{t\to \infty}\Pw_{\lambda',r}(\bfC_t^{\B_n(x),\emptyset}\cap \B_n(x)\neq\emptyset)=1
	\end{align}
	for any $x\in V$. Then complete convergence is satisfied, %for $\lambda$ and $r$
	i.e.~
	\begin{equation}
	\label{CompleteConvergence}
        (\bfC_t^{C,B},\bfB_t^B)\Rightarrow \theta(\lambda,r,C,B)\overline{\nu}+[1-\theta(\lambda,r,C,B)](\delta_{\emptyset}\otimes\pi)\quad \text{ as } t\to\infty.
    \end{equation}
	Conversely if \eqref{CompleteConvergence} holds for $(\lambda,r)\in \cS_{c_1}$ and additionally $\overline{\nu}\neq \delta_{\emptyset}\otimes \pi$, then \eqref{ConvergenceCond1.1} and \eqref{ConvergenceCond2.1} are satisfied.
\end{theorem}

Thus, our last main results, for which we apply Theorem \ref{CompleteConvergenceComplete}, 
concern our main example the \emph{contact process on a dynamical percolation} (CPDP) for which the background is given by the dynamical percolation introduced in Example~\ref{ThreeParticularSpinSystems} $(i)$. Here we have two additional parameters $\alpha$ and $\beta$ to consider, which are the opening and closing rates of the edges. Note that the dynamical percolation $\bfB$ satisfies Assumption~\ref{AssumptionBackground} for all $\alpha,\beta>0$. Since we have two additional parameters for the CPDP define, analogously to \eqref{SurvivalRegion}, the survival region for the CPDP by
\begin{equation}\label{SurvivalRegionCPDP}
\cS(C,B):=\{(\lambda,r,\alpha,\beta)\in (0,\infty)^4:\theta\big(\lambda,r,\alpha,\beta,C,B\big)>0\}.
\end{equation}
In this special case we are able to show that even if $c_1(\lambda,\rho)\leq \kappa^{-1}\rho$ the interior and the closure of the survival region is independent of the initial configuration.
\begin{proposition}\label{IndependenceGrowthCond}
	Let $x\in V$ and $(\bfC,\bfB)$ be a CPDP.
	Then for any $C\subset V$ finite and non-empty and any $B \subset E$ we have that $\inte(\cS(C,B))=\inte(\cS^{\pi})$
	and $\overline{\cS(C,B)}=\overline{\cS^{\pi}}.$
\end{proposition}

In Section~\ref{Sec:CPDP} we even show a slightly stronger but somewhat more technical statement, see Proposition~\ref{IndependenceGrowthCond2}.
We already note here that only in the cases where the boundary of %$\mathring{\cS}$

\smallskip
\noindent
For the next results we assume the underlying graph to be the $d$-dimensional lattice, i.e.  
\begin{equation*}
V=\Z^d \quad \text{ and } \quad E=\{\{x,y\}\subset \Z^d: ||x-y||_1=1\},
\end{equation*}
where $||\cdot||_1$ denotes the $1$-norm. We denote by $\zero\in \Z^d$ the $d$-dimensional vector of zeros. Note that the $d$-dimensional lattice is of subexponential growth, and thus the growth condition $c_1(\lambda,\rho)>\kappa^{-1}\rho$ is always satisfied.
Thus, the critical infection rate is given through  
\begin{equation*}
\lambda_c(r,\alpha,\beta)=\inf\{\lambda>0:\theta(\lambda,r,\alpha,\beta,\{\zero\},\emptyset)>0\}=\inf\{\lambda>0:\theta(\lambda,r,\alpha,\beta,C,B)>0\},
\end{equation*}
for any $C$ non-empty and finite as well as arbitrary $B$.
Likewise, the survival region does not depend on the initial condition, and thus we denote it simply by $\cS$.

Due to the independence of the dynamics of the edges in dynamical percolation
we can explicitly state the invariant law $\pi=\pi_{\alpha,\beta}$ of the background process.
Namely, under $\pi_{\alpha,\beta}$ the states of the edges are described by i.i.d. Bernoulli random variables 
with parameter $\frac{\alpha}{\alpha+\beta}$, i.e.~for every $e\in E$
\begin{align*}
\pi(\{B\subset E: e\in B\})=\frac{\alpha}{\alpha+\beta} \quad\text{ and } \quad  \pi(\{B\subset E: e\notin B\})=\frac{\beta}{\alpha+\beta}.
\end{align*}
We adapt the techniques developed by \cite{bezuidenhout1990critical} to the CPDP and use them show the following results.

\begin{theorem}\label{NoSurvivalAtCriticality}
	The CPDP on $\Z^d$ goes a.s. extinct at criticality, i.e. 
	\begin{align*}
	\theta\big(\lambda,r,\alpha,\beta,\{\zero\},\emptyset\big)=0
	\end{align*}
	for all $(\lambda,r,\alpha,\beta)\in (0,\infty)^4\backslash %\mathring{\cS}
	\inte(\cS)$.
\end{theorem}
With the help of Theorem~\ref{NoSurvivalAtCriticality} and results analogous to Proposition~\ref{ContinuityTheorem} we can then extend the results on the continuity of the survival probability to the whole parameter regime $(0,\infty)^4$.
\begin{proposition}\label{JointContiuityOfCPDP}
	Let $C\subset V$ and $B\subset E$. For the CPDP on $\Z^d$ the survival probability is continuous, i.e.
	\begin{equation*}
	(\lambda,r,\alpha,\beta)\mapsto \theta(\lambda,r,\alpha,\beta,C,B)
	\end{equation*}
	is continuous seen as function from $(0,\infty)^4$ to $[0,1]$.
\end{proposition}

\begin{theorem}\label{CPDPCompleteConvergence}
	The CPDP on $\Z^d$  %$(\bfC,\bfB)$
	satisfies complete convergence as stated in \eqref{CompleteConvergence}.
\end{theorem}
One reason for studying the CPDP is that the background is still simple enough so that we can adapt some techniques known for the CP and show the above results. But another reason is that we can compare some CPERE with more general background processes to it: Let $(\bfC,\bfB)$ be a CPERE such that $\bfB$ is a spin system with rate $q(\cdot,\cdot)$. Assume that $\bfB$ is of range $R$, i.e.~$q(e,B)=q(e,B\cap \B^L_R(e))$. Set $\cN^L_e(R):=\B^L_R(e)\backslash \{e\}$ and define the minimal and maximal spin rates by
\begin{align}\label{MaximalAndMinimalRates1}
\begin{aligned}
\alpha_{\min}&:=\min_{F\subset \cN^L_e(R) }q(e,F),\qquad  \beta_{\min}:=\min_{F\subset\cN^L_e(R) }q(e,F\cup \{e\})\\
\alpha_{\max}&:=\max_{F\subset \cN^L_e(R) }q(e,F)\,\, \text{ and }\,\, \beta_{\max}:=\max_{F\subset \cN^L_e(R) }q(e,F\cup \{e\}).
\end{aligned}
\end{align}
Then there exist two CPDP with the same infection and recovery rates, which
bound the CPERE from above and below.

\begin{proposition}\label{ComparisonCPEREandCPDP}
	Let $(\bfC,\bfB)$ be a CPERE with infection and recovery rate $\lambda,r$. Then there exist two CPDP $(\overline{\bfC},\overline{\bfB})$ and $(\underline{\bfC},\underline{\bfB})$ with the same infection and recovery rates $\lambda,r$ as $(\bfC,\bfB)$, for which  the rates of $\overline{\bfB}$ and $\underline{\bfB}$ are respectively $\alpha_{\max},\beta_{\min}$ and $\alpha_{\min}, \beta_{\max}$, and the property that  $(\underline{\bfC}_0,\underline{\bfB}_0)= (\bfC_0,\bfB_0) = (\overline{\bfC}_0,\overline{\bfB}_0)$ implies $\underline{\bfC}_t\subset \bfC_t\subset \overline{\bfC}_t$ and $\underline{\bfB}_t\subset \bfB_t\subset \overline{\bfB}_t$ for all $t>0$ a.s. 
\end{proposition}
As direct consequence of Proposition~\ref{ComparisonCPEREandCPDP} we obtain bounds on the survival probabilities. In the following corollary
$\theta$ denotes the survival probability of a general CPERE and $\theta_{\text{DP}}$ of the CPDP. 
\begin{corollary}\label{ComparisonSurvival}
	Assume that there exist 
	$\alpha_{\max}$,$\alpha_{\min}$, $\beta_{\max},\beta_{\min}\geq0$ as in \eqref{MaximalAndMinimalRates1}. Then
	\begin{align*}
	\theta_{\text{DP}}(\lambda,r,\alpha_{\max},\beta_{\min},C,B)\geq \theta(\lambda,r,C,B)\geq \theta_{\text{DP}}(\lambda,r,\alpha_{\min},\beta_{\max},C,B)
	\end{align*}
	where $C\subset V$ and $B\subset E$.
\end{corollary}
Furthermore, we are able to infer for a CPERE that complete convergence holds on a subset of its survival region. To be precise this subset will be the interior of the survival region of the CPDP $(\underline{\bfC},\underline{\bfB})$, which lies ``below'' the CPERE. 
Since this is a consequence of Theorem \ref{CPDPCompleteConvergence} we need the underlying graph to be $\Z^d$.
\begin{theorem}\label{CPDPCompleteConvergenceCPERE}
	Let $(\bfC,\bfB)$ be a CPERE on $\Z^d$
	and suppose that Assumption~\ref{AssumptionBackground} is satisfied
	as well as that there exist 	$\alpha_{\min}$, $\beta_{\max}$
	as in \eqref{MaximalAndMinimalRates1}. 
	If $\theta_{\text{DP}}(\lambda,r,\alpha_{\min},\beta_{\max},\{\zero\},\emptyset)>0$ then complete convergence
	as stated in \eqref{CompleteConvergence} holds.
\end{theorem}
\begin{remark}\label{ComparsisonRemark1}
    It is in fact possible to compare a CPERE on more general graphs $G$ with a CPDP on a $1$-dimensional integer lattice. For this we need to assume that $G$ is distance transitive, i.e.~for any two vertices $x,y\in V$ exists a graph automorphism $\sigma$ such that $\sigma(x)=y$. Note that distance transitivity is a stronger property than the transitivity which we assumed here since on transitive graphs only for neighboring vertices $x,y\in V$, i.e.~$\{x,y\}\in E$, such a graph automorphism exists.
    
    Then, for $(\bfC,\bfB)$ a CPERE on $G$ which satisfies Assumption~\ref{AssumptionBackground} complete convergence
	as stated in \eqref{CompleteConvergence} holds if $c_1(\lambda,\rho)>\kappa^{-1}\rho$ and   $\theta^{\Z}_{\text{DP}}(\lambda,r,\alpha_{\min},\beta_{\max},\{\zero\},\emptyset)>0$, where $\theta^{\Z}_{\text{DP}}$ denotes the survival probability of a CPDP on the $1$-dimensional integer lattice and $\alpha_{\min}$, $\beta_{\max}$ are defined as in \eqref{MaximalAndMinimalRates1}. This is further discussed in Remark \ref{rem:complconvextensionproof}.
\end{remark}
\subsection{Outline}
The rest of this paper is organized as follows. In Section~\ref{Sec:Discussion} we discuss our main results and put them into context with the current state of research and cite the relevant literature. Furthermore, we state open problems and  possible directions for future research.

In Section~\ref{Sec:GraphicalRepesentation} we present briefly a graphical representation for a broad class of interacting particle systems as found in \cite{swart2017course}. We use this representation to first construct an arbitrary finite range spin system, and thereafter we construct the CPERE. We end this section by deriving some basic properties such as for example monotonicity of the process with respect to the rates $\lambda$ and $r$.

In Section~\ref{Sec:Independece} we compare the asymptotic expansion speed of the set of all infections and the permanently coupled region and use this to prove our first main result Theorem~\ref{IndependencyCirticalRateSurvival}, which states that under a certain growth condition the critical infection rate $\lambda_c$ is independent of the initial configuration.

In the first subsection of Section~\ref{Sec:DualityAndInvariant} we derive a duality relation for the infection process $\bfC$. With this duality we study the connection between survival and non-triviality of the upper invariant law and prove Theorem~\ref{CrticalValuesAgree}. Furthermore, in the subsequent subsections we prove Proposition~\ref{ContinuityTheorem}, which states some continuity properties of the survival probability and we prove Theorem~\ref{CompleteConvergenceComplete}, which provides us with conditions which imply complete convergence of the CPERE.

In Section~\ref{Sec:CPDP} we focus on our main example the CPDP. In the first part we prove Proposition~\ref{ComparisonCPEREandCPDP} a comparison result between a general CPERE and a CPDP with adequately chosen rates. Furthermore, we show Proposition~\ref{IndependenceGrowthCond}. This result yields that even if the growth condition used in Section~\ref{Sec:Independece} does not hold we still get that the positivity of the survival probability is independent of the initial configuration, at least in the non-critical parameter regime.

In the second part we focus on the CPDP on the $d$-dimensional integer lattice and adapt techniques developed by Bezuidenhout and Grimmett~\cite{bezuidenhout1990critical}. This enables us to prove that the CPDP dies out at criticality (Theorem~\ref{NoSurvivalAtCriticality}) and as a consequence we are able to prove Proposition~\ref{JointContiuityOfCPDP}. We are also able to verify that complete convergence holds for the whole parameter regime of the CPDP (Theorem~\ref{CPDPCompleteConvergence}). We end this section by showing that complete convergence holds for the CPERE if a comparison to a suitable CPDP can be made,
namely we show Theorem~\ref{CPDPCompleteConvergenceCPERE}.
\section{Discussion}\label{Sec:Discussion}

In this paper we study a contact process in a fairly general ergodic evolving random environment. Our main example, the CPDP with $\bfB$ dynamical percolation, was first considered by Linker and Remenik \cite{linker2019contact}. They studied the survival probability and the existence of a phase transition. Among other things they could prove that in this model there exists a so called \emph{immunization phase}. This means that if we assume $r>0$ to be fixed, then for certain choices of $\alpha$ and $\beta$ there exists no infection rate $\lambda$ such that survival is possible. Furthermore, they studied the asymptotic behavior for survival as the update speed of an edge $\alpha+\beta$ tends to $0$ or $\infty$. Note that just recently \cite{hilario2021results} provide some further results in this direction for the CPDP on $\Z^d$. In all of these results the background is assumed to be stationary, i.e the initial distribution is the unique invariant law.

Thus, our first results concerning the influence of the initial configuration on the critical infection rate complement their results. We managed to show for general CPERE that if for a given $r>0$ we find a $\lambda>0$ with $\theta(\lambda,r,C,B)>0$ for some configuration $(C,B)$, which satisfies the inequality $c_1(\lambda,\rho)>\kappa^{-1}\rho$, then the survival of the infection process $\bfC$ is independent of the choice of the initial configuration, i.e.~the critical infection rate $\lambda_c(r)$ does not depend on $(C,B)$. We note that as this condition is always fulfilled if $\rho=0$ we have for subexponential growth graphs established independence of the initial condition in general.

With this we showed that in particular on subexponential growth graphs, as for example $\Z^d$, 
the results shown in \cite{linker2019contact} and \cite{hilario2021results} regarding the CPDP do not depend on their stationarity assumption of the background but hold in general.

For general CPERE it is not clear in the case $\rho>0$ if independence of the initial condition may still hold in general, or if there are background dynamics for which additional assumptions 
(such as our sufficient condition $c_1(\lambda,\rho)>\kappa^{-1}\rho$) are necessary:
\begin{problem}
	Let $x\in V$ be arbitrary but fixed and suppose $\rho>0$. Is the critical infection rate always independent of the initial conditions? In other words is $\lambda_c(r,\{x\},\emptyset)=\lambda_c(r,C,B)$ for all 
	$C\subset V$ finite and $B\subset E$? 
\end{problem}

For the CPDP we were able to show in Proposition~\ref{IndependenceGrowthCond} that even if the growth condition is not satisfied the interior of the survival region is still independent of the initial configuration. Thus, the only missing piece is the behaviour at criticality, i.e.~on the boundary of the survival region. Since the CP dies out at criticality on a large class of graphs, we would expect the same to be true for the CPDP. Hence, we come to the following conjecture.
\begin{conjecture}
	Let $(\bfC,\bfB)$ be the CPDP.
	Then $\lambda_c(r,C,B)=\lambda^{\pi}_c(r)$ for all $(C,B)$ with $C$ non-empty and finite.
\end{conjecture}

We were also able to show that the survival probability is continuous on the interior of the subset $\cS_{c_1}$ of the survival region defined in \eqref{SurvivalRegionWithGrowth}.
Furthermore, we conclude that the phase transition of survival with the background started stationary, i.e.~$\theta^{\pi}(\lambda,r,\{x\})=0$ to  $\theta^{\pi}(\lambda,r,\{x\})>0$, agrees with the phase transition of non-triviality of the upper invariant law, i.e.~$\overline{\nu}= \delta_{\emptyset}\otimes\pi$ to $\overline{\nu}\neq \delta_{\emptyset}\otimes\pi$. Thus, if additionally $c_1(\lambda_c^{\pi}(r),\rho)>\kappa^{-1}\rho$ holds the initial configuration of the background is of no importance to the question of non-triviality of $\overline{\nu}$.
This in itself is an interesting observation. But of course one would like to characterize all invariant laws for the CPERE which follows from showing complete convergence.

For the CP Durrett and Griffeath \cite{durrett1982contact} formulated equivalent conditions for complete convergence, which have been adapted to several variations of the CP, for example to 
the multi-type contact process by Remenik \cite{remenik2008contact}.  We managed to provide a similar characterization for the CPERE based on the previously stated results: Under the assumption that 
the growth condition $c_1(\lambda,\rho)>\kappa^{-1}\rho$ is satisfied we could show that Condition \eqref{ConvergenceCond1.1} and \eqref{ConvergenceCond2.1} imply complete convergence of the CPERE, i.e.
\begin{equation*}
    (\bfC_t^{C,B},\bfB_t^B)\Rightarrow \theta(C,B)\overline{\nu}+[1-\theta(C,B)](\delta_{\emptyset}\otimes\pi)
\end{equation*}
as $t\to \infty$. By using additionally monotonicity of the CPERE and continuity of the survival probability we are now able to determine a parameter subset 
 \begin{align*}
 \cS^{\text{cc}}_{c_1}:=\{ (\lambda,r): \exists \lambda'\leq \lambda \text{ s.t. } &(\lambda',r)\in\cS(\{x\},\emptyset)\\ 
 &\text{ and } c_1(\lambda',\rho)>\kappa^{-1}\rho,
 \eqref{ConvergenceCond1.1} \text{ and } \eqref{ConvergenceCond2.1} \text{ hold}\} \subset \cS_{c_1}
  \end{align*}
for which complete convergences holds. We then have that $(\lambda,r)\in \cS^{\text{cc}}_{c_1}$ if there exists a $\lambda'\leq \lambda$ such that $(\lambda',r)\in\cS(\{x\},\emptyset)$ and 
$c_1(\lambda',\rho)>\kappa^{-1}\rho$, as well as that  the conditions \eqref{ConvergenceCond1.1} and \eqref{ConvergenceCond2.1} hold. 

We illustrated the survival region $\cS(C,B)$ and the subsets $ \cS_{c_1}$ and $\cS^{\text{cc}}_{c_1}$ in Figure~\ref{fig:PhaseTransitionSurv}. On subexponential graphs, i.e.~$\rho=0$, the inequality $c_1(\lambda',\rho)>\kappa^{-1}\rho$  is trivially satisfied for all $\lambda>0$ since $c_1(\lambda,\rho)>0$ for all $\lambda>0$. Thus, $\cS_{c_1}=\cS(C,B)$ for all $(C,B)$ with $C$ non-empty and finite.
\begin{figure}[t]
	\centering 
	\includegraphics[width=120mm]{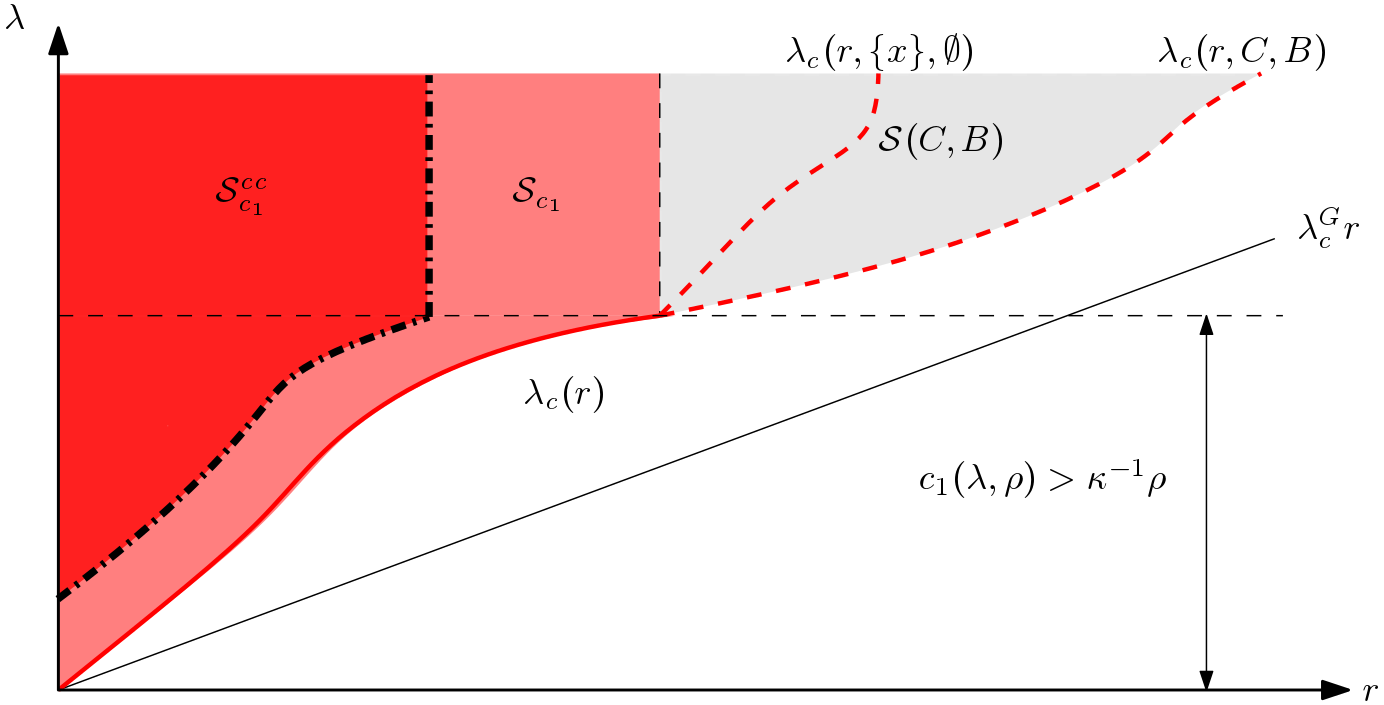}
	\caption{The solid and dashed red curve indicates the critical infection rate $\lambda_{c}(r,\cdot)$ of the CPERE. The solid black curve indicates the critical infection rate of a CP, where $\lambda_c^{G}$ is the critical infection rate for $r=1$.}
	\label{fig:PhaseTransitionSurv}
\end{figure}

Theorem~\ref{CompleteConvergenceComplete} only yields complete convergence of the CPERE if $(\lambda,r)\in \cS_{c_1}^{cc}$. Some crucial steps of our proof rely  on the assumption that the asymptotic expansion speed of the permanently coupled region is greater than that of the infection. But it is not clear that this assumption is necessary.

Let us for a moment consider the CP on a regular tree. In this setting it is possible to show that the CP has an intermediate phase, where complete convergence does not hold but the survival probability is positive. The reason for this occurrence is that in this phase global survival is possible, i.e.~the survival probability is positive, but if we consider an arbitrary vertex $x$ the probability that it gets reinfected infinitely often is $0$, and thus local survival is not possible. Thus, \eqref{ConvergenceCond1.1} and \eqref{ConvergenceCond2.1} fail. 
If local survival is not possible the CP started in a finite initial configuration will converge to $\delta_{\emptyset}$ as $t\to \infty$. But it is possible to construct several (in fact, infinitely many) non-trivial invariant laws, see \cite[Chapter I.4]{liggett2013stochastic} for details.

Now the question is if the background process can cause a similar phenomenon if the growth condition is not satisfied. Hence, the following question arises.
\begin{problem}
    What is the limiting behaviour of $(\bfC_t,\bfB_t)$ as $t\to \infty$ if for a pair of rates $(\lambda,r)\in(0,\infty)^2$ which satisfies Condition \eqref{ConvergenceCond1.1} and \eqref{ConvergenceCond2.1}, no $\lambda'\leq \lambda$ with $c_1(\lambda',\rho)\leq\kappa^{-1}\rho$ does exist which satisfies the \eqref{ConvergenceCond1.1} and \eqref{ConvergenceCond2.1}?
\end{problem}

In the special case that $G$ is the $d$-dimensional integer lattice so that the growth condition always holds, we showed for our main example, the CPDP, that complete convergence holds for all $(\lambda,r,\alpha,\beta)\in(0,\infty)^{4}$. Furthermore, for general CPERE whose  background process $\bfB$ satisfies Assumption~\ref{AssumptionBackground} complete convergence holds on the survival region of a CPDP with suitably chosen parameters, see Theorem~\ref{CPDPCompleteConvergenceCPERE}. Therefore, we might ask the following two questions.
\begin{problem}
	Does complete convergence hold for every $(\lambda,r)\in (0,\infty)^2$ for a CPERE on the $d$-dimensional integer lattice, if the background satisfies Assumption~\ref{AssumptionBackground}?
\end{problem}
\begin{problem}
	Does the CPERE on the $d$-dimensional integer lattice go extinct a.s. at criticality, if the background satisfies Assumption~\ref{AssumptionBackground}?
\end{problem}
Closely related to complete convergence is the asymptotic shape theorem. Now let $\tau:=\inf\{t\geq 0:\bfC^{\{\zero\},\emptyset}_t\neq \emptyset\}$ denote the extinction time of the infection process $\bfC$ with initial configuration $(\{\zero\},\emptyset)$.
Let $\bfH_t:=\bigcup_{s\leq t}\bfC^{\{\zero\},\emptyset}_s$ be the set of all vertices which were infected at least once until time $t$ and $\bfK_t:=\{x\in V: x\in \bfC^{\{\zero\},\emptyset}_s\symdiff\bfC_s^{V,E} \,\,\forall s\geq t\}$
be the permanently coupled region of the infection process $\bfC$. Furthermore, we set $\bfH_t':=\bfH_t+\big[\small{-}\tfrac{1}{2},\tfrac{1}{2}\big]^d$ and $\bfK'_t:=\bfK_t+\big[\small{-}\tfrac{1}{2},\tfrac{1}{2}\big]^d$.
\begin{conjecture}\label{Conjecture1}
	Let $(\bfC,\bfB)$ be a CPERE with infection rate $\lambda>0$ and recovery rate $r>0$, where $\bfB$ satisfies Assumption~\ref{AssumptionBackground}. Suppose that $\theta(\lambda,r,\{\zero\},\emptyset)>0$ and there exist constants $C_1,C_2,M>0$ such that
	\begin{align}
	\Pw(t\leq \tau<\infty)&\leq C_1 \exp(-C_2 t),\label{ConjectureEq1}\\
	\Pw(x\notin \bfH_{M||x||_1+t},\tau=\infty)&\leq C_1 \exp(-C_2 t),\label{ConjectureEq2}\\
	\Pw(x\notin \bfK_{M||x||_1+t},\tau=\infty)&\leq C_1 \exp(-C_2 t).\label{ConjectureEq3}
	\end{align}
	Then there exists a bounded and convex subset $U\subset \R^d$ such that for every $\varepsilon>0$
	\begin{align*}
	\Pw\big(\exists s\geq0 :  t(1-\varepsilon)U\subset (\bfK'_t\cap 		\bfH_t')\subset \bfH_t'\subset t(1+\varepsilon)U \,\,\forall t\geq s\big| \tau=\infty\big)=1.
	\end{align*}
\end{conjecture}
Let us briefly explain the three conditions mentioned in this conjecture. Condition \eqref{ConjectureEq1} implies that if the infection process $\bfC$ goes extinct, then this will happen most likely early on. Condition \eqref{ConjectureEq2} basically states that if $\bfC$ survives, the infection expands asymptotically at least according to some linear speed with high probability. Condition \eqref{ConjectureEq3} has a similar interpretation for the permanently coupled region.

Garet and Marchand \cite{garet2012asymptotic} proved an asymptotic shape theorem for the contact process on $\Z^d$ in a static random environments. Deshayes adapted their techniques in \cite{deshayes2014contact} to a dynamical setting and showed an asymptotic shape theorem for a contact process with ageing. Furthermore, in \cite{deshayes2015asymptotic} it was explained that this can also be extended to a broader class of time dynamical contact processes, which includes among others the contact process with varying recovery rates studied by \cite{broman2007stochastic} and \cite{steif2007critical}. Since the latter model shares a lot of similarities with the CPERE constructed here we believe that Conjecture~\ref{Conjecture1} should hold true.

Both works \cite{garet2012asymptotic} and \cite{deshayes2014contact} have proven similar conditions to \eqref{ConjectureEq1}, \eqref{ConjectureEq2} and \eqref{ConjectureEq3}, for the contact process in a static random environment and respectively for the contact process with ageing, by an adaption of the techniques developed in \cite{bezuidenhout1990critical}. Since we also adapt these techniques for the CPDP, we believe the following conjecture to be true.
\begin{conjecture}\label{Conjecture2}
	Let $(\bfC,\bfB)$ be a CPDP with rates $\lambda,r,\alpha,\beta>0$ on the $d$-dimensional integer lattice. Suppose $\theta_{\text{DP}}(\lambda,r,\alpha,\beta)>0$, then there exists $C_1,C_2,M>0$ such that \eqref{ConjectureEq1}, \eqref{ConjectureEq2} and \eqref{ConjectureEq3} are fulfilled.
\end{conjecture}
In this paper we tried to minimize the assumptions we pose on the underlying graph $G$ such that it is as general as possible. But if we restrict the class of graphs we could still relax some assumption further. 

For example, if we only consider subexponential growth graphs we could relax Assumption~\ref{AssumptionBackground} $(ii)$. Here, it would suffice to assume that $\Pw(e\in \Psi'_{n})$ decays polynomially with a high enough exponent in order to show that the critical infection rate is independent of the initial configuration. This can be seen in the proof of Proposition~\ref{ExpansionSpeedPerCouRegion} where we compare the asymptotic expansion speed of the infected area and of the permanently coupled region.

One may still be able to prove continuity of the survival probability without assuming the reversibility of Assumption~\ref{AssumptionBackground} $(iii)$. 
There are interesting background dynamics
that are not reversible, for example the noisy voter model
on more general graphs than the $1$-dimensional integer lattice.
Depending on the underlying graph one may be able to use results from ergodic theory instead of duality combined with reversibility, for example if the graph automorphism group of $G$ is finitely generated.

Furthermore, one could relax Assumption~\ref{AssumptionBackground} $(i)$ and consider non-ergodic systems as for example the ferromagnetic Ising model on $\Z^d$ for $d\geq 2$ and for a large enough inverse temperature $\beta$. In this case the setting would be fundamentally different since already the invariant distribution of the background process $\bfB$ would depend, by assumption, on the initial configuration.

Apart from these concrete questions there are a number of further research directions regarding the CPERE.
Here, we mainly considered the behaviour of the CPERE in the supercritical regime, i.e.~the parameter regime where the survival probability is positive. For the CP it is known that in the subcritical regime the infection dies out exponentially fast, i.e.~$\Pw(\bfC_t\neq \emptyset)\leq e^{-C(\lambda)t}$, where $C(\lambda)$ is a constant only depending on the infection rate $\lambda$ see \cite{bezuidenhout1991exponential}. 
Linker and Remenik \cite{linker2019contact} showed on $\Z$ that an immunization region exists where the infection cannot survive no matter how large we choose the infection parameter. One heuristic explanation for this phenomenon is that if the update speed is slow enough then the background isolates the infection long enough in a bounded area such that the infection dies out. One could roughly say that the background dynamics dominates the survival behaviour. This may also have consequences for the subcritical behaviour of CPERE which has not been studied yet in any detail.

Also, in this paper we assumed that the background process evolves autonomously, i.e.~we do not allow feedback from the infection process to the background. Depending on the application of the model this may not be a realistic assumption. For example, in the current %situation of the 
Covid-19 Pandemic testing, contact tracing, and isolation through quarantine are important  measures for curbing the spread of the infection. This could be modeled through closing edges adjacent to infected individuals at a higher rate.
There is great interest in the effectiveness of these measures, and this has been studied recently via simulations, see for example by Aleta et al.~\cite{aleta2020modelling} and Kucharski et al.~\cite{kucharski2020effectiveness}. 
Of course, allowing a feedback from the infection process to the background would lead to 
a dependency structure that is far more complex, and we thus leave the analysis of these models for future research.

\section{Graphical representation and basic facts}\label{Sec:GraphicalRepesentation}
The construction of the CPERE via a graphical representation
is essential for our proofs. Here, we
follow the construction via random maps of Swart \cite{swart2017course}. 

For  an interacting particle system on a space $\Lambda$ with state space $\cP(\Lambda)$ the changes are described by maps $m:\cP(\Lambda) \to \cP(\Lambda)$ whose actions can then also be interpreted graphically. Let $\cM$ be the set which contains all relevant maps and let $(\gamma_m)_{m\in \cM}$ be non-negative constants. 
Let $\Xi$ be a Poisson point set on $(\cM\times \R)$ 
with intensity measure $\xi$ characterized via $\xi(\{m\}\times (s,t])=\gamma_m (t-s)$ for $m \in \cM$ and $0\leq s<t$. 
If $\Lambda$ is finite and hence the set $\cM$ is finite then for every $t\geq0$ we can enumerate the points in $\Xi_{0,t}:=\Xi\cap (\cM\times (0,t])$ consecutively, i.e. 
\begin{equation*}
\Xi_{0,t}=\{(m_{1},t_{1}),\dots,(m_{n},t_{n})\}
\end{equation*}
with $t_1<\dots <t_n$ a.s., and define $\bfX_t(A):=m_{n}\circ\dots\circ m_{1}(A)$,
where $A\in \cP(\Lambda)$. By \cite[Proposition 2.5]{swart2017course} this is a Feller process with the generator
\begin{align}
\label{randommappinggenerator}
\cA f(A)=\sum_{m\in \cM} \gamma_{m}(f(m(A))-f(A))
\end{align}
where $f\in C(\cP(\Lambda))$ and $A\in \cP(\Lambda).$ 
 If $\cM$ is a countable set of local maps more care has to  be taken so that the process $\bfX$ is well defined. Let
\begin{align*}
\cD(m):=\{ x\in \Lambda : \exists A\in \cP(\Lambda) \text{ s.t. } x\in m(A)\symdiff A\}.
\end{align*}
This set is the collection of all $x\in \Lambda$ which can possibly be changed by $m$. Next, for a given $x\in \Lambda$ call $y\in \Lambda$ \emph{$m$-relevant} if there exist $ A,B\in \cP(\Lambda)$ such that $x\in m(A)\symdiff m(B)$ and $ A\symdiff B=\{y\}$.
Let
\begin{align*}
\cR_x(m):=\{ y\in \Lambda : y \text{ is } m\text{-relevant w.r.t } x \}.
\end{align*}
  Now, \cite[Theorem~4.14]{swart2017course} and \cite[Theorem~4.17]{swart2017course} state that if
 \begin{align}\label{totalrb}
\sup_{x\in \Lambda}\sum_{m\in \cM,\cD(m)\ni x} \gamma_m(|\cR_x(m)|+1)<\infty
\end{align}
 then $\bfX$ is a well defined c\`adl\`ag Feller process with values in $\cP(\Lambda)$ and generator $\cA$, i.e.~$\bfX$ has paths in $D_{[0,\infty)}(\cP(\Lambda))$, the space of all c\`adl\`ag functions mapping $[0,\infty)$ to $\cP(\Lambda)$.

\subsection{Graphical representation of finite range spin systems}\label{GraphRepSpinSystemCose}
We consider spin systems as our background processes
whose generators are generally written in the form
\begin{equation}\label{GeneratorSpinSystem}
\cA_{\text{Spin}} f(B)=\sum_{e\in E} q(e,B)(f(B\symdiff\{e\})-f(B)),
\end{equation}
where $f\in C(\cP(E))$ and $B\subset E$.
The interpretation of a spin system is that at an edge $e\in E$ a spin flip takes places with a spin rate $q(e,B)$ which depends on the current configuration $B$. 
We can rewrite this generator in the form (\ref{randommappinggenerator}).
Since we only consider finite range systems we know that there exists a range $R\in \N$ such that $q(e,B)=q(e,B\cap\B^L_R(e))$ for any $B\subset E$. 
We distinguish between an up or down flip, i.e. between having $e\notin B$ or $e\in B$ before the flip. Then we consider every possible configuration of the $R$-neighbourhood of $e$, which we denote by $\cN^L_e(R):=\B_R^L(e)\backslash\{e\}$. 
For every $e\in E $ and $F\subset \cN^L_e(R)$ we set
\begin{align*}
\mathbf{up}_{e,F}(B)&:=
\begin{cases}
B\cup \{e\}& \text{ if } e\notin B \text{ and } B\cap \cN^L_e(R)=F\\
B& \text{ otherwise},
\end{cases}\\
\mathbf{down}_{e,F}(B)&:=
\begin{cases}
B\backslash \{e\}& \text{ if } e\in B \text{ and } B\cap \cN^L_e(R)=F\\
B& \text{ otherwise},
\end{cases}
\end{align*}
for $B\subset E$ and choose the rates to be
\begin{align}\label{SpinSystemRates}
\gamma_{\mathbf{up}_{e,F}}=q(e,F) \quad \text{and} \quad \gamma_{\mathbf{down}_{e,F}}=q(e,F\cup \{e\}). 
\end{align}
Note that $e\notin F$, since $F\subset \cN^L_{e}(R)$ and that for every $e$ there are finitely many $F\subset \cN^L_{e}(R)$ and therefore finitely many relevant maps. We denote the sets of the two types of maps by 
\begin{align*}
\cM_{\text{up}}=\{\mathbf{up}_{e,F}: e\in E, F\subset \cN^L_e(R)\}\,\, \text{and} \,\,
\cM_{\text{down}}=\{\mathbf{down}_{e,F}:e\in E, F\subset \cN^L_e(R)\}
\end{align*}
and define the set of all maps as $\cM_{\text{back}}:=\cM_{\text{up}}\cup \cM_{\text{down}}$. Now let $\Xi^{q}$ be the corresponding  Poisson point set on $\cM\times \R$.
%with rates $(\gamma_m)_{m\in \cM}$. 
Obviously \eqref{totalrb} is satisfied, and thus 
the associated process $\bfX$ from the aforementioned random mapping representation is a well defined Feller process  with the generator as in (\ref{randommappinggenerator}).
Also, by plugging in the maps and rates it can be easily verified  that this generator is the same as the generator stated in \eqref{GeneratorSpinSystem}.
Graphically, we represent the state of the edge by a grey area while it is closed (the last map applied was a down flip) and by a white area while it is open (the last map applied was an up flip), see Figure~\ref{fig:CPGraphRepVis:b}.

\subsection{Construction of the CPERE via graphical representation}\label{GraphRepresentationCPERE}
In this subsection we explicitly construct the CPERE.
We assume that the maps and rates used to construct the (autonomous) background $\bfB$ via the random mapping representation are known, i.e.~$\cM_{\text{back}}$ is a countable set which contains local maps $m:\cP(E)\to \cP(E)$ with corresponding rates $(\gamma_m)_{m\in \cM_{\text{back}}}$
as in Section~\ref{GraphRepSpinSystemCose}.
In order to be in the setting of the random mapping representation we consider $\Lambda=V\cup E$ and thus the state space is $\cP(V\cup E)$ which can be identified with $\cP(V)\times \cP(E)$ 
via a one-to-one correspondence.

First we extend the maps $m\in \cM_{\text{back}}$ to maps $m^{*}:\cP(V\cup E)\to \cP(V\cup E)$. For every set $A\subset V\cup E$ there exist $C\subset V$ and $B\subset E$ such that $A=C\cup B$. Then we set $m^{*}(A):=C\cup m(B)$ for every $m\in \cM_{\text{back}}$. Let $\cM^{*}_{\text{back}}$ denote the set of all maps $m^{*}$ and set $\gamma_{m^{*}}=\gamma_m$. Next for $A\subset V\cup E$ and $x,y\in V$ such that $\{x,y\}\in E$ we define
\begin{align*}
\mathbf{coop}_{x,y}(A)&:=
\begin{cases}
A\cup \{y\}& x\in A \text{ and } \{x,y\}\in A,\\
A& \text{otherwise},
\end{cases}\\
\mathbf{rec}_{x}(A)&:=A\backslash \{x\}.
\end{align*}
and set the rates to be $\gamma_{\textbf{coop}_{x,y}}=\lambda$ and $\gamma_{\mathbf{rec}_{x}}=r$. The map $\textbf{coop}_{x,y}$ is called the cooperative infection map. The name comes from the fact that for $x$ to successfully infect $y$ it needs the edge $\{x,y\}$ to be open. In this sense $x$ and $\{x,y\}$ must cooperate such that the infection spreads to $y$. Now set
\begin{equation*}
\cM_{\text{CP}}:=\underbrace{\{\mathbf{coop}_{x,y}:x,y\in V \text{ s.t. }\{x,y\}\in E\}}_{=\cM_{\text{inf}}}\cup \underbrace{\{\mathbf{rec}_{x}:x\in V\}}_{=\cM_{\text{rec}}}
\end{equation*}
Let us denote by $\Xi=\Xi_{\lambda,r}$ the Poisson point set with respect to  $\cM:=\cM_{\text{CP}}\cup \cM^{*}_{\text{back}}$ and by $(\gamma_m)_{m\in \cM}$ the corresponding rates. Obviously, \eqref{totalrb} is satisfied, and thus there exists a c\`adl\`ag Feller process $\bfX$ on $\cP(V\cup E)$ with generator
\begin{align*}
\cA f(A)=&\sum_{m\in \cM}\gamma_{m}(f(m(A))-f(A))\\
=&\sum_{x\in V}\lambda\sum_{y\in V:\{x,y\},x\in A }(f(A\cup \{y\})-f(A))+\sum_{x\in V}r(f(A\backslash \{x\})-f(A))\\
+&\sum_{m\in \cM_{\text{back}}}\gamma_{m}(f((A\backslash E )\cup m(A\backslash V))-f(A)),
\end{align*}
where $A\subset V\cup E$. Graphically, we draw an infection arrow from $x$ to $y$ when the map $\mathbf{coop}_{x,y}$ is applied which also draws input from the state of the edge $\{x,y\}$ (as indicated by the grey/white shading of the corresponding area). The recovery events induced by the application of the map $\mathbf{rec}_{x}$ are indicated through crosses. See Figure~\ref{fig:GraphRep}(a).
While the state of the edges in $E$ is specified by the shading the state of vertices in $V$ can be deduced by defining infection paths: These follow vertices in the direction of time (upwards in Figure~\ref{fig:GraphRep}(b)). They are stopped by crosses but may also follow infection arrows if the corresponding edge is open. We then have that 
$$C_t^C=\{v \in V: \exists \text{ an infection path  connecting some $v' \in C$ at time $0$ with $v$ at time $t$}\}$$
is a CPERE with the corresponding background $B$, see Figure~\ref{fig:GraphRep}(b) where the infection paths are indicated in red and the open edges are additionally highlighted in purple. 
Here, the background process $\bfB$ is given by a dynamical percolation 
(see Example~\ref{ThreeParticularSpinSystems} $(i)$). Note that in the case of the dynamical percolation there is a much simpler choice of maps and rates since it can be constructed via the maps
\begin{equation*}
\mathbf{birth}_{e}(B):=B\cup \{e\} \quad \text{and} \quad	\mathbf{death}_{e}(B):= B\backslash \{e\}
\end{equation*}
for $B\subset E$ and rates $\gamma_{\mathbf{birth}_{e}}=\alpha$ and $\gamma_{\mathbf{death}_{e}}=\beta$ for all $e\in E$.

The process $\bfX$ is a combination of infection process and the background in one. But, it is far more convenient to treat these two parts as separate object. Therefore, we switch back to the state space $\cP(V)\times \cP(E)$, which we achieve by setting $\bfC_t:=\bfX_t\backslash E$ and $\bfB_t:=\bfX_t\backslash V$ for all $t\geq 0$. We obtain the CPERE as described in the beginning, i.e.~$(\bfC,\bfB)$ is a Feller process on the state space $\cP(V)\times \cP(E)$ and $\bfC$ has jump rates \eqref{InfectionRatesWithBackground}.

\begin{figure}[t]
	\centering 
	\subfigure[The two tailed arrows illustrate the $\mathbf{coop}$ maps and the crosses on the vertices the $\mathbf{rec}$ maps. 
	The circles and crosses on the edges illustrate the $\mathbf{birth}$ and $\mathbf{death}$ maps.
	 The grey area indicates that an edge is closed. ]{\label{fig:GraphRep:a}\includegraphics[width=68.5mm]{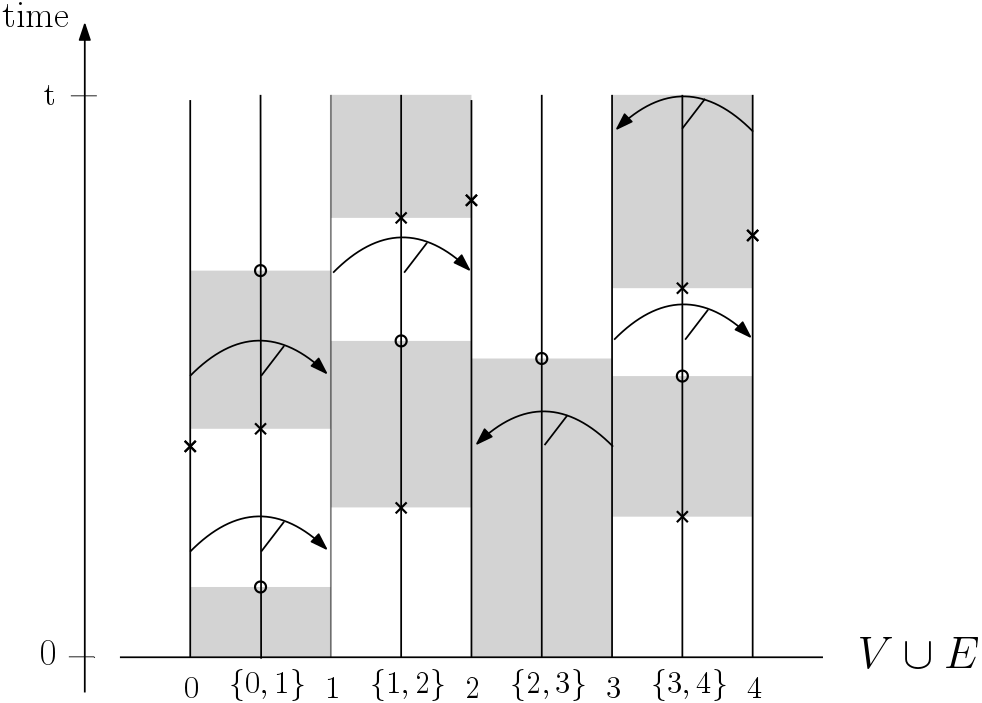}}\hfill
	\subfigure[The vertical red lines indicate 
	when a vertex is infected and the purple lines when an edge is open. An infection path consists of red vertical lines and red arrows, which leads from a vertex at time $0$ to a vertex at time $t$. For a successful transmission of an infection also the corresponding edge must be open.]{\label{fig:GraphRep:b}\includegraphics[width=68.5mm]{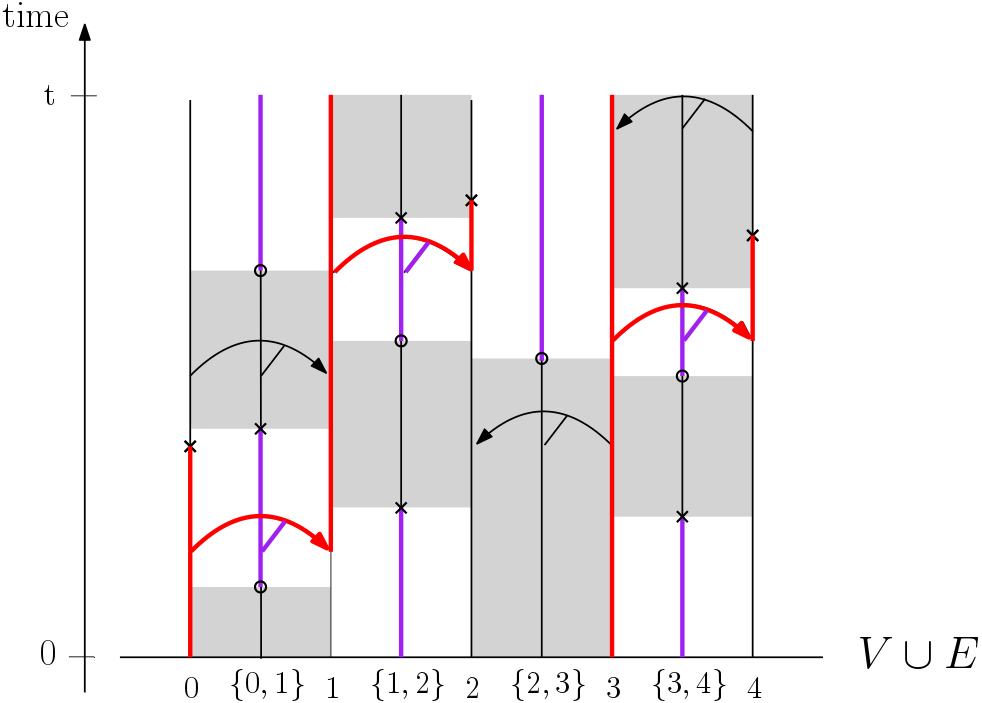}}
	\caption{Illustration of the graphical representation of a CPDP as an example. 
    Here the state space is $V\cup E$ with $V=\Z$.
    }
	\label{fig:GraphRep}
\end{figure}

\begin{remark}\label{SplitUpPointProcesses}
	The Poisson point set $\Xi$ used in the graphical representation can be represented via  three independent Poisson point set. These are $\Xi^{\inf}$ on $\cM_{\inf}\times \R$, which are in the graphical representation (see Figure~\ref{fig:GraphRep}), the infection arrows, $\Xi^{\text{rec}}$ on $\cM_{\text{rec}}\times \R$ corresponding to the recovery symbols and $\Xi^{\text{back}}$ on $\cM^{*}_{\text{back}}\times \R$ which are the maps used to construct the background process. 
\end{remark}

\subsection{Basic properties of the CPERE}\label{BasicProperties}
In this subsection we summarize some basic properties of the CPERE which are already known for the CP and which can be shown analogously through couplings of the involved processes via the graphical representation. Let  $\Pw_{\lambda,r}$ be the law of the Feller process $(\bfC,\bfB)$ constructed via the graphical representation using the Poisson point set $\Xi=\Xi_{\lambda,r}$. 
By construction it is clear that $(\bfC,\bfB)$ is a strong Markov process with respect to the filtration $(\cF_t)_{t\geq 0}$  generated by $(\Xi_{0,t})_{t\geq 0}$.

We equip $\cP(V)$ and $\cP(E)$ with the inclusion $\subset$ as a partial order and for elements in $\cP(V)\times \cP(E)$ we write $(C,B)\subset(C',B')$ if $C\subset C'$ and $B\subset B'$, which also defines a partial order. We denote by $(T(t))_{t\geq 0}=(T_{\lambda,r}(t))_{t\geq 0}$ the corresponding Feller semigroup of the CPERE. If $\mu$ is the initial distribution of the CPERE we use  $\mu T(t)$ as a short notation for
\begin{equation*}
    \mu T(t)f=\int T(t)f(x)\mu(\mathsf{d} x).
\end{equation*}
Furthermore, we denote by ``$\preceq$'' the stochastic order. 

\begin{lemma}[Monotonicity and additivity]
\label{MonotonicityAdditivityLemma}
Let $(\bfC,\bfB)$ be a CPERE with parameters $\lambda,r$ constructed using  $\Xi_{\lambda,r}.$
\begin{itemize}
    \item[(i)] (Monotoniciy with respect to the initial condition) If $C_1 \subset C_2$ and $B_1 \subset B_2$ then for all $t \geq 0$ also  $\bfC_t^{C_1,B_1}\subset \bfC_t^{C_2,B_2}$ and 
    $\bfB_t^{B_1}\subset \bfB_t^{B_2}$.
    This also implies that $(\bfC,\bfB)$ is a \emph{monotone} Feller process, i.e.~let $\mu_1$ and $\mu_2$ be probability measures on $\cP(V)\times \cP(E)$ with $\mu_1\preceq\mu_2$. Then this implies $\mu_1T_{\lambda,r}(t)\preceq \mu_2T_{\lambda,r}(t)$ for all $t\geq 0$.
    \item[(ii)] (Monotoniciy with respect to $\lambda,r$) 
    Let $\widehat{\lambda}\geq \lambda$. Then there exists a CPERE $(\widehat{\bfC},\bfB)$ with infection rate $\widehat{\lambda}$, the same initial configuration and  recovery rate $r$ such that $\bfC_t\subseteq  \widehat{\bfC}_t$ for all $t\geq0$. In words $\bfC$ is monotone increasing in $\lambda$. On the other hand $\bfC$ is monotone decreasing in $r$.
    \item[(iii)] (Additivity) Let $t\geq 0$, then $\bfC_t^{C,B}\cup \bfC_t^{C',B}=\bfC_t^{C\cup C',B}$ 
	for all $B\subseteq E$ and $C,C'\subseteq V$.
\end{itemize}
\end{lemma}
\begin{proof} All of these statements follow directly from the graphical representation. In (i) we use that $\bfB$ is an attractive spin system which implies $\bfB_t^{B_1}\subset \bfB_t^{B_2}$ and subsequently, as for the regular CP, we obtain $\bfC_t^{C_1,B_1}\subset \bfC_t^{C_2,B_2}$. For (ii) we define $(\widehat{\bfC},\bfB)$ via  $\Xi_{\lambda,r}$ and add additional independent infection arrows with rate $(\widehat{\lambda}-\lambda)$. For (iii) the background processes are identical and so this follows from the graphical representation, just as for the classical CP.
\end{proof}

Furthermore the probabilities of events where $\bfC$ depends only on a finite time horizon are continuous with respect to the infection and recovery rates if we consider finitely many  vertices are initially infected.
\begin{lemma}[Continuity for finite times and finite initial infections]\label{ContinuityLemma}
	Let $(\bfC,\bfB)$ be a CPERE with initial configuration $C\subset V$ and $B\subset E$ such that $|C|<\infty$. Furthermore, let $t\geq 0$ and  $\cA\subset D_{\cP(V)}([0,t])$. Then
	\begin{align*}
	\lambda\mapsto \Pw^{(C,B)}_{\lambda,r}((\bfC_s)_{s\leq t}\in\cA) \quad \text{and} \quad r\mapsto \Pw^{(C,B)}_{\lambda,r}((\bfC_s)_{s\leq t}\in\cA)
	\end{align*}
	are continuous.
\end{lemma}
\begin{proof}
	This can be proven analogously as for the classical CP, see \cite[Chapter~I.1]{liggett2013stochastic}. 
	To show continuity in $\lambda$ one constructs as in the proof of Lemma~\ref{MonotonicityAdditivityLemma} a CPERE $(\widehat{\bfC},\bfB)$ with infection rate $\widehat{\lambda}>\lambda$ and the same recovery rate and initial configuration as $(\bfC,\bfB)$. This process $(\widehat{\bfC},\bfB)$ is then coupled via the graphical representation to $(\bfC,\bfB)$ such that $\bfC_t\subset \widehat{\bfC}_t$ for all $t> 0$ and $\bfC_0= \widehat{\bfC}_0$. It suffices to show that as $|\widehat{\lambda}-\lambda|\to 0$ then
	\begin{equation*}
	    \Pw(\bfC_s\neq \widehat{\bfC}_s \text{ for some } s\leq t)\to 0,
	\end{equation*}
	which follows because up to time $t$ a.s.~only finitely many vertices are infected.
	The statement for $r$ follows again analogously.
\end{proof}
\section{Influence of the initial state of the background on survival}\label{Sec:Independece}
In this section we prove Theorem~\ref{IndependencyCirticalRateSurvival}, i.e.~that the critical infection rate for survival does not depend on the initial configuration of CPERE if a certain growth condition is satisfied. We assume throughout the whole section that the background $\bfB$ satisfies Assumption~\ref{AssumptionBackground} $(i)$ and $(ii)$.

At first we need to study the asymptotic expansion speed of the infection process $\bfC$ and the permanently coupled region $\Psi'$ or to be precise the asymptotic expansion speed of a connected component of an arbitrary edge $e$ separately. Then we can compare these two objects in terms of expansion speed. 

The maximal number of infected vertices can be represented by a classical contact process $\widetilde{\bfC}^C=(\widetilde{\bfC}^C_t)_{t\geq0}$ with infection rate $\lambda>0$, recovery rate $r=0$ and $\widetilde{\bfC}^{C}_0=C\subset V$, which is coupled with the CPERE $(\bfC^{C,B},\bfB^{B})$ such that $\bfC^{C,B}_t\subset\widetilde{\bfC}^{C}_t$ for all $t\geq 0$ for any $B\subset E$. 
Simply put, we ignore the background $\bfB$, and thus consider every infection arrow to be valid regardless of the state of the edge at the time of the transmission and ignore every recovery event.

In Figure~\ref{fig:Growth:a} we illustrated the spread of the maximal infection, i.e.~$t\mapsto\widetilde{\bfC}^{\{0\}}_t$, on $\Z$ and in Figure~\ref{fig:Growth:b} the expansion speed of the connected component of $\Psi_t'$ which contains $\{0,1\}$, where we consider $\bfB$ to be a dynamical percolation. The comparison suggest that $\Psi'$ expands much faster than $\widetilde{\bfC}^{\{0\}}$.
\begin{figure}[t]
	\centering 
	\subfigure[\footnotesize The black lines represent the right and left most particle of $\widetilde{\bfC}^{\{0\}}$ on $\Z$ with $\lambda=2$. The red line has a slope of $\frac{1}{2}$ and the blue line $\frac{1}{3}$.] {\label{fig:Growth:a}\includegraphics[width=68.5mm]{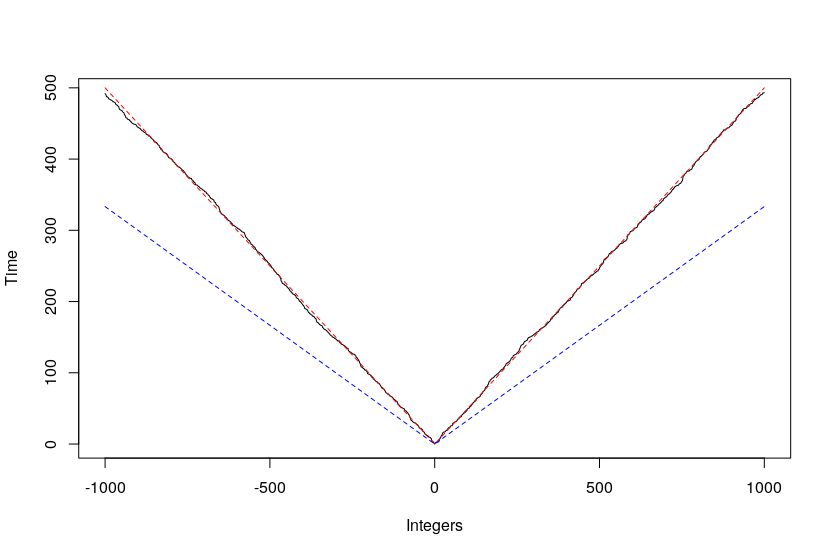}}
	\hfill
	\subfigure[\footnotesize A simulation of the first update times of a dynamical percolation on $\Z$ with speed $v=\alpha+\beta=2$. The black bars are the waiting times until the first update. The red dashed line is the right and left most edge of the connected component of $\Psi'$ containing the edge $\{0,1\}$.]{\label{fig:Growth:b}\includegraphics[width=68.5mm]{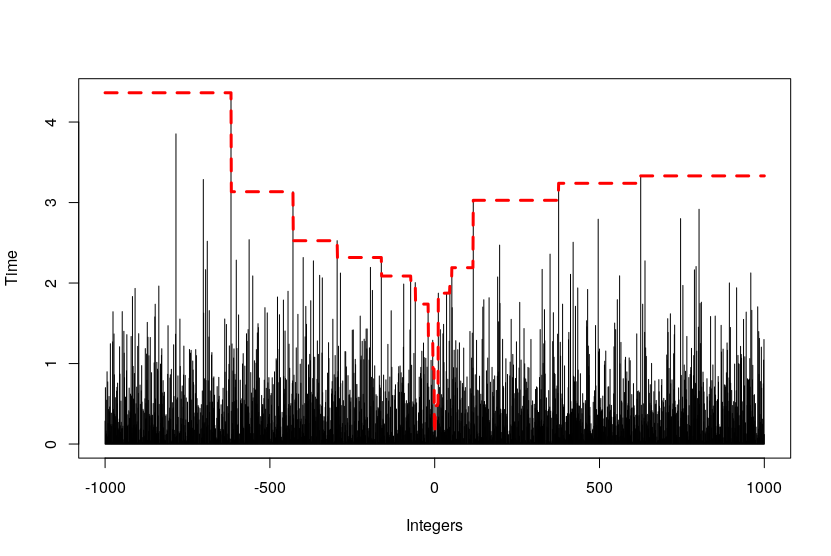}}
	\caption{Simulations of the infected area $\widetilde{\bfC}^{\{0\}}$ on the left and the first update times of a dynamical percolation and thus $\Psi'$ on the right on the integer lattice $\Z$.}
	\label{fig:Growth}
\end{figure}

We start with the asymptotic expansion speed of the set of all infections, i.e.~the process $\widetilde{\bfC}^C$. It is well known that asymptotically the infected area can grow at most at some linear speed in time. This is also illustrated in Figure~\ref{fig:Growth:a}. Next we provide an explicit upper bound for this linear speed. To be precise, for a given infection parameter $\lambda>0$ this upper bound will be $(c_1(\lambda,\rho))^{-1}$, where $c_1(\lambda,\rho)$ is the unique solution of
\begin{align}\label{GrowthConstantInfection}
c\lambda-1-\log(c\lambda |\cN_x|)=\rho
\end{align}
such that $0<c_1(\lambda,\rho)<\lambda^{-1}$.
\begin{lemma}\label{UniqueSolution}
	Let $\lambda>0$ and $x\in V$. There exists a unique solution $0<c_1(\lambda,\rho)<\lambda^{-1}$ of \eqref{GrowthConstantInfection}. Furthermore, $\lambda\mapsto c_1(\lambda,\rho)$ is continuous, strictly decreasing, $c_1(\lambda,\rho)\to \infty$ as $\lambda\to 0$ and $c_1(\lambda,\rho)\to 0$ as $\lambda\to \infty$.
\end{lemma}
\begin{proof}
	The solution $c_1(\lambda,\rho)$ can actually be stated explicitly with the help of the Lambert $W$-function, i.e. the inverse function of $t\mapsto te^t$.
	As domain of the function we consider $(-e^{-1},\infty)$ such that $W:(-e^{-1},\infty)\to (-1,\infty)$. 
	
	Now one can show that  $c_1(\lambda,\rho)=-\frac{1}{\lambda}W\big(-|\cN_x|^{-1}e^{-(1+\rho)}\big)$, which can be verified by inserting our guess into \eqref{GrowthConstantInfection}. First wee see that $c\lambda-1-\log(c\lambda |\cN_x|)=\rho$ if and only if $|\cN_x|^{-1}\exp(-(1+\rho))= c\lambda \exp(-c\lambda)$. Therefore, inserting our guess in the right-hand side and using that $W$ is the inverse function of $t\mapsto te^t$ verifies that this is a solution of \eqref{GrowthConstantInfection}.
	
	The function $W$ is continuous and strictly increasing. Furthermore, $W(s)\to -1$ as $s\to -e^{-1}$, $W(s)\to \infty$ as $s\to \infty$ and $W(0)=0$. We see that 
	\begin{equation*}
	-e^{-1}<|\cN_x|^{-1}\exp(-(1+\rho))<0
	\end{equation*}
	and $-1<W(s)<0$ for $-e^{-1}<s<0$, and thus it follows that $c_1(\lambda,\rho)<\frac{1}{\lambda}$. For $\lambda>0$ fixed we set $g_{\rho}(c):=c\lambda-1-\log(c\lambda |\cN_x|))-\rho$ for all $c>0$. Obviously the function $g_{\rho}$ is smooth on $(0,\infty)$ and its derivative is $g'_\rho(c)=\lambda-\frac{1}{c}>0$ for all $c<\frac{1}{\lambda}$ which implies that $g_{\rho}$ is strictly decreasing on $(0,\frac{1}{\lambda})$, and thus $c_1(\lambda, \rho)$ must be the unique solution of \eqref{GrowthConstantInfection} on $(0,\frac{1}{\lambda})$. At last the two properties follow immediately by the properties of the Lambert $W$-function.
\end{proof}

Next we define the first hitting time of $y\in V$ for $\widetilde{\bfC}$ with initial infections $C\subset V$ as $\tau_y(C):=\inf\{t\geq 0: y\in\widetilde{\bfC}^{C}_t\}$.
\begin{lemma}\label{LinearGrowth}
	Let $\lambda>0$ and set $g_{\rho}(c):=c\lambda-1-\log(c\lambda |\cN_x|))-\rho$ for all $c>0$. Then for every $0<c< c_1(\lambda,0)$ we have $g_0(c)>0$ and
	\begin{align*}
	\Pw\big(\tau_y(\{x\})<cd(x,y)\big)\leq \frac{\exp(-g_0(c)d(x,y))}{1-\exp(-g_0(c))},
	\end{align*}
	where $x\neq y$. This implies in particular for all $c<c_1(\lambda,\rho)$ that for any $x\in V$
	\begin{align*}
	\Pw(\exists s\geq0: \widetilde{\bfC}^{\{x\}}_{ct}\subset \B_{\lfloor t\rfloor}(x)\,\, \forall t\geq s)=1.
	\end{align*}
\end{lemma}
To understand this result more clearly let us consider Figure~\ref{fig:Growth:a}. In this figure we visualized that the set of all infections expands asymptotically linear in time with some slope $c'>0$. What Lemma~\ref{LinearGrowth} basically states is that for every slope $c<c_1(\lambda,\rho)$ from some time point $s\geq 0$ onwards the boundary of the set of all infected individuals will expand with a steeper slope than $c$, and thus $c<c'$.
\begin{proof}
	With some minor changes one can show  the first claim analogously as in \cite[Lemma 1.9]{durrett1988lecture}.
	In order to find an upper bound on $\tau_y(\{x\})$ one needs to consider all possible infection paths from $x$ to $y$ for the process $\widetilde{\bfC}$. Note that we only need to consider paths without loops. Now assume that we have a path that visits $m$ vertices while getting from $x$ to $y$. Then the total time $T_m$ to reach $y$ only via infections along this path is the sum of $m$ independent exponential random variables with parameter $\lambda$, and thus $T_m\sim \Gamma(n,\lambda)$. One can calculate that
	\begin{equation*}
	    \Pw(T_m<c m)\leq \exp\big(-m(c\lambda-1-\log(c\lambda))\big)
	\end{equation*}
	and note that there are at most $|\cN_x|^m$ possible paths between $x$ and $y$ of length $m$. Summing over all $m\geq d(x,y)$ yields	
	that for $0<c<c_1(\lambda,0)$
	\begin{align*}
	    \Pw\big(\tau_y(\{x\})<c d(x,y)\big)\leq\frac{\exp\big(-d(x,y)(c\lambda-1-\log(c\lambda |\cN_z|))\big)}{1-\exp\big(1-c \lambda+\log(c \lambda |\cN_z|)\big)},
	\end{align*}
	where $z\in V$ is arbitrary and $c\lambda-1-\log(c\lambda |\cN_z|)>0$ since $c<c_1(\lambda,0)$, and thus the first claim follows.

For the second claim we conclude that
	\begin{align*}
	\Pw\big(\widetilde{\bfC}^{\{x\}}_{c(n+1)}\not\subseteq \B_n(x)\big)&\leq\Pw\big(\exists y\in V : d(x,y)=n+1, \tau_y(\{x\})<c(n+1)\big)\\
	&\leq \sum_{y\in V:\,d(x,y)= n+1} \Pw\big(\tau_y(\{x\})<cd(x,y)\big)\\
	&\leq |\partial \B_{n+1}(x)|\frac{\exp(-g_0(c)(n+1))}{1-\exp(-g_0(c))}.
	\end{align*}
	Note that if $c<c_1(\lambda,\rho)$, then $g_0(c)>\rho$. Thus,
	\begin{align*}
	\Pw\big(\widetilde{\bfC}^{\{x\}}_{c(n+1)}\not\subseteq \B_n(x)\big)&\leq\Big(\frac{\sup_{k\geq 0}|\partial \B_{k+1}(x)|e^{-\rho(k+1)}}{1-\exp(-g_0(c))}\Big)\exp(-(\underbrace{g_0(c)-\rho}_{=g_{\rho}(c)>0})(n+1))
	\end{align*}
	and since $G$ is of exponential growth $\rho$ the first factor is finite. Since $g_{\rho}(c)>0$ by a comparison with the geometric sum we see that the right hand side is summable. Thus, applying the Borel-Cantelli Lemma we get that
	\begin{align*}
	\Pw\big(\exists N\geq 1: \widetilde{\bfC}^{\{x\}}_{c(n+1)}\subseteq \B_n(x), \forall n\geq N \big)=1.
	\end{align*}
	Since $\widetilde{\bfC}^{\{x\}}_{ct}\subset \widetilde{\bfC}^{\{x\}}_{c(n+1)}$ for all $t\in(n,n+1]$ it follows that
	\begin{align*}
	\Pw\big(\exists s\geq 0: \widetilde{\bfC}^{\{x\}}_{ct}\subseteq \B_{\lfloor t\rfloor}(x), \forall t\geq s \big)=1. &\qedhere
	\end{align*}
\end{proof}
Next we consider the speed of expansion of the permanently coupled region $\Psi'$. Recall from \eqref{DefinitionPermanetlyCoupledRegion} that the permanently coupled region at time $t$, $\Psi'_t$ contains all edges $e$ such that $e\in \bfB^{B_1}_s\symdiff\bfB^{B_1}_s$ for all $B_1,B_2\subset E$ and all $s\geq t$. In words this means that $e$ is in state open or closed regardless of the initial configuration and from $t$ on wards the state will never depend on the initial configuration again. Thus, we call such an edge permanently coupled.  Recall that $\B^L_k(e)$ denotes the ball of radius $k\in \N$ around an edge $e\in E$ in the line graph $L(G)$.
\begin{proposition}\label{ExpansionSpeedPerCouRegion}
	Let $e\in E$ and $\kappa$ as in Assumption~\ref{AssumptionBackground} $(ii)$. If $c>\kappa^{-1}\rho$, then
	\begin{align*}
	\Pw(\exists s\geq 0: \B^L_{\lfloor t\rfloor+1}(e)\subset\Psi_{ct}' \,\,\forall t\geq s)=1.
	\end{align*}
\end{proposition}
\begin{proof}
	Fix an arbitrary $e\in E$ and recall that by Assumption~\ref{AssumptionBackground} $(ii)$ there exist $T,K,\kappa>0$ such that $\Pw(e\notin  \Psi'_{t})\leq K e^{-\kappa t}$ for all $t>T$. Thus, it follows that
	\begin{align}\label{PartialSets}
	\sum_{n=\lceil T\rceil }^{\infty}\Pw(\B^L_{n+1}(e)\not\subseteq\Psi_{cn}')\leq K\sum_{n=\lceil T\rceil}^{\infty}|\B^L_{n+1}(e)|e^{-\kappa c n}.
	\end{align}
    By assumption we know that $\frac{1}{n}\log(|\B_n(x)|)\to \rho$. Note that transitivity of the graph $G$ implies that $|\B_n(x)|\leq |\B_n^L(e)|\leq |\cN_x||\B_n(x)|$ with $e=\{x,y\}$. Furthermore $G$ is of bounded degree which implies that $\frac{1}{n}\log(|\B^L_n(e)|)\to \rho$. Now using continuity of $\exp(\cdot)$ we see that
    \begin{equation*}
    |\B^L_n|^{1/n}=\exp\big(n^{-1}\log(|\B^L_n|)\big)\to e^\rho
    \end{equation*}
    This implies that
    \begin{equation*}
        \limsup_{n\to \infty}(|\B^L_{n+1}|e^{-\kappa c n})^{1/n+1} =e^{-\kappa c}\limsup_{n\to \infty}(|\B^L_{n+1}|)^{1/n+1}=e^{\rho-\kappa c}<1,
    \end{equation*}
    since by assumption $\kappa c>\rho$. 
     
    Hence, the root test for summability yields that the right hand side of \eqref{PartialSets} is finite. 
    Therefore the Borel-Cantelli Lemma yields
	\begin{align*}
	\Pw(\exists N\geq 1: \B^L_{n+1}(e) \subseteq \Psi_{cn}', \forall n\geq N )=1.
	\end{align*}
	Note that $\Psi_{cn}'\subset \Psi_{ct}'$ for all $t\geq n$, which proves the claim.
\end{proof}

We use these two results such that we can compare the asymptotic expansion speed of the infection and the coupled region. Since one process has values in $\cP(V)$ and the other in $\cP(E)$ we need to introduce the following notation. We denote by
\begin{equation*}
\Phi_{t}:=\{x\in V: \{x,y\}\in \Psi'_t \,\, \forall y\in \cN_x \}.
\end{equation*}
the set of all vertices whose attached edges are already permanently coupled at time $t$. 

\begin{theorem}\label{InfectionVSBackground}
	Let $\lambda>0$, $C\subset V$ be non-empty and finite, $\kappa$ as in Assumption~\ref{AssumptionBackground} $(ii)$. If $c_1(\lambda,\rho)>\kappa^{-1}\rho$, then
	\begin{equation*}
	\Pw(\exists s\geq 0: \widetilde{\bfC}^{C}_t\subseteq\Phi_t\,\, \forall t\geq s )= 1.
	\end{equation*}
\end{theorem}
\begin{proof}
	Let $x\in V$ and $y\in \cN_x$. First we consider $C=\{x\}$. Note that we assumed $c_1(\lambda,\rho)>\kappa^{-1}\rho$, and thus there exists a $c<c_1(\lambda,\rho)$ such that $c\kappa>\rho$. Since $c<c_1(\lambda,\rho)$ by Lemma \ref{LinearGrowth} we get that
	\begin{equation}\label{InfectionVSBackgroundEq1}
	\Pw(\exists s>0: \widetilde{\bfC}^{\{x\}}_{ct}\subset \B_{\lfloor t\rfloor}(x)\,\, \forall t\geq s)=1.
	\end{equation}
	On the other hand we know that $c\kappa>\rho$, and hence Proposition~\ref{ExpansionSpeedPerCouRegion} implies that 
	\begin{equation*}
	\Pw(\exists s>0 :\B^{L}_{\lfloor t\rfloor+1}(\{x,y\}) \subset \Psi'_{ct} \,\, \forall t>s)=1.
	\end{equation*}
	Since $\B^{L}_{\lfloor t\rfloor+1}(\{x,y\})$ contains all edges attached to any vertex in $\B_{\lfloor t\rfloor}(x)$, we see by definition of the random set $\Phi_{ct}$ that
	\begin{equation}\label{InfectionVSBackgroundEq2}
	\Pw(\exists s>0 :\B_{\lfloor t\rfloor}(x) \subset \Phi_{ct} \,\, \forall t>s)=1.
	\end{equation}
	By combining \eqref{InfectionVSBackgroundEq1} and \eqref{InfectionVSBackgroundEq2} we get that
	\begin{equation*}
	\Pw(\exists s\geq 0: \widetilde{\bfC}^{\{x\}}_t\subseteq\Phi_t\,\, \forall t\geq s )= 1.
	\end{equation*}
	Now let $C\subset V$ be an arbitrary non-empty and finite subset. Then we see with additivity that
	\begin{equation*}
	\Pw(\nexists s\geq 0: \widetilde{\bfC}^{C}_t\subseteq\Phi_t\,\, \forall t\geq s )\leq \sum_{x\in C}\Pw( \nexists s\geq 0: \widetilde{\bfC}^{\{x\}}_t \subseteq \Phi_{t}\,\, \forall t\geq s ).
	\end{equation*}
	But we already showed that $\Pw( \nexists s\geq 0: \widetilde{\bfC}^{\{x\}}_t \subseteq \Phi_{t}\,\, \forall t\geq s )=0$ for all $x\in V$ and thus, the right hand side is already equal to $0$. This proves the claim.
\end{proof}
We are finally ready to prove the main results of this chapter. We begin to prove independence of initial configuration of the background process $\bfB$.
\begin{proposition}\label{IndependencySurvival}
	Let $C\subset V$ be finite and non-empty.
    Suppose that $c_1(\lambda,\rho)>\kappa^{-1}\rho$ and $B\subset E$, then $\theta(\lambda,r,C,B)>0$ if and only if $\theta(\lambda,r,C,B')>0$ for all $B'\subset E$.
\end{proposition}
\begin{proof}
	Recall that we assumed that Assumptions~\ref{AssumptionBackground} $(i)$ and $(ii)$ are satisfied. Furthermore let $x\in V$ be fixed. The proof strategy is to use $\theta^{\pi}(\{x\})$ as a reference, i.e.~$\bfB_0\sim \pi$. Note that we omit the infection and recovery rate as variables since they are considered constant throughout the whole proof. By monotonicity it suffices to show that $\theta(C,\emptyset)>0$ if and only if $\theta(C,E)>0$.
	Recall that the translation invariance of the CPERE with $\bfB_0\sim\pi$ 
	implies that $\theta^{\pi}(\{x\})=\theta^{\pi}(\{y\})$ for all $y\in V$ and by additivity also 
	$\theta^{\pi}(\{x\})>0$  if and only if $\theta^{\pi}(C)>0$ for any  finite non-empty set $C$.
	
	Thus, it suffices to show:
	\begin{enumerate}
		\item[a)] If $\theta^{\pi}(\{x\})>0$, then $\theta(\{x\},\emptyset)>0$.
		\item[b)] If $\theta^{\pi}(\{x\})=0$, then $\theta(\{x\},E)=0$.
	\end{enumerate}
	The key idea is that we prove this by coupling the CPERE $(\bfC,\bfB)$ to processes $\underline{\bfC}$ and $\overline{\bfC}$, which act as a upper and lower bound, i.e.~$\underline{\bfC}_0=\bfC_0=\overline{\bfC}_0$ and $\underline{\bfC}_t\subset \bfC_t\subset\overline{\bfC}_t$ for all $t>0$. Note that all three infection processes will depend on the same background process $\bfB$. Let $s>0$, then we define $\underline{\bfC}^{C,B,s}$ as follows.
	\begin{enumerate}
		\item We set $\underline{\bfC}^{C,B,s}_0=C$. On $[0,s]$ we only consider the recovery symbols caused by $\Xi^{\text{rec}}$ and ignore all infection arrows, i.e.~$\mathbf{coop}_{x,y}$ maps.
		\item On $(s,\infty)$ we use the same graphical representation as for the $\bfC^{C,B}$, i.e.~the same infection arrows and recovery symbols generated by $\Xi^{\text{inf}}$ and $\Xi^{\text{rec}}$ and the same background $\bfB^B$.
	\end{enumerate}
	Next we define $\overline{\bfC}^{C,B,s}$ as follows.
	\begin{enumerate}
		\item We set $\overline{\bfC}^{C,B,s}_0=C$. On $[0,s]$ we only consider the infection events caused by $\Xi^{\text{inf}}$. This means we ignore all recovery symbols caused by $\Xi^{\text{rec}}$ and also the background $\bfB^{B}$ in the sense that we treat all edges as open. 
		\item On $(s,\infty)$ we again use the same graphical representation as for $\bfC^{C,B}$ and we use the same background $\bfB^B$.
	\end{enumerate}
	See Figure~\ref{fig:Coupling} for a illustration of $\overline{\bfC}^s$, $\bfC$ and $\underline{\bfC}^s$ on the same realization of $\bfB$.
	\begin{figure}[b]
		\centering 
		\subfigure[\footnotesize Construction of $\bfC$. ]{\label{fig:Coupling:a}\includegraphics[width=43mm]{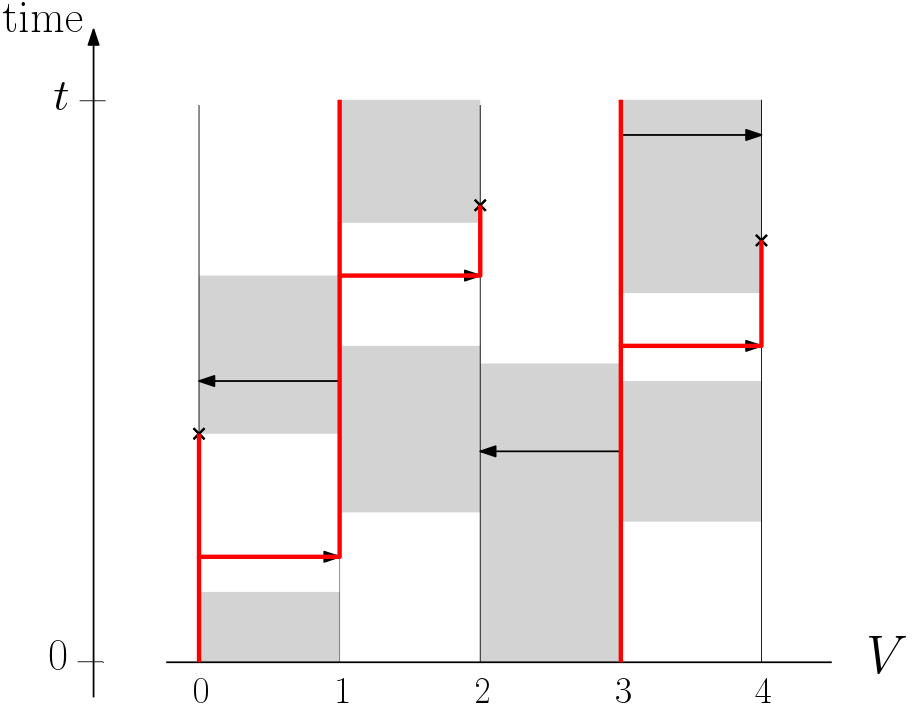}}\hspace{3mm}
		\subfigure[\footnotesize Construction of $\protect\underline{\bfC}^s$.] {\label{fig:Coupling:b}\includegraphics[width=43mm]{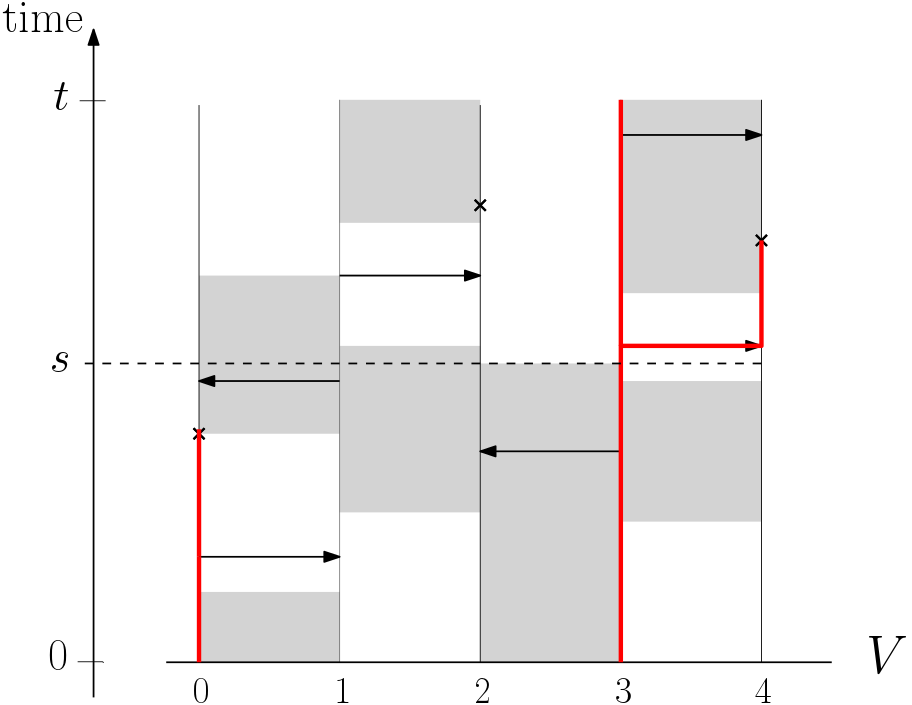}}\hspace{3mm}
		\subfigure[\footnotesize Construction of $\protect\overline{\bfC}^s$.] {\label{fig:Coupling:c}\includegraphics[width=43mm]{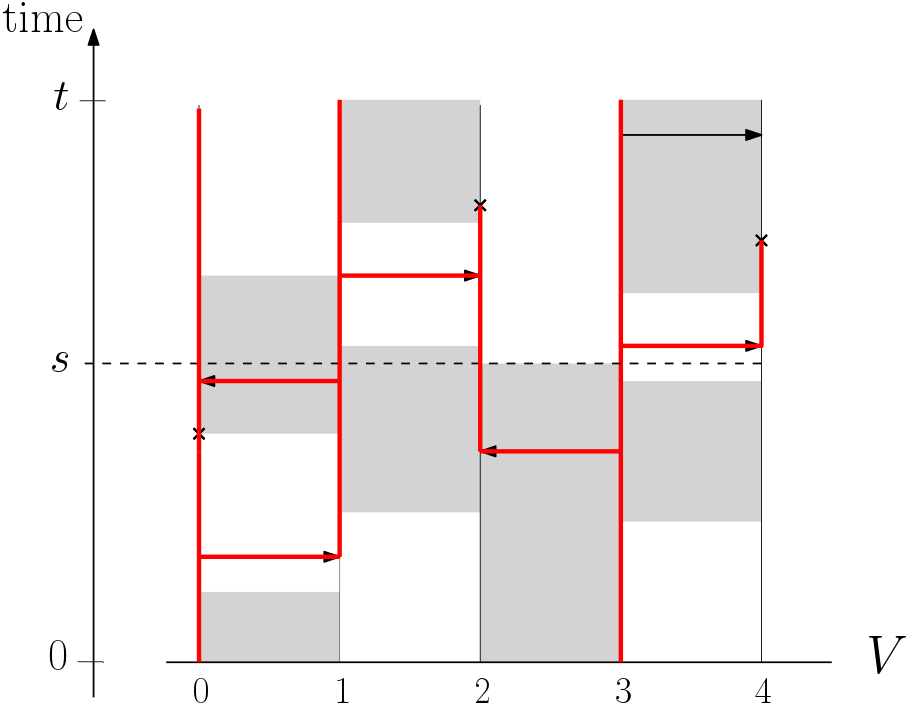}}
		\caption{An illustration of the construction of the three process  $\protect\underline{\bfC}^s$, $\bfC$ and $\protect\overline{\bfC}^s$ using the same graphical representation.}
		\label{fig:Coupling}
	\end{figure}
	
	Recall that $\widetilde{\bfC}^C$ is the classical contact process without recoveries which is coupled to the CPERE $(\bfC^{C,B},\bfB^B)$ such that $\bfC^{C,B}_0= \widetilde{\bfC}^{C}_0=C$ and $\bfC^{C,B}_t\subset \widetilde{\bfC}^{C}_t$ for all $t\geq 0$. By construction $\overline{\bfC}^{C,B,s}_t=\widetilde{\bfC}^{C}_t$ for all $t\leq s$.
	
	We set $A_{s}(C):=\{ \widetilde{\bfC}^{C}_t\subseteq \Phi_t\,\, \forall t\geq s\}$. Another reason why we consider these two processes is that by the construction of $\underline{\bfC}^s$ and $\overline{\bfC}^s$ it is clear that
	\begin{align}
	\Pw(A_{s}(C)\cap\{\underline{\bfC}^{C,\emptyset,s}_t\neq \emptyset \,\, \forall t\geq 0\})&=\Pw(A_{s}(C)\cap\{\underline{\bfC}^{C,E,s}_t\neq \emptyset \,\, \forall t\geq 0\}),\label{HelpIndependence1}\\
	\Pw(A_{s}(C)\cap\{\overline{\bfC}^{C,\emptyset,s}_t\neq \emptyset \,\, \forall t\geq 0\})&=\Pw(A_{s}(C)\cap\{\overline{\bfC}^{C,E,s}_t\neq \emptyset \,\, \forall t\geq 0\}),\label{HelpIndependence2}
	\end{align}
	since both processes are independent of the background $\bfB$ on $[0,s]$ and in the time interval $(s,\infty)$ all infection paths stay in the coupled region, i.e.~the initial configuration of the background process has no influence.
	
	We start by proving a). To avoid clutter we set $A_s:=A_s(\{x\})$.  We see that 
	\begin{align*}
	\theta(\{x\},\emptyset)\geq \Pw(	A_{s}\cap\{\underline{\bfC}^{\{x\},\emptyset,s}_t\neq \emptyset \,\, \forall t\geq 0\})
	\end{align*}
	for every $s>0$ and by \eqref{HelpIndependence1} we get that
	\begin{align}\label{BackIndependencyInequality1}
	\theta(\{x\},\emptyset)\geq \int \Pw(A_{s}\cap\{\underline{\bfC}^{\{x\},B,s}_t\neq \emptyset \,\, \forall t\geq 0\})\pi(\mathsf{d} B).
	\end{align}
	The state $\emptyset$ is obviously an absorbing state for the infection. Hence, 
	\begin{align}\label{BackIndependencyInequality2}
	\begin{aligned}
	&\int \Pw(A_{s}\cap\{\underline{\bfC}^{\{x\},B,s}_t\neq \emptyset \,\, \forall t\geq 0\})\pi(\mathsf{d} B)\\
	=&\int \Pw(A_{s}\cap\{\underline{\bfC}^{\{x\},B,s}_t\neq \emptyset \,\, \forall t\geq s\})\pi(\mathsf{d} B).
	\end{aligned}	
	\end{align}
	Let $\underline{\underline{\bfC}}^s$ be a process which is constructed analogously as $\underline{\bfC}^s$ with the difference that on $[0,s]$ also no recovery symbols have an effect. Therefore, $\underline{\underline{\bfC}}$ is just a delayed CPERE. By construction it is clear that it is only possible for $\underline{\bfC}^{s,\{x\},B}$ to survive if until time $s$ the vertex $x$ is not hit by a recovery symbol, i.e.~if we set $T:=\inf\{t>0:(\mathbf{rec}_{x},t)\in\Xi^{\text{rec}}\}$, then $\underline{\bfC}^{s,\{x\},B}$ goes extinct a.s. on the event $\{T\leq s\}$. Note that $\underline{\underline{\bfC}}^{\{x\},B,s}=\underline{\bfC}^{\{x\},B,s}$ on $\{T> s\}$ and thus,
	\begin{equation}\label{BackIndependencyInequality3}
	\begin{aligned}
	&\int \Pw\big(A_{s}\cap\{\underline{\bfC}^{\{x\},B,s}_t\neq \emptyset \,\, \forall t\geq s\}\big)\pi(\mathsf{d} B)\\
	=&\int \Pw\big(A_{s}\cap\{\underline{\underline{\bfC}}^{\{x\},B,s}_t\neq \emptyset \,\, \forall t\geq s\}\cap \{T>s\})\big)\pi(\mathsf{d} B).
	\end{aligned}	
	\end{equation}
	Furthermore we know that the event $\{T>s\}$ only depends on $\Xi^{\text{rec}}$ in the time interval $[0,s]$. Since $A_s$ only depends on $\Xi^{\text{inf}}$ and the point processes $\Xi^{\text{back}}$ and $\Xi^{\text{rec}}$ have no impact on the survival of $\underline{\underline{\bfC}}$ on $[0,s]$, we get that
	\begin{align}\label{BackIndependencyInequality4}
	\begin{aligned}
	&\int \Pw\big(A_{s}\cap\{\underline{\underline{\bfC}}_t^{\{x\},B,s}\neq \emptyset \,\, \forall t\geq s\}\cap \{T>s\})\big)\pi(\mathsf{d} B)\\
	=&\Pw(T>s)	\int \Pw\big(A_{s}\cap\{\underline{\underline{\bfC}}_t^{\{x\},B,s}\neq \emptyset \,\, \forall t\geq s\}\big)\pi(\mathsf{d} B).
	\end{aligned}
	\end{align}
	By construction it follows that $(\underline{\underline{\bfC}}_t^{s})_{t\leq s}$ and $(\bfB_t)_{t\leq s}$ are independent. Also since $\pi$ is the unique invariant law of the background process we see that 
	\begin{align*}
	\int \Pw\big(\underline{\underline{\bfC}}_t^{\{x\},B,s}\neq \emptyset \,\, \forall t\geq 0\big) \pi(\mathsf{d}B)=\int \Pw\big(\bfC^{\{x\},B}_t\neq \emptyset \,\, \forall t\geq 0\big)\pi(\mathsf{d}B)=\theta^{\pi}(\{x\})>0,
	\end{align*}
	for every $s\geq 0$, where the last inequality follows by assumption. As already mentioned $\underline{\underline{\bfC}}$ is just a delayed CPERE and if it is started stationary the survival probability is constant in $s$.	By Theorem~\ref{InfectionVSBackground} for every $\theta^{\pi}(\{x\})>\varepsilon>0$ there exists a $S>0$ such that $\Pw(A_s)>1-\varepsilon$ for all $s> S$, where we used that $A_s\subset A_{s'}$ if $s\leq s'$.  
	We can use this to conclude that
	\begin{align}\label{IndependencySurvivalEq1}
	\bigg|\int \Pw\big(\underline{\underline{\bfC}}_t^{\{x\},B,s}\neq \emptyset \,\, \forall t\geq 0\big) \pi(\mathsf{d}B)-\int \Pw\big(A_{s}\cap\{\underline{\underline{\bfC}}_t^{\{x\},B,s}\neq \emptyset \,\, \forall t\geq s\}\big)\pi(\mathsf{d} B)\bigg|<\varepsilon.
	\end{align}
	Now using \eqref{BackIndependencyInequality1}-\eqref{IndependencySurvivalEq1} successively yields that $\theta(\{x\},\emptyset)\geq \Pw(T>s)(\theta^{\pi}(\{x\})-\varepsilon)>0$, where we used that $\Pw(T>s)>0$ for all $s\geq 0$. This proves a).
	
	It remains to show b). Here, it suffices to show that
	\begin{align}\label{IndependencySurvivalEq2}
	\Pw(A_s\cap\{\bfC^{\{x\},E}_t\neq \emptyset \,\,\forall t\geq 0\})=0
	\end{align}
	for all $s>0$. This is because Theorem~\ref{InfectionVSBackground} yields that
	\begin{align*}
	\Pw\Big(\bigcup_{n\in\N_0} A_n\Big)=\Pw(\exists s\geq 0: \widetilde{\bfC}^{C}_t\subseteq\Phi_t\,\, \forall t\geq s)=1,
	\end{align*}
	where we used in the first equality that $A_s\subset A_{s'}$ if $s\leq s'$. Hence,
	\begin{align*}
	    \Pw(\bfC^{\{x\}, E}_t\neq \emptyset \,\, \forall t\geq 0)&=\Pw(\{\exists s\geq 0: \widetilde{\bfC}^{\{x\}}_t\subseteq  \Phi_t\, \forall t\geq s\}\cap \{\bfC^{\{x\}, E}_t\neq \emptyset \,\, \forall t\geq 0\})\\
	    &\leq\sum_{n=0}^{\infty}\Pw(A_n \cap \{\bfC^{\{x\}, E}_t\neq \emptyset \,\, \forall t\geq 0\}),
	\end{align*}
	and therefore \eqref{IndependencySurvivalEq2} implies that the right hand side is $0$.
	By construction of $\overline{\bfC}$ we see that
	\begin{equation*}
	    \Pw(A_s\cap\{\bfC^{\{x\},E}_t\neq \emptyset \,\,\forall t\geq 0\})\leq \Pw(A_s\cap\{\overline{\bfC}^{\{x\},E,s}_t\neq \emptyset \,\,\forall t\geq 0\}).
	\end{equation*}
	Furthermore by \eqref{HelpIndependence2} it follows that 
	\begin{equation*}
	    \Pw(A_s\cap\{\overline{\bfC}^{\{x\},E,s}_t\neq \emptyset \,\,\forall t\geq 0\})=\int\Pw(A_s\cap\{\overline{\bfC}^{\{x\},B,s}_t\neq \emptyset \,\,\forall t\geq 0\})\pi(\mathsf{d}B)
	\end{equation*}
	and since $\overline{\bfC}^{\{x\},B,s}_s=\widetilde{\bfC}^{\{x\}}_s$ for all $B\subset E$ we get
	\begin{align*}	
	    \int\Pw(A_s\cap\{\overline{\bfC}^{\{x\},B,s}_t\neq \emptyset \,\,\forall t\geq 0\})\pi(\mathsf{d}B)
	    \leq \E^{\{x\}}[\Pw^{(\widetilde{\bfC}_s,\pi)}(\bfC_t\neq \emptyset \,\,\forall t\geq 0)]=0,
	\end{align*}
	where we used that by assumption $\theta^{\pi}(C)=0$ for all finite $C$ and $|\widetilde{\bfC}^{\{x\}}_s|< \infty$ a.s.. Therefore,
	\begin{equation*}
	    \Pw^{\{x\},E}(A_s\cap\{\bfC^{\{x\},E}_t\neq \emptyset \,\,\forall t\geq 0\})=0
	\end{equation*}
	for all $s\geq 0$, which implies $\theta(\{x\},E)=0$.
\end{proof}
Now we have shown that if the growth condition $c_1(\lambda,\rho)>\kappa^{-1}\rho$ holds, then the chance to survive is independent of the initial configuration of the background process. Next we finally show the main result which is that if for a given $r$ there exist a non-empty and finite set $C'\subset V$ and a set $B'\subset E$ such that $c_1(\lambda_c(r,C',B'),\rho)>\kappa^{-1}\rho$, then it follows that $\lambda_c(r,C,B)=\lambda^{\pi}_c(r)$ for all non-empty and finite $C\subset V$ and $B\subset E$. 
\begin{proof}[Proof of Theorem~\ref{IndependencyCirticalRateSurvival}]
	Let $r>0$ and suppose there exists a non-empty and finite $C'\subset V$ and set $B'\subset E$ such that $c_1\big(\lambda_{c}(r,C',B'),\rho\big)>\kappa^{-1}\rho$. By Lemma~\ref{UniqueSolution} before $\lambda\mapsto c_1(\lambda,\rho)$ is continuous and strictly decreasing. Hence, there exists an $\varepsilon>0$ such that all $\lambda<\lambda_{c}(r,C',B')+\varepsilon$ satisfy $c_1(\lambda,\rho)>\kappa^{-1}\rho$. Now we consider $\lambda<\lambda_{c}(r,C',B')+\varepsilon$. Proposition~\ref{IndependencySurvival} implies in particular that
	\begin{align}\label{Equivalence1}
	\theta(\lambda,r,C',B')>0\Leftrightarrow\theta^\pi(\lambda,r,C')>0.
	\end{align}
	Furthermore, since the CPERE is translation invariant if $\bfB_0\sim \pi$ we see that
	\begin{align}\label{Equivalence2}
	\theta^\pi(\lambda,r,C')>0\Leftrightarrow \theta^\pi(\lambda,r,C)>0,
	\end{align}
	for every non-empty and finite $C\subset V$. This, in particular implies that
	\begin{align*}
	\lambda_{c}(r,C',B')=\lambda^{\pi}(r).
	\end{align*}
	Next we use again that $c_1(\lambda,\rho)>\kappa^{-1}\rho$ such that Proposition~\ref{IndependencySurvival} together with \eqref{Equivalence1} and \eqref{Equivalence2} yield that $\theta(\lambda,r,C',B')>0$ if and only if $\theta(\lambda,r,C,B)>0$ for all non-empty and finite $C\subset V$ and all $B\subset E$. This obviously implies that 
	\begin{align*}
	\lambda_{c}(r,C',B')=\lambda_c^{\pi}(r)=\lambda_{c}(r,C,B)
	\end{align*}
	for all finite and non-empty $C\subset V$ and $B\subset E$.
\end{proof}
\section{The CPERE and its invariant laws}\label{Sec:DualityAndInvariant}
In this section we mainly study the invariant laws of the CPERE. We assume throughout the whole section that the background $\bfB$ satisfies Assumption~\ref{AssumptionBackground} $(i)$-$(iii)$.
\subsection{Upper invariant law and the dual process of $\bfC$}\label{InvariantLaw}
In this subsection we formulate a duality relation for the infection process $\bfC$ and among other things we establish a connection between the survival probability and the upper invariant law.

First we introduce the notion of duality. Let $\bfX=(\bfX_u)_{0\leq u\leq t}$ and $\bfY=(\bfY_u)_{0\leq u\leq t}$ be two processes on the same probability space and let the Polish spaces $\IS_X$ and $\IS_Y$ denote their respective state spaces. 
We call $\bfX$ and $\bfY$ \emph{dual} with respect to a function $H:\IS_X\times \IS_Y\to \R$ if $s\mapsto\E[H(\bfX_{t-s},\bfY_{s})]$ is a constant function for $0\leq s\leq t$.

Now we construct a dual process for the infection process $\bfC$ on a finite time interval $[0,t]$ for any $t\geq 0$. We first fix the background $\bfB$ in the time interval $[0,t]$, i.e.~we set  $\widehat{\bfB}^{B,t}_s:=\bfB^B_{(t-s){-}}$ for all $0\leq s\leq t$. We denote by $\cG:=\sigma(\bfB_s:0\leq s\leq t)$ the $\sigma$-algebra generated from the background process until time $t$. Next we define the dual process $(\widehat{\bfC}^{A,B,t}_s)_{0\leq s\leq t}$ by reversing the time flow and start at $t$. Let $A\subset V$ be the initial configuration at time $t$, i.e.~$\widehat{\bfC}^{A,B,t}_0=A$. We define this process analogously to $\bfC$ via the graphical representation using the same infection and recovery events just backwards in time and the direction of the infection arrows is reversed, i.e.
\begin{align*}
	(u,\mathbf{coop}_{x,y})\to (t-u,\mathbf{coop}_{y,x}) \text{ and } (u,\mathbf{rec}_{x}) \to (t-u,\mathbf{rec}_{x}),
\end{align*}
where $x,y \in V$ such that $\{x,y\}\in E$. Now we see that we coupled $\widehat{\bfC}^t$ to $\bfC$ in such a way that the conditional duality relation
\begin{align}\label{ConditinalDuality}
	\Pw(\bfC^{C,B}_{t}\cap A \neq \emptyset |\cG)=\Pw(\bfC^{C,B}_{s}\cap \widehat{\bfC}^{A,B,t}_{t-s}\neq \emptyset |\cG)=\Pw(C\cap \widehat{\bfC}^{A,B,t}_{t} \neq \emptyset |\cG) 
\end{align}
holds a.s. for all $s\leq t$. Note that the process $\bfC$ and $\widehat{\bfC}^t$ are dual with respect to the functions $H(A,B):=\1_{\{A\cap B\neq \emptyset\}}$. Obviously $(\widehat{\bfC}^{t},\widehat{\bfB}^{t})$ will in general not be CPERE, but this process will nevertheless prove useful. See \autoref{fig:Dualsym} for a illustration of the construction.
\begin{figure}[t]
	\centering
	\includegraphics[width=85mm]{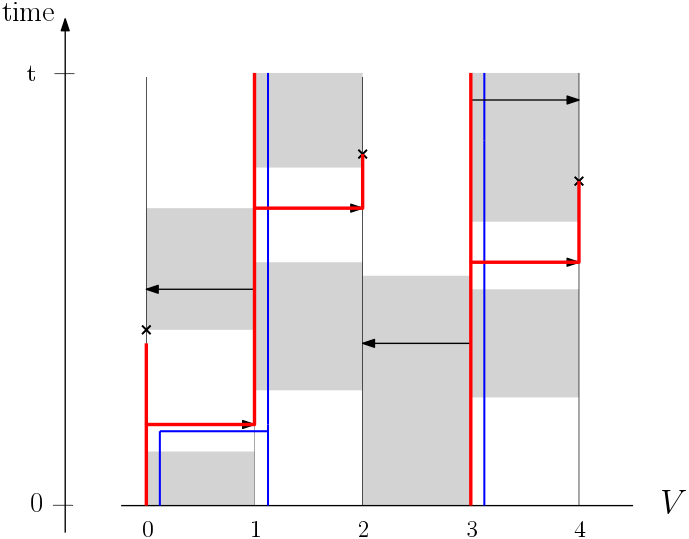}
	\caption{Again arrows and crosses represent respectively infection and recovery events and grey areas are the blocked edges. The red lines are again the infection paths of the forward-time process $\bfC$. The blue lines are the infection paths backwards in time with respect to the mirrored arrows, which are used to define the process $\widehat{\bfC}$.}
	\label{fig:Dualsym}
\end{figure}

Next we show amongst other things that we can recover a self duality in the case where we assume stationarity of $\bfB$, i.e.~$\bfB_0\sim \pi$.
\begin{proposition}[Distributional duality]\label{DistributionalDuality}
	Let $t\geq 0$, $A,C\subseteq V$ and $B,H\subset E$ then
	\begin{align*}
		s\mapsto\Pw(\bfC^{C,B}_{s}\cap\widehat{\bfC}_{t-s}^{A,B,t} \neq \emptyset,\bfB^B_t\cap H\neq \emptyset) \quad \text{ and }\quad s\mapsto\Pw(\bfC^{C,B}_{s}\cap\widehat{\bfC}_{t-s}^{A,B,t} \neq \emptyset)
	\end{align*}
	are constant functions. If $\bfB$ is reversible this implies in particular that for all $t\geq 0$
	\begin{align*}
		\Pw^{(C,\pi)}(\bfC_{t}\cap A \neq \emptyset)=\Pw^{(A,\pi)}(\bfC_{t}\cap C \neq \emptyset).
	\end{align*}
\end{proposition}
\begin{proof} 
	Let $t\geq 0$. By using \eqref{ConditinalDuality} we see that
	\begin{align*}
	\Pw(\bfC^{C,B}_{t}\cap A \neq \emptyset,\bfB^B_t\cap H\neq \emptyset)&=\E[\Pw(\bfC^{C,B}_{t}\cap A \neq \emptyset|\cG)\1_{\{\bfB^B_t\cap H\neq \emptyset\}}]\\
	&=\E[\Pw(\bfC^{C,B}_{s}\cap \widehat{\bfC}^{A,B,t}_{t-s}\neq \emptyset |\cG)\1_{\{\bfB^B_t\cap H\neq \emptyset\}}]\\
	&=\Pw(\bfC^{C,B}_{s}\cap \widehat{\bfC}^{A,B,t}_{t-s}\neq \emptyset,\bfB^B_t\cap H\neq \emptyset)
	\end{align*}
	for all $s\leq t$. The equality $\Pw(\bfC^{C,B}_{t}\cap A \neq \emptyset)=\Pw(\bfC^{C,B}_{s}\cap \widehat{\bfC}^{A,B,t}_{t-s}\neq \emptyset)$ for all $s\leq t$ follows by taking the expectation in \eqref{ConditinalDuality}, which proves the first claim.
	
	For the second claim choose $s=0$ and integrate both sides with respect to $\pi$, and thus
	\begin{equation}\label{DualityEquality1}
		\int \Pw^{(C,B)}(\bfC_{t}\cap A \neq \emptyset)\pi(\mathsf{d} B)=	\int\Pw(\widehat{\bfC}^{A,B,t}_{t}\cap C\neq \emptyset)\pi(\mathsf{d} B).
	\end{equation}
	We assumed that $\bfB$ is reversible with respect to its invariant law $\pi$. Let us consider $(\bfB_s)_{s\leq t}$ with $\bfB_0\sim\pi$ and as before set $\widehat{\bfB}_s^{\pi,t}:=\bfB_{(t-s)-}$ for $0\leq s\leq t$, then by reversibility it follows that $(\bfB_s)_{s\leq t}\stackrel{d}{=}(\widehat{\bfB}^{\pi,t}_{s})_{s\leq t}$. 
	We define again via the reversed graphical representation the process $(\widehat{\bfC}^{A,\pi,t}_s)_{s\leq t}$ with respect to the background $(\widehat{\bfB}^{\pi,t}_{s})_{s\leq t}$. Now the process $(\widehat{\bfC}^{A,\pi,t}_s ,\widehat{\bfB}^{\pi,t}_{s})_{s\leq t}$ is again a CPERE with initial distribution $\delta_{A}\otimes \pi$. Hence, this fact together with \eqref{DualityEquality1} yields that
	\begin{equation*}
		\Pw^{(C,\pi)}(\bfC_{t}\cap A \neq \emptyset)=\Pw^{(A,\pi)}(\bfC_{t}\cap C \neq \emptyset).\qedhere
	\end{equation*}
\end{proof}
Recall that we denoted by $T(t)=T_{\lambda,r}(t)$ the Feller semigroup corresponding to the CPERE $(\bfC,\bfB)$ with parameters $\lambda$ and $r$ and by $\preceq$ the stochastic order. Since we assumed that $\bfB$ is an attractive spin system it follows immediately that $(\bfC,\bfB)$ is also attractive and thus by standard methods follows that $(\delta_{V}\otimes\delta_{E})T(t)\Rightarrow\overline{\nu}$ as $t\to \infty$, which provides the existence of the upper invariant law $\overline{\nu}$ of $(\bfC,\bfB)$ .

It is called the upper invariant law since, if $\nu$ is an arbitrary invariant law of $(\bfC,\bfB)$, then $\nu\preceq \overline{\nu}$. Furthermore it is monotone in $\lambda$ and $r$, i.e.~if  $\lambda_1\leq \lambda_2$,  then $\overline{\nu}_{\lambda_1,r}\preceq \overline{\nu}_{\lambda_2,r}$ and if $r_1\geq r_2$ then $\overline{\nu}_{\lambda,r_1}\preceq \overline{\nu}_{\lambda,r_2}$.

It is not necessary to start the background with every edge being open, i.e.~$\bfB_0=E$, to ensure convergence towards the upper invariant law. As long as the initial distribution of the background dominates $\pi$ stochastically, this ensures convergence towards $\overline{\nu}$.
\begin{lemma}\label{ConvergenceToUpperInv}
	Let $\mu$ be a probability measure with $\pi\preceq \mu$ then $(\delta_{V}\otimes\mu)T(t)\Rightarrow \overline{\nu}$  as $t\to \infty$.
\end{lemma}
\begin{proof}
	First of all it is clear that $\delta_{V}\otimes \pi\preceq \delta_{V}\otimes \mu$, and therefore
	\begin{align*}
		\lim_{t\to \infty}(\delta_{V}\otimes \pi)T(t)\preceq \lim_{t\to \infty}(\delta_{V}\otimes \mu)T(t).
	\end{align*}
	if the limit exists. So its enough to prove convergence for $\pi=\mu$. Since $\pi$ is the invariant law of the background and the infection process can only occupy fewer vertices than all of $V$ it follows that $(\delta_{V}\otimes \pi)T(s)\preceq(\delta_{V}\otimes \pi)$ for all $s\geq 0$
	and by Lemma~\ref{MonotonicityAdditivityLemma} we get that
	\begin{equation*}
		(\delta_{V}\otimes \pi) T(t+s)\preceq(\delta_{V}\otimes \pi) T(t) \quad \text{ for all } t,s\geq 0.
	\end{equation*}
	Since $(\delta_{V}\otimes \pi) T(t)$ is bounded from below and decreasing with respect to $\preceq$ it follows that $(\delta_{V}\otimes \pi)T(t)\Rightarrow\nu'$ as  $t\to \infty$. Since $\overline{\nu}$ is the upper invariant law we know that $\nu'\preceq\overline{\nu}$, i.e.~it suffices to show that $\overline{\nu}\preceq\nu'$ holds. By Assumption~\ref{AssumptionBackground} $(i)$ we know that $\pi$ is the unique invariant law of $\bfB$. Thus, the second marginal of any invariant law of $(\bfC,\bfB)$ must be $\pi$. Hence, it is clear that for every invariant law $\nu$, it must hold that $\nu\preceq\delta_{V}\otimes \pi$. Therefore, by monotonicity and stationarity we know that
	\begin{align*}
		\nu=\nu T(t) \preceq(\delta_{V}\otimes \pi)T(t) \Rightarrow \nu' \quad\text{   as } t\to \infty.
	\end{align*}
	Since this is true for any invariant law $\nu$ it also holds for the upper invariant law, i.e.~$\nu=\overline{\nu}$.
\end{proof}
This enables us to deduce a connection between the survival probability $\theta^{\pi}$ of the infection process $\bfC$ started with stationary background and the upper invariant law $\overline{\nu}$ in the next result.
\begin{proposition}\label{EqualityOfCriticalValues}
	Let $C\subset V$ be finite, then
	\begin{align*}
		\theta^{\pi}(C)=\Pw^{(C,\pi)}(\bfC_t\neq \emptyset\,\, \forall t\geq 0)=\overline{\nu}(\{A\subset V:C\cap A\neq\emptyset\}\times \cP(E)),
	\end{align*}
	and thus in particular $\theta^{\pi}(\lambda,r,\{x\})>0$ if and only if $\overline{\nu}_{\lambda,r}\neq  \delta_{\emptyset}\otimes\pi$,	where $x\in V$ is arbitrary.
\end{proposition}
\begin{proof}
	By the self duality relation from Proposition~\ref{DistributionalDuality} we get for $C\subset V$ 
	\begin{align*}
		\Pw^{(V,\pi)}(\bfC_{t}\cap C \neq \emptyset )=\Pw^{(C,\pi)}(\bfC_t\neq \emptyset)\to \Pw^{(C,\pi)}(\bfC_t\neq \emptyset\,\, \forall t\geq 0)\quad \text{ as } t\to \infty,
	\end{align*}
	where we used continuity of the probability measure. On the other hand, since $C$ is finite we get
	\begin{align*}
		\Pw^{(V,\pi)}(\bfC_{t}\cap C\neq \emptyset)&=\int \1_{\{ A \cap C\neq \emptyset\}}(\delta_V \otimes \pi)T(t)(\mathsf{d}(A,B))
		\to\int \1_{\{ A \cap C\neq \emptyset\}}\overline{\nu}(\mathsf{d}(A,B))
	\end{align*}
	as $t\to \infty$, where we used Lemma~\ref{ConvergenceToUpperInv}. Now we can conclude that
	\begin{align}\label{EquationUIL&SP}
		\overline{\nu}(\{A\subset V:A\cap C\neq\emptyset\}\times\cP(E) )=\Pw^{(C,\pi)}(\bfC_t\neq \emptyset\,\, \forall t\geq 0).
	\end{align}
	which yields the first claim.
	
	By translation invariance we know that $\theta^{\pi}(\{x\})=\theta^{\pi}(\{y\})$ for all $x,y\in V$. This yields in particular that the second claim does not depend on the choice of $x$. 
	Now choose $C=\{x\}$ for some $x\in V$. Suppose that $\theta^{\pi}(\{x\})>0$, then we see by \eqref{EquationUIL&SP} that
	\begin{align}\label{UIL&SPHelp1}
		\overline{\nu}(\{(A,B)\in \cP(V)\times\cP(E):x\in A\})>0.
	\end{align}
	This implies that $\overline{\nu}\neq \delta_{\emptyset}\otimes \pi$. For the converse direction we assume that $\theta^{\pi}(\{x\})=0$, and hence $\theta^{\pi}(\{y\})=0$ for all $y\in V$. Now we see by \eqref{EquationUIL&SP} that
	\begin{equation}\label{UIL&SPHelp2}
	\overline{\nu}(\{(A,B)\in \cP(V)\times\cP(E):y\in A\})=0.
	\end{equation}
	for all $y\in V$. Now let us consider the set $\cD:=\{(A,B)\in \cP(V)\times\cP(E):A\neq \emptyset\}$.
	By using $\sigma$-subadditivity and \eqref{UIL&SPHelp2} we see that
	\begin{equation*}
		\overline{\nu}(\cD)\leq \sum_{y\in V}\overline{\nu}(\{(A,B)\in \cP(V)\times\cP(E):y\in A\})=0,
	\end{equation*}
	and thus it follows that $\overline{\nu}= \delta_{\emptyset}\otimes \pi$. This provides the second claim.
\end{proof}

As a consequence we can show that the critical value $\lambda'_c(r)$ of the phase transition between triviality and non-triviality of the upper invariant law indeed agrees with the critical value for survival $\lambda^{\pi}_c(r)$, where the background is assumed to be stationary. If we additionally assume that $c_1\big(\lambda^{\pi}_c(r),\rho\big)>\kappa^{-1} \rho$, then we know that the critical infection rate of survival does not depend on the initial configuration.  
\begin{proof}[Proof of Theorem~\ref{CrticalValuesAgree}]
	Let $r>0$, then as a direct consequence of Proposition~\ref{EqualityOfCriticalValues} follows that $\lambda'_c(r)=\lambda_c^{\pi}(r)$. If we assume additionally $c_1\big(\lambda^{\pi}_c(r),\rho\big)>\kappa^{-1} \rho$ by Theorem~\ref{IndependencyCirticalRateSurvival} follows that there exists a $\lambda_{c}(r)$ such that $\lambda_{c}(r)=\lambda_c(r,C,B)$ for every $C\subset V$ non-empty and finite and every $B\subset E$, and thus in particular $\lambda'_c(r)=\lambda_c(r)$.
\end{proof}
For the remainder of this section we provide some results which we need in the subsequent sections. where we consider continuity properties of the survival probability and complete convergence.
\begin{proposition}\label{NonnegagtivitiyLemma}
	The measure $\overline{\nu}$ has the property that $\overline{\nu}(\{\emptyset\}\times\cP(E))\in \{0,1\}$.
\end{proposition}
\begin{proof}
    This can be shown analogously as for the CP. See \cite[Proposition~6.4]{seiler2021} for a proof of this result.
\end{proof}
	A consequence of this proposition is that if $\overline{\nu}\neq \delta_{\emptyset}\otimes \pi$, then
	\begin{align}\label{SurvivalOfAllInfected}
		\begin{aligned}
			\lim_{n\to \infty}\theta^{\pi}(\B_n(x))&=\lim_{n\to \infty}\overline{\nu}(\{A\subset V:A\cap \B_n(x)\neq\emptyset\}\times\cP(E) )\\
			&=\overline{\nu}(\{A\subset V:A\neq\emptyset\}\times\cP(E) )=1,
		\end{aligned}
	\end{align}
	where $x\in V$ is arbitrary and we used Proposition~\ref{EqualityOfCriticalValues} in the first equality. We want to extend this result to
	\begin{equation*}
	\lim_{n\to \infty}\theta(\B_n(x),\emptyset)=1.
	\end{equation*}
	Recall that $\bfB$ is an autonomous Feller process. Thus, we denote by $(S(t))_{t\geq 0}$ the Feller semigroup associated with the background process. Let $s>0$, then we set $\pi_{s}:=\delta_{\emptyset}S(s)$ and
	\begin{equation*}%\label{SequenceOfMeasures}
		\theta^{\pi_s}(C):=\int\Pw(\bfC_t^{C,B}\neq \emptyset\,\, \forall t\geq 0)\pi_{s}(\mathsf{d}B).
	\end{equation*}
	By Assumption \ref{AssumptionBackground} $(i)$ there exists a unique invariant law $\pi$ of the background process $\bfB$ such that $\pi_s\Rightarrow \pi$ as $s\to \infty$. Recall that $\widetilde{\bfC}$ denotes a classical contact process with infection rate $\lambda>0$ without recovery, i.e.~only infection arrows are taken into account and the background as well as recovery symbols are completely ignored.
\begin{lemma}\label{MaximalSpread}
	Let $t>0$, $\varepsilon>0$ and $A\subset V$ finite. Then there exists a finite $D=D(t,\varepsilon, A)\subset V$ such that 
	\begin{equation*}
		\Pw(\widetilde{\bfC}_t^{A}\subset D)>1-\varepsilon.
		\end{equation*}
	\end{lemma}
\begin{proof}
	Let $t>0$ and $A\subset V$ finite and fixed. We know that for every finite initial configuration $A$ the random set $|\widetilde{\bfC}_t^{A}|<\infty$ a.s.. This implies that for some $x\in A$,
	\begin{equation*}
		\Pw\big(\exists n\geq 1:\widetilde{\bfC}_t^{A}\subset \B_n(x)\big)=1.
	\end{equation*}
	Thus, since $\{\widetilde{\bfC}_t^{A}\subset \B_n(x)\}\subset\{\widetilde{\bfC}_t^{A}\subset \B_m(x)\}$ if $m\geq n$ and because of continuity of $\Pw$, it follows that for every $\varepsilon>0$ there exists an $N\in \N$ such that
	\begin{align*}
		\Pw\big(\widetilde{\bfC}_t^{A}\subset \B_n(x)\big)>1-\varepsilon
	\end{align*}
	for all $n>N$, which proves the claim.
	\end{proof}
Recall that $\B^L_{n}(e)$ denotes the ball in the line graph $L(G)$ of radius $n\in \N$ with $e\in E$ as centre.
\begin{lemma}\label{CommonDistribution}
	Let $e\in E$ and $k\in \N$. There exists a probability law $\mu_s$ on $\cP(E^2)$ with marginals $\pi$ and $\pi_s$ such that for every $\varepsilon>0$ there exists a $s>0$ such that
	\begin{align*}
		\mu_s\big(\{(B,D)\in E^2:B\cap \B^L_{k}(e)= D\cap \B^L_{k}(e)\}\big)>1-\varepsilon.
	\end{align*}
\end{lemma}
\begin{proof}
	Let $\bfB^{\pi}$ be the background process such that $\bfB^{\pi}_0\sim \pi$. Now let $\bfB^{\pi}$ be coupled to $\bfB^{\emptyset}$ via the graphical representation. Recall that the coupled region was defined by
	\begin{equation*}
		\Psi_{t}=\{e\in E: e\notin \bfB_t^{B_1}\symdiff\bfB_t^{B_2}\,\,\forall B_1,B_2\subset E\}.
	\end{equation*}
	Choose $c>0$ such that $c\kappa>\rho$. By Theorem \ref{ExpansionSpeedPerCouRegion} we know that
	\begin{align*}
		\Pw(\exists s\geq 0: \B^L_{t+1}(e)\subset \Psi_{ct} \,\, \forall t\geq s)=1.
	\end{align*}
	By continuity of the law $\Pw$ and monotonicity of the event, there exists an $s>k$ such that $\Pw(\B^L_{t+1}(e)\subset \Psi_{ct} \,\, \forall t\geq s)>1-\varepsilon$, which in particular implies that
	\begin{align*}
		\Pw(\bfB_{cs}^{\pi}\cap \B^L_{k}(e)= \bfB_{cs}^{\emptyset}\cap \B^L_{k}(e))>1-\varepsilon.
	\end{align*}
	Now set $s'=cs$ and let $\mu_{s'}$ be the joint probability distribution of $(\bfB_{s'}^{\pi},\bfB_{s'}^{\emptyset})$. This distribution satisfies the claim.
\end{proof}
\begin{lemma}\label{BounOnDecouplingLemma}
	Let $\bfB$ be the background process with spin rate $q(\cdot,\cdot)$, $B\subset E$ and $e\in B$. Furthermore, let $u> 0$ and $n\in \N$, then for every $\varepsilon>0$ there exists a $k>n$ such that for all sets $D\subset E$ with $B\cap \B^L_{k}(e)=D\cap \B^L_{k}(e)$,
	\begin{align*}
		\Pw(\bfB^{B}_t\cap \B^L_{n}(e)=\bfB^D_t\cap \B^L_{n}(e) \,\, \forall t< u)>1-\varepsilon.
	\end{align*}
\end{lemma}
\begin{proof}
    The idea of the proof is that for a finite range spin systems the state of some edge $e$ can only be influenced by the state of a second edge $a$, with high probability, after an at most linear amount of time proportional to the graph distance between $e$ and $a$. See \cite[Proposition~3.2.5]{seiler2021} for a detailed proof.
\end{proof}
With these three lemmas we are able to show the following useful approximation result of the survival probability. Recall that $c_1(\lambda,\rho)$ is the solution of \eqref{UniqueSolution2}, $\kappa$ is the constant from Assumption~\ref{AssumptionBackground} $(ii)$ and $\rho$ denotes the exponential growth of the graph $G$.
\begin{lemma}\label{ApproximationLemma}
	Let $\lambda,r>0$ and suppose that $c_1(\lambda,\rho)>\kappa^{-1} \rho$. Then for any $C\subset V$,
	\begin{align*}
	\lim_{s\to \infty}\theta^{\pi_{s}}(\lambda,r,C)=\theta^{\pi}(\lambda,r,C).
	\end{align*}
\end{lemma}
\begin{proof}
	Note that if $|C|=\infty$ or $C=\emptyset$ the statement is trivial, since either both sides are $1$ or $0$. Thus, we assume that $C$ is a finite non-empty subset of $V$. Fix $x\in C$ and $y\in \cN_x$. Since $c_1(\lambda,\rho)>\kappa^{-1}\rho$ by Proposition~\ref{InfectionVSBackground} we know that
	\begin{align*}
	\Pw(\exists u\geq 0: \widetilde{\bfC}_t^{C}\subset \Phi_t \,\, \forall t\geq u)=1.
	\end{align*}
	Set $A^1_{u}(C):=\{ \widetilde{\bfC}^{C}_t\subseteq \Phi_t\,\, \forall t\geq u\}$. We see that for every $\varepsilon>0$ there exists a $T>0$ such that $
	\Pw(A^1_{u}(C))\geq 1-\varepsilon$
	for all $u\geq T$, where we used that $A^1_{u}(C)\subset A^1_{u'}(C)$ for $u\leq u'$ and continuity of the law $\Pw$.
	
	Next we fix $u\geq T$ and define $A^2_{u,m}(C):=\{\widetilde{\bfC}_t^{C}\subset \B_m(x) \,\, \forall t\leq u\}$ for $m\in \N$. By Lemma~\ref{MaximalSpread} we can choose a $m=m(u)$ large enough such that $\Pw(A^2_{u,m}(C))>1-\varepsilon$. This yields that
	\begin{equation}\label{HelpApproximation1}
		\theta(C,B)\leq \Pw\big(A_u^1(C)\cap A^2_{u,m}(C)
		\cap\{\bfC^{C,B}_t\neq \emptyset\,\, \forall t\geq 0\}\big)+2\varepsilon
	\end{equation}
	for any $B\subset E$. By Lemma~\ref{BounOnDecouplingLemma} we can choose a $k=k(m)>m+1$ large enough such that
	\begin{equation}\label{HelpApproximation2}
		\Pw(\bfB^{B}_t\cap\B^L_{m+1}(\{x,y\})=\bfB^D_t\cap\B^L_{m+1}(\{x,y\}) \,\, \forall t\leq u)>1-\varepsilon,
	\end{equation}
	for any $D\subset E$ with $B\cap \B^L_{k}(\{x,y\})=D\cap \B^L_{k}(\{x,y\})$. Note that $\B^L_{m+1}(\{x,y\})$ contains in particular all edges which are attached to all vertices in $\B_{m}(x)$. Now for notational convenience define $A^3_{u,m}(C):=A_u^1(C)\cap A^2_{u,m}(C)$. Furthermore, we set
	\begin{align*}
		A_{u,m}(C,(B,D)):=&\{\bfB^{B}_t\cap\B^L_{m+1}(\{x,y\})=\bfB^D_t\cap\B^L_{m+1}(\{x,y\}) \,\, \forall t\leq u\}\cap A^3_{u,m}(C),\\
		E_{k}(B,D):=&\{(B,D)\in E^2:B\cap \B^L_{k}(\{x,y\})=D\cap \B^L_{k}(\{x,y\})\}.
	\end{align*}
	By Lemma~\ref{CommonDistribution} there exists a distribution $\mu_s$ on $\cP(E^2)$ with marginals $\pi$ and $\pi_s$, such that for $s>0$ large enough 
	\begin{equation}\label{HelpApproximation3}
	\mu_{s}\big(E_{k}(B,D)\big)>1-\varepsilon.
	\end{equation}
	Note that by choice of these events
	\begin{align}\label{HelpApproximation4}
	\begin{aligned}
		&A_{m,u}(C,(B,D)) \cap\{\bfC^{C,B}_t\neq \emptyset\,\, \forall t\geq 0\}\\
		=&A_{m,u}(C,(B,D)) \cap\{\bfC^{C,D}_t\neq \emptyset\,\, \forall t\geq 0\}
		\subset \{\bfC^{C,D}_t\neq \emptyset\,\, \forall t\geq 0\},
	\end{aligned}
	\end{align}
	since on the event $A_{m,u}(C,(B,D))$ the infection stays in $\B_{m}(x)$ until time $u$ and afterwards only travels along edges already contained in the permanently coupled region. But, for any of the initial configuration $B$ or $D$ the background does not differ in the ball $\B^L_{m+1}(\{x,y\})$ at any time $t\in [0,u]$ and thus, we can interchange $B$ and $D$ on $A_{m,u}(C,(B,D))$.
	Finally we can conclude that
	\begin{align}\label{HelpApproximation5}
		\begin{aligned}
			&\int \Pw\big(A_u^1(C)\cap A^2_{u,m}(C)
			\cap\{\bfC^{C,B}_t\neq \emptyset\,\, \forall t\geq 0\}\big)\pi(\mathsf{d}B)\\
			\leq&\int \Pw\big(A_{m,u}(C,(B,D))\cap \{\bfC^{C,B}_t\neq \emptyset\,\, \forall t\geq 0\}\big)\1_{E_k(B,D)}\mu_{s}(\mathsf{d}(B,D))+2\varepsilon\\
			\leq&\int \Pw\big(\bfC^{C,D}_t\neq \emptyset\,\, \forall t\geq 0\}\big)\pi_{s}(\mathsf{d}D)+2\varepsilon,
		\end{aligned}
	\end{align}
	where we used \eqref{HelpApproximation2} and \eqref{HelpApproximation3} in the first inequality and in the second the definition of $E_k(B,D)$ together with \eqref{HelpApproximation4}. Hence, by combining  \eqref{HelpApproximation1} and \eqref{HelpApproximation5} we obtain  
	\begin{align*}
	\theta^{\pi}(C)\leq \theta^{\pi_s}(C)+4\varepsilon.
	\end{align*}
	On the other hand we have that $\pi_s=\delta_{\emptyset} S(s)$. Since $\bfB$ is by assumption a monotone Feller process we get that $\pi_s\preceq\pi$ for all $s\geq 0$, and thus by monotonicity of the survival probability it follows that $\theta^{\pi_s}(C)\leq \theta^{\pi}(C)\leq \theta^{\pi_s}(C)+4\varepsilon$, which proves the claim.
\end{proof}
With this approximation result we are able to show the desired result. 
\begin{lemma}\label{SurvivalContinuityLemma2}
	Let $x\in V$ and $r>0$. Suppose that $c_1(\lambda_c^{\pi}(r),\rho)>\kappa^{-1} \rho$, then for all $\lambda>\lambda_c(r)=\lambda_c^{\pi}(r)$ 
	\begin{equation*}
		\lim_{n\to \infty}\theta(\lambda,r,\B_n(x),\emptyset)=1.
	\end{equation*}
\end{lemma}
\begin{proof}
	Let us fix $x\in V$. By Lemma~\ref{UniqueSolution} we know that $\lambda\mapsto c_1(\lambda,\rho)$ is continuous and strictly decreasing. Thus, if $c_1(\lambda^{\pi}(r),\rho)>\kappa^{-1} \rho$, then there exists an $\varepsilon'>0$ such that $c_1(\lambda,\rho)>\kappa^{-1}\rho$ for all $\lambda\in (\lambda_c^{\pi}(r),\lambda_c^{\pi}(r)+\varepsilon')$. Note that by Theorem~\ref{IndependencyCirticalRateSurvival} $\lambda^{\pi}_c(r)=\lambda_c(r)$. Let $n\geq 0$ and fix $\lambda\in (\lambda_c^{\pi}(r),\lambda_c^{\pi}(r)+\varepsilon')$ by \eqref{SurvivalOfAllInfected} we know that for every $\varepsilon>0$ there exists $n$ large enough such that $\theta^{\pi}(\B_n(x))>1-\varepsilon$ and by Lemma~\ref{ApproximationLemma} we know that for given $n$ and $\varepsilon$ there exist $s>0$ large enough such that
	\begin{equation}\label{SurvivalContinuityLemmaHelp1}
		\theta^{\pi_{s}}(\B_n(x))>1-\varepsilon.
	\end{equation}
	Choose a set $\{x_i: i\in\N \}\subset V$ such that $d(x_i,x_j)> 2n$ for $i\neq j$. Note that by this choice the sets $(\B_{n}(x_i))_{i\in\N}$ are disjoint. Let us consider the event
	\begin{align*}
	A_{m,n}^s:=\{\exists i\leq m: (t,\mathbf{rec}_x)\notin \Xi^{\text{rec}} \,\,\forall (t,x)\in[0,s]\times\B_n(x_i)\},
	\end{align*}
	i.e.~for some $i\leq m$ no recovery symbols occurs up to time $s$ in $\B_n(x_i)$. For given $s$ and $n$ choose $m$ large enough such that
	\begin{equation}\label{SurvivalContinuityLemmaHelp2}
		\Pw(A_{m,n}^s)>1-\varepsilon.
	\end{equation}
	Let $k=k(m,n)$ be large enough such that $\bigcup_{i=1}^{m}\B_{n}(x_i)\subset\B_{k}(x)$. Now by the choice of $s$ it follows that
	\begin{align}\label{SurvivalContinuityLemmaHelp3}
		\Pw(\bfC^{\B_{k}(x),\emptyset}_t\neq \emptyset\,\, \forall t\geq 0|A_{m,n}^s)\geq \Pw^{(\B_{n}(x),\pi_{s})}(\bfC_t\neq \emptyset\,\, \forall t\geq 0)=\theta^{\pi_s}(\B_{n}(x))
	\end{align}
	where we used the translation invariance of $(\bfC,\bfB)$ and that $A_{m,n}^s$ is independent of the background. Now by \eqref{SurvivalContinuityLemmaHelp1}, \eqref{SurvivalContinuityLemmaHelp2} and \eqref{SurvivalContinuityLemmaHelp3} we get that
	\begin{align*}
		\theta(\B_k(x),\emptyset)\geq \Pw(\bfC^{\B_{k}(x),\emptyset}_t\neq \emptyset\,\, \forall t\geq 0|A_{m,n}^s)\Pw(A_{m,n}^s)\geq \theta^{\pi_s}(\B_{n}(x))(1-\varepsilon) \geq (1-\varepsilon)^2,
	\end{align*}
	which yields that $\lim_{n\to \infty}\theta(\lambda,r,\B_n(x),\emptyset)=1$, for all $\lambda\in (\lambda_c(r),\lambda_{c}(r)+\varepsilon)$. Since the map $\lambda\mapsto c_1(\lambda,\rho)$ is strictly decreasing it is possible that there exists $\lambda'>\lambda$ such that $c_1(\lambda',\rho)>\kappa^{-1}\rho$ is no longer satisfied. In this case we can use monotonicity and see that
	\begin{equation*}
		\lim_{n\to \infty}\theta(\lambda',r,\B_n(x),\emptyset)\geq \lim_{n\to \infty}\theta(\lambda,r,\B_n(x),\emptyset)=1.\qedhere
	\end{equation*}
\end{proof}

\subsection{Continuity of the survival probability}\label{SectionContinuity}
In this section we study continuity of the survival probability with respect to the infection rate $\lambda$ and recovery rate $r$. We start by determining on which regions of the parameter space the functions
\begin{equation*}
	\lambda\mapsto\theta(\lambda,r,C,B) \quad \text{ and }\quad  r\mapsto\theta(\lambda,r,C,B)
\end{equation*}
are left or right continuous. Before we proceed we need the following result concerning the limit of a sequence of monotone and continuous functions.
\begin{lemma}\label{RightContinuity}
	Let $f,f_n:(0,\infty)\to [0,1]$ for every $n\geq 1$ with $\lim_{n\to \infty} f_n(x)=f(x)$ for all $x\in (0,\infty)$. Let $f_n$ be a continuous and monotone function for all $n\in \N$, and furthermore $f_n(x)\geq f_{n+1}(x)$ for all $x\in (0,\infty)$. If $f_n$ is a increasing function for all $n\in \N$, it follows that $f$ is right continuous and if $f_n$ is decreasing, then $f$ is left continuous. 
\end{lemma}
\begin{proof}
	See \cite{seiler2021} for a proof of this result.
\end{proof}
As a direct consequence of this lemma we can conclude right continuity in the following proposition.
\begin{proposition}\label{RightContinuityOfSurv}
	Let $C\subset V$ and $B\subset E$. Then, for $r>0$ the function
	\begin{align*}
		\lambda\mapsto\theta(\lambda,r,C,B),
	\end{align*}
	is right continuous on $(0,\infty)$ and for $\lambda>0$ the function $r\mapsto\theta(\lambda,r,C,B)$ is left continuous on $(0,\infty)$.
\end{proposition}
\begin{proof}
	By Lemma \ref{ContinuityLemma} we know that the function $\lambda\mapsto
	\Pw^{(C,B)}_{\lambda,r}(\bfC_t\neq \emptyset)$ is continuous for any $t\geq 0$ and also
	$\Pw^{(C,B)}_{\lambda,r}(\bfC_s\neq \emptyset)\geq \Pw^{(C,B)}_{\lambda,r}(\bfC_t\neq \emptyset)$ if $s\leq t$. Thus, we can conclude that
	\begin{equation*}
	\Pw^{(C,B)}_{\lambda,r}(\bfC_t\neq \emptyset)\downarrow \theta(\lambda,r,C,B) \quad \text{ as } \quad t\to \infty,
	\end{equation*}
	by continuity of $\Pw$. Since $\Pw^{(C,B)}_{\lambda,r}(\bfC_t\neq \emptyset)$ is increasing with respect to the infection rate $\lambda$, we can use Lemma~\ref{RightContinuity} to conclude that $\lambda\mapsto\theta(\lambda,r,C,B)$ is right continuous.
	
	Analogously it follows that $r\mapsto\theta(\lambda,r,C,B)$ is left continuous since $\Pw^{(C,B)}_{\lambda,r}(\bfC_t\neq \emptyset)$ is decreasing with respect to the recovery rate $r$.
\end{proof}
The continuity from the respective other side is more difficult to prove. Before we proceed with this we need the following somewhat technical result.
\begin{lemma}\label{SurvivalContinuityLemma1}
	Let $(\bfC,\bfB)$ be a CPERE, $\emptyset \neq C\subset V$ be finite and $B\subset E$. Set
	\begin{align*}
		D_{n,t}(C,B)&:=\{ \exists x\in V \text{ such that } \B_n(x) \subseteq \bfC^{C,B}_s \text{ for some } s\leq t\}
	\end{align*}
	for $n\in \N$ and $t\geq 0$. In words $D_{n,t}(C,B)$ is the event that for some $s\leq t$ there exists a vertex $x$ such that all vertices in the ball $\B_n(x)$ with centre $x$ are infected at time $s$. Then 
	\begin{align*}
		\lim_{t\to \infty}\Pw\big(D_{n,t}(C,B)\big)\geq  \theta(C,B) \quad  \text{ for all } n\in \N.
	\end{align*}
\end{lemma}
\begin{proof}
	We can assume that $\pi\neq \delta_{\emptyset}$ since otherwise the survival probability is $0$ which makes the statement trivial. We omit for most parts of the proof the initial configuration $(C,B)$ since it remains unchanged throughout this proof. Note that since $D_{n,t}$ is increasing in $t$, it follows that $\lim_{t\to \infty}\Pw(D_{n,t})=\Pw(D_{n,\infty})$.	The idea of this proof is that if a vertex $x$ is infected at time $k\in \N$, i.e.~$x\in \bfC_k$, the probability that all vertices in a radius of $n$ get infected by time $k+1$, i.e.~$\bfC_{k+1}\supseteq \B_n(x)$, is positive for every fixed $n\in \N$. But if we assume that $\bfC$ survives we know that for every $t\geq 0$ there exists an $x\in V$ such that $x\in \bfC_t$ and this will imply $\Pw_{\lambda,r}(D_{n,\infty})\geq \theta(\lambda,r)$ for every $n\in \N$. In fact
	\begin{equation*}
		\{\bfC_t\neq \emptyset \,\, \forall t\geq 0\,\}=\{\forall k\in \N_0, \,\exists x\in V  \text{ such that } x\in \bfC_k\}
	\end{equation*}
	since $\emptyset$ is an absorbing state.
	
	Recall that $\cF_k$ is the $\sigma$-algebra generated from the Poisson point processes $\Xi$ used in the graphical representation until time $k$. Then we set
	\begin{equation*}
		\varepsilon=\varepsilon(n)=\Pw^{(\{x\},\emptyset)}( \bfC_1 \supseteq \B_n(x))>0.
	\end{equation*} We see that $\Pw( \bfC_{k+1}\supseteq \B_n(x) |\cF_k)\geq \varepsilon$ a.s. on $\{x\in \bfC_k\}$, where we used monotonicity with respect to the initial configurations (see Lemma \ref{MonotonicityAdditivityLemma}). This yields that for any $x^*\in V$
	\begin{align*}
		\Pw\Big(\bigcup_{x\in V}\{\bfC_{k+1}\supseteq \B_n(x)\}\Big|\cF_k\Big)\geq\Pw(\bfC_{k+1}\supseteq \B_n(x^*)|\cF_k) \geq \varepsilon .
	\end{align*}
	a.s. on $\{x^*\in \bfC_k\}$. We set $A^n_{k+1}:=\bigcup_{x\in V}\{\bfC_{k+1}\supseteq \B_n(x)\}\in \cF_{k+1}$ for $k\in \N_0$. We see that
	\begin{align*}
		\sum_{k=0}^{\infty}\Pw(A^n_{k+1}|\cF_k)=\infty \,\,\, \text{ a.s.~on }\,\,\, \{\forall k\in \N_0, \, \exists x\in V \text{ such that } x\in \bfC_k\}.
	\end{align*}
	Now we can use the extension of the Borel-Cantelli Lemma, found in \cite[Theorem 4.3.4]{durrett2019probability} and get that
	\begin{equation*}
		\Big\{\sum_{k=0}^{\infty}\Pw(A^n_{k+1}|\cF_k)=\infty\Big\}=\{ A^n_{k} \text{ i.o.}\}.
	\end{equation*}
	This implies $\{\bfC_t=\emptyset\,\,\forall t\geq 0\}\subset \{ A_{k} \text{ i.o.}\}$. Obviously $\Pw^{(C,B)}( A^n_{k} \text{ i.o.} )\leq \theta(C,B)$, and thus with what we just shown it follows that actually $\Pw^{(C,B)}( A^n_{k} \text{ i.o.} )= \theta(C,B)$ holds. This yields for all $n> 0$
	\begin{equation*}
		\Pw_{\lambda,r}(D_{n,\infty}(C,B))\geq \Pw^{(C,B)}_{\lambda,r}( A^{n}_{k} \text{ i.o.})= \theta(\lambda,r,C,B).\qedhere
	\end{equation*}
\end{proof}
Finally, we are prepared to prove the second continuity property. Recall from \eqref{SurvivalRegionWithGrowth} that
\begin{equation*}
	\cS_{c_1}=\{(\lambda,r): \exists \lambda'\leq \lambda \text{ s.t. } (\lambda',r)\in\cS(\{x\},\emptyset) \text{ and }  c_1(\lambda',\rho)>\kappa^{-1}\rho\},
\end{equation*}
where $\cS(\{x\},\emptyset)$ denotes the survival region for the initial configuration $(\{x\},\emptyset)$ defined in \eqref{SurvivalRegion}, i.e.~$(\lambda,r)\in\cS(\{x\},\emptyset)$ if and only if $\theta(\lambda,r,\{x\},\emptyset)>0$.
\begin{proposition}\label{LeftContinuityOfSurv}
	Let $C\subset V$ and $B\subset E$. 
	\begin{enumerate}
		\item[$(i)$] For $r>0$ the function $\lambda\mapsto\theta(\lambda,r,C,B)$
		is left continuous, and thus continuous, on $\{\lambda:(\lambda,r)\in 
		\inte(\cS_{c_1})$.
		\item[$(ii)$] For $\lambda>0$ the function $r\mapsto\theta(\lambda,r,C,B)$
		is right continuous, and thus continuous, on $\{r:(\lambda,r)\in
		\inte(\cS_{c_1})$.
	\end{enumerate} 	
\end{proposition}
\vspace{-1em}
\begin{proof}
	We assume that $C\subset V$ is finite and non-empty. Otherwise the survival probability is $0$ or $1$ and a constant function is obviously continuous. 
	We only show $(i)$ since $(ii)$ follows analogously, i.e only some minor changes are needed in the proof. We fix $r>0$ and assume that $\{\lambda:(\lambda,r)\in 
	\inte(\cS_{c_1})\neq \emptyset\}$. Thus, let $(\lambda,r)\in 
	\inte(\cS_{c_1})$, fix some $x\in V$ and define
	$\tau=\tau_n:=\inf\{t\geq 0 : \exists x\in V \text{ s.t. } \bfC_t\supseteq \B_n(x)\}$,
	where $n\in \N$. We see that
	\begin{align*}
	\theta(\lambda)&=\Pw(\bfC_s\neq \emptyset\,\, \forall s\geq 0)\geq \Pw(\{ \tau<t\}\cap\{\bfC_s\neq \emptyset\,\, \forall s\geq \tau\})\\
	&= \E[\1_{\{ \tau<t\}}\Pw(\bfC_s\neq \emptyset\,\, \forall s\geq \tau|\cF_\tau)]
	\end{align*}
 	for any $t\geq 0$, where we used again that if $\bfC_t\neq \emptyset$ for $t\geq \tau$, then this must also be true for all $t\leq \tau$. Now we use the fact that $(\bfC,\bfB)$ is a Feller process, and see that 
	\begin{align*}
	\Pw(\bfC_s\neq \emptyset\,\, \forall s\geq \tau|\cF_\tau)=\Pw\big(\bfC_{\tau+s}\neq \emptyset\,\, \forall s\geq 0\,\big|\,(\bfC_{\tau},\bfB_{\tau}) \big),
	\end{align*}
	where we used the strong Markov property.
	From the definition of $\tau$ it is clear that there exists an $x\in V$ such that $\bfC_{\tau}\supseteq \B_n(x)$. Now we know that
	\begin{align*}
	\Pw\big(\bfC_{\tau+s}\neq \emptyset\,\, \forall s\geq 0\,\big|\,(\bfC_{\tau},\bfB_{\tau}) \big)\geq \Pw^{(\B_n(x),\emptyset)}\big(\bfC_{s}\neq \emptyset\,\, \forall s\geq 0\big),
	\end{align*}
	and by translation invariance the right-hand side is independent of $x$. Thus, we can omit the vertex $x$ and write $\B_n$. So we get that
	\begin{align*}	
	\theta(\lambda)\geq \Pw_{\lambda}(D_{n,t})\Pw_{\lambda}^{(\B_n,\emptyset)}\big(\bfC_{s}\neq \emptyset\,\, \forall s\geq 0\big)=\Pw_{\lambda}(D_{n,t})\theta(\lambda,\B_n,\emptyset),
	\end{align*}
	where we used that $\{\tau<t\}=D_{n,t}$. The set $D_{n,t}$ is defined as in Lemma~\ref{SurvivalContinuityLemma1}. Now let $\lambda_c(r)<\lambda''<\lambda'<\lambda$, and thus $\lambda',\lambda''\in \{\lambda:(\lambda,r)\in 	\inte(\cS_{c_1})\}$. 
	Then we see that
	\begin{align*}
	\theta(\lambda')\geq \Pw_{\lambda'}(D_{n,t})\theta(\lambda',\B_n,\emptyset)\ge \Pw_{\lambda'}(D_{n,t})\theta(\lambda'',\B_n,\emptyset),
	\end{align*}
	where we used monotonicity, see  Lemma~\ref{MonotonicityAdditivityLemma}. Letting $\lambda'\uparrow \lambda$ yields
	\begin{align}\label{ContHelp1}
	\theta(\lambda-)\geq \Pw_{\lambda}(D_{n,t})\theta(\lambda'',\B_n,\emptyset),
	\end{align}
	where we used continuity of $\lambda \mapsto\Pw_{\lambda}(D_{n,t})$ which follows by Lemma~\ref{ContinuityLemma}. Recall that $(C,B)$ was the initial configuration of the CPERE, using Lemma \ref{SurvivalContinuityLemma1} we get that
	\begin{align*}
		\lim_{t\to \infty}\Pw_{\lambda}\big(D_{n,t}(C,B)\big)\geq  \theta(\lambda,C,B),
	\end{align*}
	and thus letting $t\to \infty$ in \eqref{ContHelp1} yields
	\begin{align*}
	\theta(\lambda-,C,B)\geq \lim_{t\to \infty} \Pw_{\lambda}(D_{n,t}(C,B))\theta(\lambda'',\B_n,\emptyset) \geq\theta(\lambda,C,B)\theta(\lambda'',\B_n,\emptyset).
	\end{align*}
	Since we know that $\lambda''\in \{\lambda:(\lambda,r)\in \inte(\cS_{c_1})\}$, by Lemma \ref{SurvivalContinuityLemma2} it follows that
	\begin{align*}
	\theta(\lambda'',\B_n,\emptyset)\to 1 \text{ as } n\to \infty.
	\end{align*}
	Putting everything together yields $\theta(\lambda-,C,B)\geq \theta(\lambda,C,B)$. Since we know that this function is increasing in $\lambda$, this yields left continuity on the set $\{\lambda:(\lambda,r)\in
	\inte(\cS_{c_1})\}$. Right continuity, and therefore continuity follows by Proposition~\ref{RightContinuityOfSurv}.
\end{proof}
We end this section with the following proof:
\begin{proof}[Proof of Proposition~\ref{ContinuityTheorem}]
\begin{enumerate}
    \item By Proposition~\ref{RightContinuityOfSurv} and Proposition~\ref{LeftContinuityOfSurv} it follows that 
	\begin{align*}
		(\lambda,r)\mapsto\theta(\lambda,r,C,B)
	\end{align*}
	is separately continuous on the open set $
	\inte(\cS_{c_1})\subset \R^2$, which means that the function is continuous in all variable separately, i.e.~$\lambda\mapsto\theta(\lambda,r,C,B)$ and $r\mapsto\theta(\lambda,r,C,B)$ are continuous on $\{\lambda:(\lambda,r)\in 	\inte(\cS_{c_1})\}$ and $\{r:(\lambda,r)\in 
	\inte(\cS_{c_1})\}$ respectively. Since the survival probability $\theta$ is monotone in the infection rate $\lambda$ and the recovery rate $r$ it follows that the function is jointly continuous on $	\inte(\cS_{c_1})$, see 
	\cite[Proposition~2]{kruse1969joint}.
	\item If $\rho=0$, then it follows that $\cS_{c_1}=\cS(C,B)$ for all $(C,B)$ with $C$ non-empty but finite. This implies in particular that the critical infection rate $\lambda_c(r)$ is independent of the initial configuration $(C,B)$, where we used Theorem~\ref{IndependencyCirticalRateSurvival}. We know by Proposition~\ref{LeftContinuityOfSurv} that the function $\lambda\mapsto\theta(\lambda,r,C,B)$ is left continuous on $(\lambda_c(r),\infty)$ and is zero on $(0,\lambda_c(r))$. Now by assumption $\theta(\lambda_c(r),r,C,B)=0$, and thus the function is left continuous on $(0,\infty)$. But we also know by Proposition~\ref{RightContinuityOfSurv} that $\lambda\mapsto\theta(\lambda,r,C,B)$ is right continuous on $(0,\infty)$, and therefore continuous on $(0,\infty)$. Continuity in $r$ follows analogously, and thus it follows that 
	\begin{align*}
		(\lambda,r)\mapsto\theta(\lambda,r,C,B)
	\end{align*}
	is separately continuous on $(0,\infty)^2$. We can conclude, analogously as in $(i)$, the joint continuity with \cite[Proposition~2]{kruse1969joint}.
\end{enumerate}
\end{proof}
\subsection{Equivalent conditions for complete convergence}\label{CompleteConvergenceChapter}
This section is dedicated to proving Theorem~\ref{CompleteConvergenceComplete}. Recall that $c_1(\lambda,\rho)$ is the solution of \eqref{UniqueSolution2}, $\kappa$ is the constant from Assumption~\ref{AssumptionBackground} $(ii)$ and $\rho$ denotes the exponential growth of the graph $G$. 

Therefore, the main goal is to show that if for given $\lambda,r>0$ there exists a $\lambda'\leq \lambda$ such that for $(\lambda',r)$ the two conditions \eqref{ConvergenceCond1.1} and \eqref{ConvergenceCond2.1} are satisfied, i.e.
\begin{align}\label{ConvergenceCond1}
\Pw_{\lambda',r}^{(C,B)}(x\in \bfC_t \text{ i.o.})=\theta(\lambda',r,C,B)
\end{align}
for all $x\in V$, $C\subset V$ and $B\subset E$ and
\begin{align}\label{ConvergenceCond2}
\lim_{n\to\infty}\limsup_{t\to \infty}\Pw_{\lambda',r}(\bfC_t^{\B_n(x),\emptyset}\cap \B_n(x)\neq\emptyset)=1
\end{align}
for any $x\in V$, then this implies complete convergence of the CPERE, i.e.
\begin{align}\label{CompleteConvergence1}
	(\bfC_t^{C,B},\bfB_t^B)\Rightarrow \theta(\lambda,r,C,B)\overline{\nu}+[1-\theta(\lambda,r,C,B)](\delta_{\emptyset}\otimes\pi)\quad \text{ as } t\to \infty
\end{align}
for all $C\subset V$ and $B\subset E$. On the other hand if for $\lambda$ the inequality $c_1(\lambda,\rho)>\kappa^{-1} \rho$ and \eqref{CompleteConvergence1} are satisfied, then this already implies that \eqref{ConvergenceCond1} and \eqref{ConvergenceCond2} are satisfied. We begin with the first part, and thus show convergence of the marginals $\bfC$ and $\bfB$ and then conclude that this already implies that the CPERE $(\bfC,\bfB)$ convergences.

By Assumption \ref{AssumptionBackground} $(i)$ we already know that $\bfB_t^{B} \Rightarrow \pi$ as $t\to \infty$ for all $B\subset E$. Hence it remains to show that the two conditions \eqref{ConvergenceCond1} and \eqref{ConvergenceCond2} imply that the infection process $\bfC$ convergences weakly as $t\to \infty$. We show that for any $C\subset V$ and $B\subset E$
\begin{equation}\label{ConvergenceCond3}
	\Pw_{\lambda,r}(\bfC_{t}^{C,B}\cap C'\neq \emptyset)\to \theta(\lambda,r,C,B)\theta^{\pi}(\lambda,r,C'),
\end{equation}
as $t\to \infty$ for every $C'\subset V$ finite, which suffices to conclude weak convergence of the infection process $\bfC$ since the function class $\{\1_{\{\cdot\,\cap C'\neq \emptyset\}}: C'\subset V\text{ finite}\}$ is convergence determining.
Then we show that the converse holds true as well, which provides that \eqref{CompleteConvergence1} implies \eqref{ConvergenceCond1} and \eqref{ConvergenceCond2}. Once we know that the marginals converge we show that this already implies the convergence of the joint distribution, i.e.~$(\bfC_t,\bfB_t)$ converges weakly as $t\to \infty$.

We first introduce some shorthand notation to keep the formulas somewhat cleaner. For $A\subset V$ we set
\begin{equation*}
\begin{aligned}
    A_E&:=\big\{\{x,y\}\in E: x\in A\big\},\\
	A^N&:=\bigcup_{x\in A}\B_{N}(x),\\
	A_E^N&:=\big\{\{x,y\}\in E: x\in A^N\big\},
\end{aligned}
\end{equation*}
where $\B_N(x)$ is the ball with centre $x$ and radius $N$ with respect to the graph distance of $G$.

Let $(\widecheck{\bfB}^{s/2}_r )_{r\geq s/2}$ denote a process with same dynamics as the background process $\bfB$, which is coupled with the original background in such a way that it starts at time $s/2$ with an initial distribution $\pi$ and is assumed to be independent from $(\bfB^{B}_r)_{r\leq s/2}$, but from $s/2$ onward it uses the same graphical representation as $\bfB^B$. For an illustration see Figure~\ref{fig:RestartedDual:b}. 
\begin{lemma}\label{ControlBackground}
	Let $D\subset V$ be finite and $B\subset E$. Then for every $\varepsilon>0$ there exists an $S>0$ such that for all $s\geq S$
	\begin{equation*}
		\Pw\big(\widecheck{\bfB}^{s/2}_u\cap D_E=\bfB^{B}_u\cap D_E \,\, \forall u\geq s\big)>1-\varepsilon.
	\end{equation*}
\end{lemma}
\begin{proof}
	Let $x\in D$. Let $c>0$ be chosen such that $c\kappa>\rho$. Then by Proposition~\ref{ExpansionSpeedPerCouRegion} we know that $\Pw(\exists s\geq 0: \B_{\lfloor c^{-1}t \rfloor}(x) \subseteq \Phi_{t}\,\, \forall t\geq s )=1$. Let $S'>0$ be chosen such that $D\subset \B_{\lfloor c^{-1}t \rfloor}$ for all $t\geq S'/2$. By continuity of the measure $\Pw$, for every $\varepsilon>0$ there exists an $S>S'>0$ such that $\Pw\big(D \subset \B_{\lfloor c^{-1}t \rfloor}(x) \subseteq \Phi_{t}, \forall t\geq \tfrac{S}{2} \big)>1-\varepsilon$. 
	Now fix an $s \geq S.$
	By 	using the Markov property 
	we can interpret $(\bfB^{B}_{u})_{u\geq s/2}$ as a background process started with the initial configuration $\bfB^{B}_{s/2}$, and $(\widecheck{\bfB}^{s/2}_u)_{u\geq s/2}$ is also a background process which is coupled via the graphical representation to the former with the difference that the initial distribution is its stationary distribution $\pi$  chosen to be independent of $\bfB^{B}_{s/2}$. Since at time $u \geq s$ these processes have evolved for a time of length $u-s/2 \geq s/2 \geq S/2$ we can apply the above result in order to conclude that for all $s\geq S$,
	\begin{equation*}
		\Pw(\widecheck{\bfB}^{s/2}_u\cap D_E=\bfB^{B}_u\cap D_E \,\, \forall u\geq s)>1-\varepsilon.\qedhere
	\end{equation*}
\end{proof}
Let $t,s>0$ and recall the dual process $(\widehat{\bfC}^{A,B,t+s}_r)_{r\leq t+s}$ of $(\bfC_r^{C,B})_{r\leq t+s}$. In the definition of the dual process we fixed the background $(\bfB^{B}_r)_{r\leq t+s}$, reversed the graphical representation with respect to the time axis at the time point $t+s$ and fixed $A$ as the initial set of infected vertices for the dual process.

Now let $(\widecheck{\bfC}_{u}^{A,s/2,t+s})_{u\leq t+s/2}$ be a process coupled to $\widehat{\bfC}^{A,B,t+s}$ by using the same time-reversed infection arrows and recovery symbols, but the background at time $s/2$ %(forwards in time)
is reset and independently drawn according to the law $\pi$, i.e.~we use $(\widecheck{\bfB}^{s/2}_r )_{r\geq s/2}$ instead of $(\bfB^{B}_r)_{r\geq s/2}$. Again see Figure~\ref{fig:RestartedDual} for a illustration.
\begin{figure}[t]
	\centering 
	\subfigure[Illustration of the graphical representation of $(\bfC_u,\bfB_u)_{u\leq t+s}$ and the dual $(\widehat{\bfC}^{t+s}_u)_{u\leq t+s}$. ]{\label{fig:RestartedDual:a}\includegraphics[width=68.5mm]{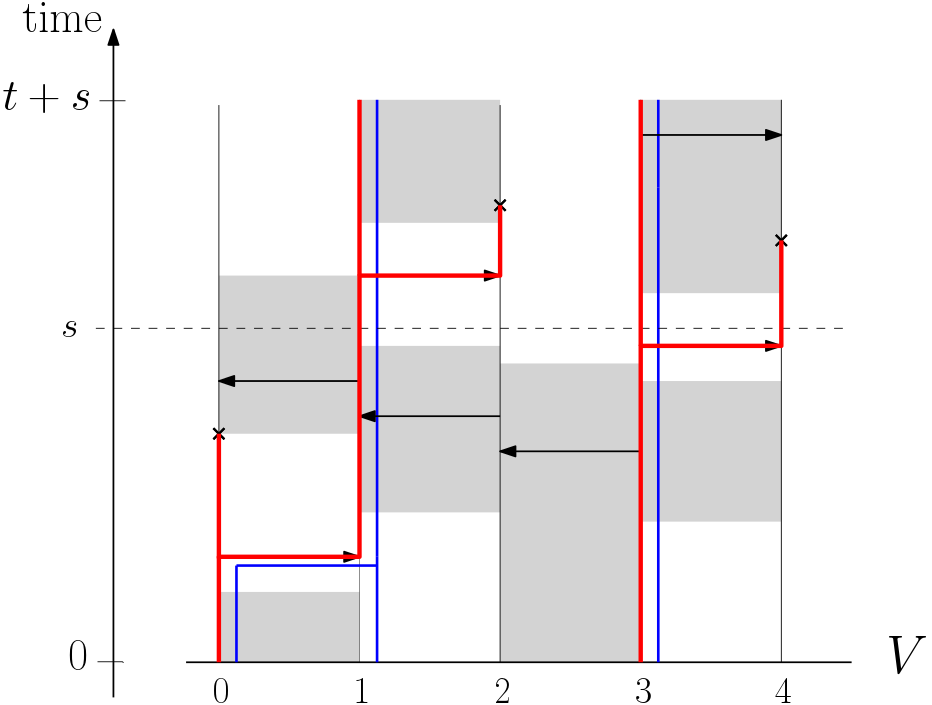}}\hfill
	\subfigure[Illustration of the graphical representation of $(\widecheck{\bfB}^{s/2}_{u})_{s/2\leq u \leq t}$ and $(\widecheck{\bfC}^{s/2,t+s}_u)_{u\leq t+s/2}$. ]{\label{fig:RestartedDual:b}\includegraphics[width=68.5mm]{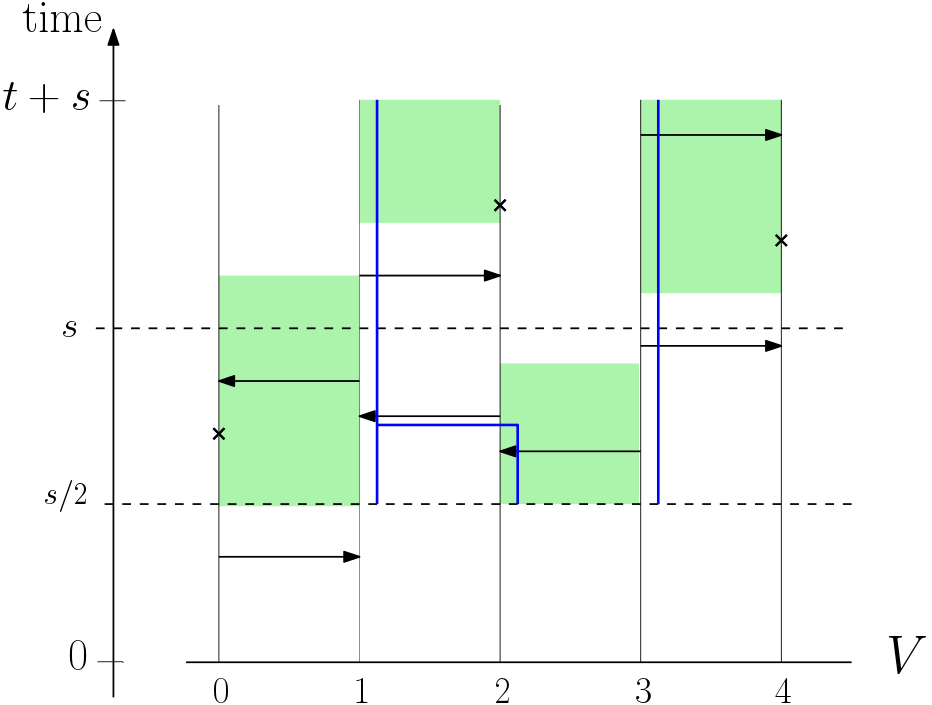}}
	\caption{As usual the arrows and crosses denote the infection and recovery symbols. The grey area illustrate the closed edges according to $\bfB$ (left picture) and the green areas the closed edges according to $\widecheck{\bfB}$ (right picture). The red lines are the infection paths forward in time, i.e.~$\bfC$, and the blue lines the infection path backward in time, i.e.~$\widehat{\bfC}$ in the left and $\widecheck{\bfC}$ on the right.}
	\label{fig:RestartedDual}
\end{figure}
\begin{lemma}\label{ControlDualProcess}
	Let $t>0$, $A\subset V$ be finite and $B\subset E$. Then for every $\varepsilon>0$ there exists an $S>0$ such that for all $s>S$,
	\begin{equation*}
		\Pw(\widehat{\bfC}_{u}^{A,B,t+s}=\widecheck{\bfC}_{u}^{A,s/2,t+s}\,\, \forall u\leq t)>1-\varepsilon.
	\end{equation*}
\end{lemma}
\begin{proof}
	First, by Lemma \ref{MaximalSpread} we know that for every $\varepsilon_1>0$ there exists a finite $D=D(t,\varepsilon_1,A)\subset V$ such that
	\begin{equation*}
		\Pw(\widehat{\bfC}_{u}^{A,B,t+s},\widecheck{\bfC}_{u}^{A,s/2,t+s}\subset D \,\, \forall u\leq t)>1-\varepsilon_1.
	\end{equation*}
	For $D$ given via Lemma \ref{ControlBackground} we get that for every $\varepsilon_2>0$ there exists an $S>0$ such that for every $s>S$ 
	\begin{equation*}
		\Pw(\widecheck{\bfB}^{s/2}_u\cap D_E=\bfB^{B}_u\cap D_E \,\, \forall u\geq s)>1-\varepsilon_2.
	\end{equation*}
	Recall that $D_E\subset E$ was the set which contains every edge attached to $D$. Now we see that
	\begin{equation*}
	    \begin{aligned}
	        \{\widehat{\bfC}_{u}^{A,B,t+s},\widecheck{\bfC}_{u}^{A,s/2,t+s}&\subset D \,\, \forall u\leq t\}\cap \{\widecheck{\bfB}^{s/2}_u\cap D_E=\bfB^{B}_u\cap D_E \,\, \forall u\geq s\}\\
		    \subseteq &\,\,\{\widehat{\bfC}_{u}^{A,B,t+s}=\widecheck{\bfC}_{u}^{A,s/2,t+s}\,\, \forall u\leq t\}.
	    \end{aligned}
	\end{equation*}
	But by choosing $\varepsilon_1,\varepsilon_2\leq\frac{\varepsilon}{2}$ we see that
	$\Pw(\widehat{\bfC}_{u}^{A,B,t+s}\neq\widecheck{\bfC}_{u}^{A,s/2,t+s}\,\, \forall u\leq t)\leq \varepsilon$, which yields the claim.
\end{proof}
Now we can begin to show the convergence of the first marginal. We will split this in two steps by first proving an upper bound and in the second step we use  \eqref{ConvergenceCond1} and \eqref{ConvergenceCond2} to show that this upper bound also acts as a lower bound which provides the desired result. Note that the subsequent upper bound still holds true if the growth assumption $c_1(\lambda,\rho)>\kappa\rho^{-1}$ is not satisfied.
\begin{proposition}\label{UpperBoundConvergence}
	Let $t,s>0$, $C,C'\subset V$ with $C'$ being finite and $B\subset E$, then for every $\varepsilon>0$ there exist $S,T>0$ such that
	\begin{align*}
		\Pw_{\lambda,r}(\bfC_{s}^{C,B}\neq \emptyset,\widehat{\bfC}_{t}^{C',B,t+s}\neq \emptyset)\leq  \theta(\lambda,r,C,B)\theta^{\pi}(\lambda,r,C')+\varepsilon
	\end{align*}
	for all $s>S$, $t>T$. This implies in particular for any finite $C'\subset V$
	\begin{align*}
		\limsup_{t\to \infty} \Pw_{\lambda,r}(\bfC_{t}^{C,B}\cap C'\neq \emptyset)\leq  \theta(\lambda,r,C,B)\theta^{\pi}(\lambda,r,C').
	\end{align*}
\end{proposition}
\begin{proof}
	By Proposition~\ref{DistributionalDuality} it follows that
	\begin{align*}
		\Pw(\bfC_{t+s}^{C,B}\cap C'\neq \emptyset)=\Pw(\bfC_{s}^{C,B}\cap \widehat{\bfC}_{t}^{C',B,t+s}\neq \emptyset)\leq\Pw(\bfC_{s}^{C,B}\neq \emptyset,\widehat{\bfC}_{t}^{C',B,t+s}\neq \emptyset).
	\end{align*}
	Thus, it suffices to show that for every $\varepsilon>0$ there exist $S,T>0$ such that
	\begin{align*}
		\Pw(\bfC_{s}^{C,B}\neq \emptyset,\widehat{\bfC}_{t}^{C',B,t+s}\neq \emptyset)\leq  \theta(C,B)\theta^{\pi}(C')+\varepsilon
	\end{align*}
	for all $s>S$ and $t>T$. We denote the extinction time of the infection process $\bfC$ by
	\begin{align*}
		\tau_{ex}=\tau_{ex}(C,B):=\inf\{t>0:\bfC^{C,B}_t=\emptyset\}.
	\end{align*}
	First we observe that for $C'\subset V$ finite that $\Pw^{(C',\pi)}(\tau_{ex}>t)\to \theta^{\pi}(C')$ as $t\to \infty$. Thus, for every $\varepsilon>0$ there exists a $T>0$ such that $|\Pw^{(C',\pi)}(\tau_{ex}>t)- \theta^{\pi}(C')|<\varepsilon$
	for all $t>T$. So we fix $t$ such that this is satisfied. Note that 
	\begin{align*}
		\Pw^{(C,B)}(u< \tau_{ex}<\infty)=1-\Pw^{(C,B)}(\tau_{ex} \leq  u)-\Pw^{(C,B)}(\tau_{ex}=\infty),
	\end{align*}
	and thus it follows that $\lim_{u\to \infty}\Pw^{(C,B)}(u<\tau_{ex}<\infty)=0$. Now we can use that $\{\bfC_{s}^{C,B}\neq \emptyset\}=\{\tau_{ex}>s\}$ to see that for every $\varepsilon>0$ there exists an $S_1>0$ such that
	\begin{align}\label{DifferenceLateSurvival}
		|\Pw(\tau_{ex}>s/2,\widehat{\bfC}_{t}^{C',B,t+s}\neq \emptyset)-\Pw(\bfC_{s}^{C,B}\neq \emptyset,\widehat{\bfC}_{t}^{C',B,t+s}\neq \emptyset)|\leq \Pw(s/2<\tau_{ex}<\infty)<\varepsilon,
	\end{align}
	for all $s>S_1$, which implies that
	\begin{align*}
		\Pw(\bfC_{s}^{C,B}\neq \emptyset,\widehat{\bfC}_{t}^{C',B,t+s}\neq \emptyset)\leq\Pw(\tau_{ex}>s/2,\widehat{\bfC}_{t}^{C',B,t+s}\neq \emptyset)+\varepsilon.
	\end{align*}
	Applying Lemma \ref{ControlDualProcess} yields that for given $\varepsilon>0$ there exists $S_2>0$ such that
	\begin{align*}
		\Pw(\widehat{\bfC}_{u}^{C',B,t+s}=\widecheck{\bfC}_{u}^{C',s/2,t+s}\,\, \forall u\leq t)>1-\varepsilon
	\end{align*}
	for all $s>S_2$, and thus for $s>\max(S_1,S_2)$
	\begin{align*}
		\Pw(\bfC_{s}^{C,B}\neq \emptyset,\widehat{\bfC}_{t}^{C',B,t+s}\neq \emptyset)\leq\Pw(\tau_{ex}>s/2,\widecheck{\bfC}_{t}^{C',s/2,t+s}\neq \emptyset)+2\varepsilon.
	\end{align*}
	Furthermore, we know that for every $\varepsilon>0$ there exists an $S_3>0$ such that
	\begin{align*}
		|\Pw^{(C,B)}(\tau_{ex}>s/2)-\theta(C,B)|<\varepsilon
	\end{align*}
	for all $s>S_3$. Note that by construction $(\bfC_{r}^{C,B})_{r\leq s/2}$ and $(\widecheck{\bfC}_{r}^{C',s/2,t+s})_{r\leq t+s/2}$ are independent, and thus
	\begin{align*}
		\Pw(\tau_{ex}>s/2,\widecheck{\bfC}_{t}^{C',s/2,t+s}\neq \emptyset)&=\Pw^{(C,B)}(\tau_{ex}>s/2)\Pw(\widecheck{\bfC}_{t}^{C',s/2,t+s}\neq \emptyset)\\
		&=\Pw^{(C,B)}(\tau_{ex}>s/2)\Pw^{(C',\pi)}(\tau_{ex}>t),
	\end{align*}
	where we used in the second equality that $\big(\widecheck{\bfC}_u^{C',s/2,t+s},\widecheck{\bfB}_{t+s-u}^{s/2}\big)_{u\leq t+s/2}$ is again a CPERE with intial distribution $\delta_{C'}\otimes\pi$. Set $\kappa:=4\varepsilon+\varepsilon^2$. We obtain at last that for any $t>T$ and  $s>S:=\max(S_1,S_2,S_3)$ (note that $S$ depends on $T$) we have
	\begin{align*}
		\Pw(\bfC_{s}^{C,B}\neq \emptyset,\widehat{\bfC}_{t}^{C',B,t+s}\neq \emptyset)&\leq \Pw^{(C,B)}(\tau_{ex}>s/2)\Pw^{(C',\pi)}(\tau_{ex}>t)+2\varepsilon\\
		&\leq \theta(C,B)\theta^{\pi}(C')+\kappa,
	\end{align*}
	which proves the claim.
\end{proof}
The next step is to prove a lower bound. For that we need the following stopping time
\begin{align}\label{StoppingTimeInfection}
	\tau_{A,H}(C,B):=\inf\{t\geq0:(\bfC^{C,B}_t,\bfB^{B}_t)\supset (A,H)\},
\end{align}
which is the first time that at least all vertices in $A$ are infected and all edges in $H$ are open.
\begin{lemma}\label{StoppingTimeRecurrence}
	Let $A,C\subset V$ and $H,B\subset E$ be non-empty and $A$ and $H$ be finite. Let $x\in C$ then
	\begin{equation*}
		\Pw^{(C,B)}(\tau_{A,H}<\infty)\geq \Pw^{(C,B)}(x\in \bfC_t\text{ i.o.}).
	\end{equation*}
\end{lemma}
\begin{proof}
	Suppose that $\pi\neq\delta_{\emptyset}$. Otherwise $\Pw^{(C,B)}(x\in \bfC_t\text{ i.o.})=0$, and thus the inequality is trivially true. First of all note that 
	\begin{align}\label{HelpStoppingTime1}
		\{\tau_{A,H}(C,B)<\infty\}=\{(\bfC^{C,B}_t,\bfB^{B}_t)\supset (A,H) \text{ for some } t \geq 0\}.
	\end{align}
	Next we define the stopping times $T_k=\inf \{t>T_{k-1}+1: x\in \bfC_t \}$,	where $T_0=0$. Recall that $\cF_t$ is the $\sigma$-algebra generated from all Poisson point processes used in the graphical representation until time $t$. Let us assume that $x\in C$, since $\pi\neq\delta_{\emptyset}$ and we know that the background process is translation invariant, we can guarantee that $\varepsilon=\Pw^{(\{x\},\emptyset)}( \bfC_1 \supseteq A,\bfB_1\supseteq H)>0$. This implies by monotonicity
	\begin{align}\label{HelpStoppingTime2}
		\Pw( \bfC_{T_k+1}\supseteq A,\bfB_{T_k+1}\supseteq H |\cF_{T_k})\geq \varepsilon\,\, \text{ a.s. on }\,\, \{T_k<\infty\}.
	\end{align}
	Set $J_{k}:=\{\bfC_{T_k+1}\supseteq A,\bfB_{T_k+1}\supseteq H\}\in \cF_{T_{k}+1}\subset \cF_{T_{k+1}}:=\cG_{k+1}$  and then we see that
	\begin{align*}
		\sum_{k=0}^{\infty}\Pw(J_{k}|\cG_{k})=\infty \quad \text{ a.s. on } \quad \bigcap_{k=0}^{\infty} \{T_k<\infty\}.
	\end{align*}
	Since $J_{k}\in\cG_{k+1}$ for all $k\in\N$ we can analogously as in Lemma~\ref{SurvivalContinuityLemma1} apply the extension of the Borel-Cantelli Lemma \cite[Theorem 4.3.4]{durrett2019probability} and we get that
	\begin{align*}
		\Big\{\sum_{k=0}^{\infty}\Pw(J_{k}|\cG_{k})=\infty\Big\}=\{ J_{k} \text{ i.o.}\}.
	\end{align*}
	Note that $\bigcap_{k=1}^{\infty}\{T_k<\infty\}=\{x\in \bfC_t \text{ i.o.}\}$. Hence, by \eqref{HelpStoppingTime1} and \eqref{HelpStoppingTime2} we get that 
	\begin{align*}
		\Pw^{(C,B)}(\tau_{A,H}<\infty)\geq\Pw^{(C,B)}(\{J_k \text{ i.o.}\}\cap \{x\in \bfC_t \text{ i.o.}\})=\Pw^{(C,B)}(x\in \bfC_t \text{ i.o.}).&\qedhere
	\end{align*}
\end{proof}
\begin{proposition}\label{CompleteConvergenceFirstMarginal}
	Let $C\subset V$ and $B\subset E$. Suppose \eqref{ConvergenceCond1} and \eqref{ConvergenceCond2} are satisfied, then
	\begin{equation*}
		\liminf_{t \to \infty}\Pw_{\lambda,r}(\bfC_{t}^{C,B}\cap C'\neq \emptyset)\geq \theta(\lambda,r,C,B)\theta^{\pi}(\lambda,r,C')
	\end{equation*}
	for every $C'\subset V$ finite.
\end{proposition}
\begin{proof}
	First assume that for $\lambda$ the inequality $c_1(\lambda,\rho)>\kappa^{-1}\rho$ is satisfied. Let $A\subset V$ and $H\subset E$ both be finite sets. We can assume that $\pi\neq\delta_{\emptyset}$, since if $\pi=\delta_{\emptyset}$, then $\theta^{\pi}(C')=0$ for all $C'\subset V$ finite, and thus the right hand side is zero. Recall from \eqref{StoppingTimeInfection} that the first time that at least all vertices in $A$ are infected and all edges in $H$ are open is denoted by $\tau_{A,H}(C,B)$. Furthermore, set $\sigma^N_A:=\tau_{A,A^N_E}$ and $\tau_A:=\tau_{A,\emptyset}$. Now we see that
	
	\begin{align}\label{RestartArgument}
		\begin{aligned}
			\Pw^{(C,B)}(\bfC_{t+s+u}\cap C'\neq \emptyset)&\geq \Pw^{(C,B)}(\sigma^N_A<s,\bfC_{t+u+s}\cap C'\neq \emptyset)\\
			&= \E^{(C,B)}[\1_{\{\sigma^N_A<s\}}\Pw(\bfC_{t+u+s}\cap C'\neq \emptyset|\cF_{\sigma^N_A})]\\
			&= \E\Big[\1_{\{\sigma^N_A<s\}}\Pw^{(\bfC_{\sigma^N_A},\bfB_{\sigma^N_A})}(\bfC_{t+u+(s-\sigma^N_A)}\cap C'\neq \emptyset\Big]\\
			&\geq \Pw^{(C,B)}(\sigma^N_A<s)\inf_{r>t+u}\Pw^{(A,A^{N}_{E})}(\bfC_{r}\cap C'\neq \emptyset),
		\end{aligned}
	\end{align}
	where we used that $(\bfC,\bfB)$ is a strong Markov process. One major issue is that in comparison to the classical case our duality is weaker in the sense that $\widehat{\bfC}^{C',A^N_{E},t+u+r}$ is not a CPERE, and therefore our process is not self dual. But now we show that the difference is not big if we choose $t+u$ large enough. Recall that $\Phi_t$ was the set of all vertices $x$ such that all edges attached to $x$ are contained in the permanently coupled region at time $t$. By Proposition \ref{InfectionVSBackground} we know that
	\begin{align*}
		&\Pw(\exists s>0: \widetilde{\bfC}^A_t\subset \Phi_t \,\, \forall t\geq s)=1,
	\end{align*}
	and thus for every $\varepsilon'>0$ there exists an $S>0$ such that
	\begin{align*}
		&\Pw(\widetilde{\bfC}^A_t\subset \Phi_t \,\, \forall t\geq S)>1-\varepsilon'.
	\end{align*}
	As an application of Lemma \ref{MaximalSpread} we find an $N=N(S)\in\N$ such that
	\begin{align*}
		&\Pw(\widetilde{\bfC}^A_t\subset A^N\,\, \forall t\leq S)>1-\varepsilon'.
	\end{align*}
	Furthermore, by Lemma~\ref{BounOnDecouplingLemma} there exists an $M>N$ such that
	\begin{align*}
	\Pw(\underbrace{\bfB^{A^M_{E}}_t=\bfB^{A^M_{E}\cup B}_t \text{ on } A^N_{E}\text{ for all } t\leq S}_{:=E_{S}(B,N,M)})>1-\varepsilon'
	\end{align*}
	where $B\subset E$ is chosen arbitrarily and $\varepsilon'$ is independent of the choice of $B$. Thus, we can conclude for a given $A\subset V$ that for every $\varepsilon>0$ there exists an $S=S(\varepsilon)>0$, $N=N(S)\in \N$ and $M>N$ such that
	\begin{align*}
		\Pw(\{\widetilde{\bfC}^A_t\subset \Phi_t \,\, \forall t\geq S,\widetilde{\bfC}^A_t\subset A^{N}\,\, \forall t\leq S\}\cap E_{S}(B,N,M))>1-\varepsilon
	\end{align*}
	for all $B\subset E$. Note that $\varepsilon$ depends on $A$. On this event the process $\bfC^{A,A^M_{E}}$ does not differ from $\bfC^{A,A^M_{E}\cup B}$ for any $B\subset E$, since on this event the infection paths have either not yet left $A^N$ and the edges in $A^N_{E}$ will have the same state open or closed with the two chosen initial configuration or the infection paths stay in $\Phi$, the area where every edge attached to an infected vertex has already been coupled. Thus, we get
	\begin{align}\label{TechnicalCompleteConvergensEquation}
		\Big|\Pw^{(A,A^M_{E})}(\bfC_{r}\cap C'\neq \emptyset)-\int\Pw^{(A,A^M_{E}\cup B)}(\bfC_{r}\cap C'\neq \emptyset)\pi(\mathsf{d}B)\Big|<\varepsilon.
	\end{align}
	Furthermore, by monotonicity (see Lemma \ref{MonotonicityAdditivityLemma}) it follows that
	\begin{align*}
		\Pw^{(A,A^{M}_{E})}(\bfC_{r}\cap C'\neq \emptyset)>\Pw^{(A,\pi)}(\bfC_{r}\cap C'\neq \emptyset)-\varepsilon.
	\end{align*}
	Using this and the fact that if the background is started stationary the CPERE is self dual by Proposition \ref{DistributionalDuality} we get with \eqref{RestartArgument} that
	\begin{align*}
		\Pw^{(C,B)}(\bfC_{t+s+u}\cap C'\neq \emptyset)&\geq \Pw^{(C,B)}(\sigma^N_A<s)\inf_{r>t+u}\big(\Pw^{(C',\pi)}(A\cap \bfC_{r}\neq \emptyset)-\varepsilon\big).
	\end{align*}
	Then, analogously to \eqref{RestartArgument} by considering $\tau_D$ with $D\subset V$ finite instead of $\sigma_A^{N}$ we can find a similar lower bound for the last probability such that
	\begin{align*}
		\Pw^{(C,B)}(\bfC_{t+s+u}\cap C'\neq \emptyset)&\geq \Pw^{(C,B)}(\sigma^N_A<s)\Pw^{(C',\pi)}(\tau_D<t)\inf_{r>u}\Pw^{(D, \emptyset) }(A\cap \bfC_{r}\neq \emptyset)-\varepsilon.
	\end{align*}
	For $A\subset V$ and $B\subset E$ finite we know by Condtion~\eqref{ConvergenceCond1} and Lemma~\ref{StoppingTimeRecurrence} that
	\begin{align*}
		\Pw^{(C,B)}(\tau_{A,H}<\infty)\geq\theta(C,B),
	\end{align*}
	and thus by letting $s,t,u\to \infty$ we see that
	\begin{align*}
		\liminf_{t\to\infty} \Pw^{(C,B)}(\bfC_{t}\cap C'\neq \emptyset)&\geq \theta(C,B)\theta^{\pi}(C')\liminf_{t\to \infty}\Pw^{(D, \emptyset) }(A\cap \bfC_{t}\neq \emptyset)-\varepsilon.
	\end{align*}
	Now for an arbitrary $x\in V$ we choose $A=D=\B_n(x)$ and use \eqref{ConvergenceCond2} which means that for all $\delta>0$ there exists $n_0\in \N$ such that $\liminf_{t\to \infty}\Pw^{(\B_n(x), \emptyset) }(\B_n(x)\cap \bfC_{t}\neq \emptyset)>1-\delta$ for all $n>n_0$. Note that $\varepsilon$ depends on $\B_n(x)$, which means we first need to choose $n_0$ and then the parameter accordingly such that \eqref{TechnicalCompleteConvergensEquation} holds for  $\varepsilon=\delta$ and such that
	\begin{align*}
		\liminf_{t\to\infty} \Pw^{(C,B)}(\bfC_{t}\cap C'\neq \emptyset)&\geq \theta(C,B)\theta^{\pi}(C')-2\delta.
	\end{align*}
	Since this holds for all $\delta>0$, the claim follows.

	%xxxx New proof part
	Now assume that $c_1(\lambda,\rho)\leq\kappa^{-1}\rho$ holds but there exists a $\lambda'\leq \lambda$ such that $c_1(\lambda',\rho)>\kappa^{-1}\rho$ and for $\lambda'$ and $r$ the conditions \eqref{ConvergenceCond1} and \eqref{ConvergenceCond2} are satisfied. Then we just showed that
	\begin{equation*}
	    \liminf_{t\to\infty} \Pw_{\lambda',r}^{(C,B)}(\bfC_{t}\cap C'\neq \emptyset)\geq \theta(\lambda',r,C,B)\theta^{\pi}(\lambda',r,C').
	\end{equation*}
	Since \eqref{ConvergenceCond2} is satisfied this implies in particular a positive survival probability. Thus, $(\lambda,r)\in \cS_{c_1}$ because $c_1(\lambda',\rho)>\kappa^{-1}\rho$ and $\theta(\lambda',r,\{x\},\emptyset)>0$ for any $x\in V$. Since $\lambda\geq \lambda'$ by monotonicity it follows that $\Pw_{\lambda,r}^{(C,B)}(\bfC_{t}\cap C'\neq \emptyset)\geq \Pw_{\lambda',r}^{(C,B)}(\bfC_{t}\cap C'\neq \emptyset)$, and therefore
	\begin{equation*}
	    \liminf_{t\to\infty} \Pw_{\lambda,r}^{(C,B)}(\bfC_{t}\cap C'\neq \emptyset)\geq \theta(\lambda',r,C,B)\theta^{\pi}(\lambda',r,C').
	\end{equation*}
	Furthermore by continuity of the survival probability on $\inte(\cS_{c_1})$ shown in Theorem~\ref{ContinuityTheorem} we get that
	\begin{equation*}
	    \liminf_{t\to\infty} \Pw_{\lambda,r}^{(C,B)}(\bfC_{t}\cap C'\neq \emptyset)\geq \theta(\lambda,r,C,B)\theta^{\pi}(\lambda,r,C').
	\end{equation*}
	by letting $\lambda'\uparrow \lambda$.
	\end{proof}
	We showed one direction of the equivalence. Next we show the converse direction.
	\begin{proposition}\label{CompleteConvergenceReverse}
		Let $(\lambda,r)\in\cS_{c_1}$. Suppose \eqref{ConvergenceCond3} holds and assume that $\overline{\nu}_{\lambda,r}\neq \delta_{\emptyset}\otimes \pi$, then \eqref{ConvergenceCond1} and \eqref{ConvergenceCond2} are satisfied.
	\end{proposition}
	\begin{proof}
	Note that $\overline{\nu}_{\lambda,r}\neq \delta_{\emptyset}\otimes \pi$ can only occur if $\pi\neq \delta_{\emptyset}$. Choose $C=C'=\B_n$ and $B=\emptyset$, then by \eqref{ConvergenceCond3} it follows that $\lim_{t\to \infty}\Pw^{(\B_n, \emptyset) }(\B_n\cap \bfC_{t}\neq \emptyset)=\theta(\B_n,\emptyset)\theta^{\pi}(\B_n)$.
	Using Lemma \ref{SurvivalContinuityLemma2} yields that the right hand side converges to $1$ as $n\to \infty$. This proves \eqref{ConvergenceCond2}. Now all what is left to show is \eqref{ConvergenceCond1}. We see that
	\begin{equation*}
		\{\bfC_t\cap C'\neq \emptyset \text{ i.o.}\}=\bigcap_{n\in\N}\{\bfC_s\cap C'\neq \emptyset \text{ for some }s\geq n\},
	\end{equation*}
	and thus by continuity of the law $\Pw$ we get that
	\begin{align*}
		\Pw^{(C,B)}(\bfC_t\cap C'\neq \emptyset \text{ i.o.})= \lim_{n\to \infty}\Pw^{(C,B)}(\bfC_s\cap C'\neq \emptyset \text{ for some }s\geq n)\geq \theta(C,B)\theta^{\pi}(C'),
	\end{align*}
	where we again used \eqref{ConvergenceCond3}. Now using the fact that $\Pw^{(\{y\},\emptyset)}(x\in \bfC_1)>0$ for all $x,y\in V$ it follows analogously as in the proof of Lemma~\ref{SurvivalContinuityLemma1}, that the event $\{x\in \bfC^{C,B}_t \text{ i.o.}\}$ a.s. happens on $\{\bfC^{C,B}_t\cap C'\neq \emptyset \text{ i.o.}\}$, and thus
	\begin{align*}
		\Pw^{(C,B)}(x\in \bfC_t \text{ i.o.})\geq \theta(C,B)\theta^{\pi}(C').
	\end{align*}
	Furthermore, if we choose $C'=\B_n$ and let $n\to \infty$, then Lemma~\ref{SurvivalContinuityLemma2} yields that
	\begin{align*}
		\Pw^{(C,B)}(x\in \bfC_t \text{ i.o.})\geq \theta(C,B)
	\end{align*}
	for all $x\in V$ and $C\subset V$. Since the reversed inequality ``$\leq$" obviously holds as well, this provides \eqref{ConvergenceCond1}.
\end{proof}
Since we have shown that the conditions \eqref{ConvergenceCond1} and \eqref{ConvergenceCond2} are equivalent to the fact that the two marginal processes converge, the only thing left to show is that convergence of the marginals already implies convergences of the joint distribution.
\begin{proposition}\label{CompleteConvergenceJoint}
	Suppose that \eqref{ConvergenceCond3} holds, then for all $C\subset V$ and $B\subset E$ it follows that
	\begin{equation}\label{ConvergenceCond4}
	\begin{aligned}
		\Pw^{(C,B)}_{ \lambda,r}(\bfC_{t}\cap A&\neq \emptyset,\bfB_{t}\cap H\neq \emptyset)\\
		\to&\,\,  \theta(C,B)\overline{\nu}(\{(C',B'): C'\cap A\neq \emptyset, B'\cap H\neq\emptyset \})
	\end{aligned}
	\end{equation}
	as $t\to \infty$, for every $A\subset V$ and $H\subset E$ finite.
\end{proposition}
\begin{proof}
	Let $A,C\subset V$ and $B,H\subset E$ be chosen arbitrary with $A\subset V$ and $H\subset E$ finite. We consider these sets as fixed. We again exploit the duality relation we derived in Proposition \ref{DistributionalDuality}, which states that
	\begin{align}\label{HelpDuality}
		\Pw(\bfC_{t+s}^{C,B}\cap A\neq \emptyset,\bfB^{B}_{t+s}\cap H\neq \emptyset)=\Pw(\bfC_{s}^{C,B}\cap \widehat{\bfC}_{t}^{A,B,t+s}\neq \emptyset,\bfB^{B}_{t+s}\cap H\neq \emptyset)
	\end{align}
	where $t,s>0$. Let $\tau_{ex}=\tau_{ex}(C,B)$ denote the extinction time with initial configuration $C\subset V$ and $B\subset E$.
	Some simple calculations yield that
	\begin{align}\label{HelpLengthyCalc}
		0\leq & \Pw(\bfC_{s}^{C,B}\neq \emptyset,\widehat{\bfC}_{t}^{A,B,t+s}\neq \emptyset,\bfB^B_{t+s}\cap H\neq \emptyset)-\Pw(\bfC_{s}^{C,B}\cap \widehat{\bfC}_{t}^{A,B,t+s}\neq \emptyset,\bfB^{B}_{t+s}\cap H\neq \emptyset)\notag\\
		=&\Pw(\bfC_{s}^{C,B}\neq \emptyset,\widehat{\bfC}_{t}^{A,B,t+s}\neq \emptyset,\bfC_{s}^{C,B}\cap \widehat{\bfC}_{t}^{A,B,t+s}= \emptyset,\bfB^B_{t+s}\cap H\neq \emptyset)\notag\\
		\leq &\Pw(\bfC_{s}^{C,B}\neq \emptyset,\widehat{\bfC}_{t}^{A,B,t+s}\neq \emptyset,\bfC_{s}^{C,B}\cap \widehat{\bfC}_{t}^{A,B,t+s}= \emptyset)\\
		=&\Pw(\bfC_{s}^{C,B}\neq \emptyset,\widehat{\bfC}_{t}^{A,B,t+s}\neq \emptyset)-\Pw(\bfC_{s}^{C,B}\cap \widehat{\bfC}_{t}^{A,B,t+s}\neq \emptyset)\notag.
	\end{align}
	Now we fix an arbitrary $\varepsilon>0$. Then by a combination of  Proposition~\ref{UpperBoundConvergence} and \eqref{ConvergenceCond3} we get that there exists a $S_1>0$ and $T>0$ such that 
	\begin{align*}
		\big|\Pw(\bfC_{s}^{C,B}\neq \emptyset,\widehat{\bfC}_{t}^{A,B,t+s}\neq \emptyset)-\Pw(\bfC_{s}^{C,B}\cap \widehat{\bfC}_{t}^{A,B,t+s}\neq \emptyset)\big|<\frac{\varepsilon}{3}
	\end{align*}
	for all $s>S_1$ and $t>T$. By using the duality relation \eqref{HelpDuality} together with \eqref{HelpLengthyCalc} we can conclude that
	\begin{align*}
		\big|\Pw(\bfC^{C,B}_{t+s}\cap A\neq \emptyset,\bfB^{B}_{t+s}\cap H\neq \emptyset)- \Pw(\bfC_{s}^{C,B}\neq \emptyset,\widehat{\bfC}_{t}^{A,B,t+s}\neq \emptyset,\bfB^B_{t+s}\cap H\neq \emptyset)\big|<\frac{\varepsilon}{3}
	\end{align*}
	for all $s>S_1$ and $t>T$. Furthermore, there exists an $S_2=S_2(C,B,\varepsilon)>0$ such that
	\begin{align*}
		|\Pw(\bfC_{s}^{C,B}\neq \emptyset,&\widehat{\bfC}_{t}^{A,B,t+s}\neq \emptyset,\bfB^B_{t+s}\cap H\neq \emptyset)\\
		&- \Pw(\tau_{ex}> s/2,\widehat{\bfC}_{t}^{A,B,t+s}\neq \emptyset,\bfB^B_{t+s}\cap H\neq \emptyset)|<\frac{\varepsilon}{3}
	\end{align*}
	for $s\geq S_2$, which can be shown analogously to \eqref{DifferenceLateSurvival}. In the last step we conclude that there exists an $S_3=S_3(t,A,H,\varepsilon)>0$ such that for $s\geq S_3$
	\begin{align*}
		\big|\Pw(\tau_{ex}> s/2,&\widehat{\bfC}_{t}^{A,B,t+s}\neq \emptyset,\bfB^B_{t+s}\cap H\neq \emptyset)\\
		&- \Pw(\tau_{ex}> s/2,\widecheck{\bfC}_{t}^{A,s/2,t+s}\neq \emptyset,\widecheck{\bfB}^{s/2}_{t+s/2}\cap H\neq \emptyset)\big|<\frac{\varepsilon}{3},
	\end{align*}
	which follows as a combination of Lemma~\ref{ControlBackground} and Lemma~\ref{ControlDualProcess}. Finally by putting everything together and using the triangle inequality we get that for every $t>T$ there exists an $S>0$ such that
	\begin{align*}
		\big|\Pw&(\bfC^{C,B}_{t+s}\cap A\neq \emptyset,\bfB^{B}_{t+s}\cap H\neq \emptyset)- \Pw(\tau_{ex}> s/2,\widecheck{\bfC}_{t}^{A,s/2,t+s}\neq \emptyset,\widecheck{\bfB}^{s/2}_{t+s/2}\cap H\neq \emptyset)\big|<\varepsilon
	\end{align*}
	for every $s>S$. To be precise one can choose $S=\max\{S_1,S_2,S_3\}$. Thus, it suffices to show that the second 
	 probability converges to the desired limit as we let $s\to \infty$ and then $t\to \infty$ in order to conclude this also for the first probability.
    Therefore, it suffices to show that
	\begin{align*}
		\Pw(\tau_{ex}> s/2,&\widecheck{\bfC}_{t}^{A,s/2,t+s}\neq \emptyset,\widecheck{\bfB}^{s/2}_{t+s/2}\cap H\neq \emptyset)\\
		\to &\,\, \theta(C,B)\overline{\nu}(\{(C,B): C\cap A\neq \emptyset, B\cap H\neq\emptyset \})
	\end{align*}
	as $s,t\to \infty$. Recall that we already concluded above that $(\bfC^{C,B}_r)_{r< s/2}$ is independent of $(\widecheck{\bfC}_{r}^{A,s/2,t+s})_{r\leq t+s/2}$ and it is also independent of $(\widecheck{\bfB}^{s/2}_{r})_{r\geq s/2}$. Thus, we get that
	\begin{align*}
		\Pw(\tau_{ex}> s/2,&\widecheck{\bfC}_{t}^{A,s/2,t+s}\neq \emptyset,\widecheck{\bfB}^{s/2}_{t+s/2}\cap H\neq \emptyset)\\
		=&\,\Pw(\tau_{ex}> s/2)\Pw(\widecheck{\bfC}_{t}^{A,s/2,t+s}\neq \emptyset,\widecheck{\bfB}^{s/2}_{t+s/2}\cap H\neq \emptyset).
	\end{align*}
	Next we use that $(\widecheck{\bfC}_{r}^{A,s/2,t+s},\widecheck{\bfB}^{s/2}_{t+s-r})_{r\leq t}$ is again a CPERE with initial distribution $\delta_{A}\otimes \pi$, and thus by duality
	\begin{align*}
		\Pw(\widecheck{\bfC}_{t}^{A,s/2,t+s}\neq \emptyset,\widecheck{\bfB}^{s/2}_{t+s/2}\cap H\neq \emptyset)=\Pw^{(V,\pi)}(\bfC_t\cap A\neq\emptyset,\bfB_t\cap H\neq \emptyset),
	\end{align*}
	which converges to the desired limit since we have already shown that $(\delta_{V}\otimes\pi)T(t)\Rightarrow\overline{\nu}$ as $t\to \infty$ in Lemma \ref{ConvergenceToUpperInv}. The claim follows since
	$\Pw^{(C,B)}(\tau> s/2)\to \theta(C,B)$
	as $s\to \infty$.
\end{proof}
Note that analogously as before \eqref{ConvergenceCond4} is equivalent to complete convergence, i.e for every initial configuration $C\subset V$ and $B\subset E$
\begin{align*}
	(\bfC_t^{C,B},\bfB_t^B)\Rightarrow \theta(C,B)\overline{\nu}+[1-\theta(C,B)](\delta_{\emptyset}\otimes\pi),
\end{align*}
since the function class $\{\1_{\{\,\cdot\,\cap A\neq \emptyset,\,\cdot \,\cap H\neq \emptyset\}}: C'\subset V, H\subset E \text{ finite} \}$ is convergence determining. Now we can conclude the main result of this chapter.
\begin{proof}[Proof of Theorem~\ref{CompleteConvergenceComplete}]
	The theorem follows as a combination of the four Propositions~ \ref{UpperBoundConvergence}, \ref{CompleteConvergenceFirstMarginal}, \ref{CompleteConvergenceReverse} and \ref{CompleteConvergenceJoint}. To be precise Propositions~\ref{UpperBoundConvergence} and \ref{CompleteConvergenceFirstMarginal} yield that \eqref{ConvergenceCond1} and \eqref{ConvergenceCond2} imply the convergence of the first marginal, i.e.~\eqref{ConvergenceCond3}. But in Proposition~\ref{CompleteConvergenceJoint} we already concluded that \eqref{ConvergenceCond3} suffices to conclude weak convergence of the CPERE, i.e.~\eqref{ConvergenceCond4}. At last Propostion~\ref{CompleteConvergenceReverse} provides equivalence of the conditions and complete convergence.
\end{proof}

\section{Contact process on a dynamical percolation (CPDP)}\label{Sec:CPDP}
In the previous sections we considered the CPERE in a fairly general setting. In this section we focus on the dynamical percolation as the background process, which
was introduced in Example~\ref{ThreeParticularSpinSystems} $(i)$. Recall that $\bfB$ is a Feller process with transitions
\begin{align*}
\bfB_{t-}=B&\to B\cup \{e\}	\quad \text{ at rate } \alpha>0,\\  
\bfB_{t-}=B&\to B\backslash \{e\}	\quad \text{ at rate } \beta>0.
\end{align*}
One reason to study this special case is that we can compare a CPERE $(\bfC,\bfB)$
with spin rates of the background given through $q(\cdot,\cdot)$ and a CPDP 
with adequately chosen rates for the background.
Recall from \eqref{MaximalAndMinimalRates1} that 
\begin{align*}
    \begin{aligned}
      \alpha_{\min}&:=\min_{F\subset \cN^L_e(R) }q(e,F),\qquad  \beta_{\min}:=\min_{F\subset\cN^L_e(R) }q(e,F\cup \{e\})\\
      \alpha_{\max}&:=\max_{F\subset \cN^L_e(R) }q(e,F),\qquad \beta_{\max}:=\max_{F\subset \cN^L_e(R) }q(e,F\cup \{e\}).
    \end{aligned}
\end{align*}
Proposition~\ref{ComparisonCPEREandCPDP} states that 
there exist two CPDP $(\overline{\bfC},\overline{\bfB})$ and $(\underline{\bfC},\underline{\bfB})$ with the same infection and recovery rates as $(\bfC,\bfB)$ such that $\overline{\bfB}$ and $\underline{\bfB}$ respectively have the rates $\alpha_{\max},\beta_{\min}$ and $\alpha_{\min}, \beta_{\max}$ and such that if $(\underline{\bfC}_0,\underline{\bfB}_0)= (\bfC_0,\bfB_0) = (\overline{\bfC}_0,\overline{\bfB}_0)$ then $\underline{\bfC}_t\subset \bfC_t\subset \overline{\bfC}_t$ and $\underline{\bfB}_t\subset \bfB_t\subset \overline{\bfB}_t$ for all $t>0$ a.s. 
\begin{proof}[Proof of Proposition~\ref{ComparisonCPEREandCPDP}]
Recall that in Section~\ref{GraphRepSpinSystemCose} we constructed a finite range spin system with range $R$ and with spin rate $q(\cdot,\cdot)$ via the maps $\mathbf{up}_{e,F}$ and $\mathbf{down}_{e,F}$. Furthermore, we chose the rates to be $\gamma_{\mathbf{up}_{e,F}}=q(e,F)$ and $\gamma_{\mathbf{down}_{e,F}}=q(e,F\cup \{e\})$ for any $x\in E$ and any $F\subset \cN^L_{e}(R)$. Then we obtained a spin system $\bfB$ with state space $\cP(E)$ and spin rate $q$. Note that by choice of the rates it is not difficult so see that for the rates defined in \eqref{MaximalAndMinimalRates1} 
\begin{equation}\label{MaximalAndMinimalRates}
\begin{aligned}
    \alpha_{\min}&=\min_{F\subset \cN_x(R) }\gamma_{\mathbf{up}_{x,F}},\qquad  \beta_{\min}=\min_{F\subset \cN_x(R) }\gamma_{\mathbf{down}_{x,F}},\\
    \alpha_{\max}&=\max_{F\subset \cN_x(R) }\gamma_{\mathbf{up}_{x,F}},\qquad  \beta_{\max}=\max_{F\subset \cN_x(R) }\gamma_{\mathbf{down}_{x,F}}.
\end{aligned}
\end{equation}
Now we adjust the construction as follows. We use the same maps but choose the rates $\underline{r}_{\mathbf{up}_{e,F}}=\alpha_{\min}$ and $\underline{r}_{\mathbf{down}_{e,F}}=\beta_{\max}$ for any $e\in E$ and any $F\subset \cN^L_{e}(R)$. Then it is not difficult to see that the resulting spin system $\underline{\bfB}$ has the spin rate 
\begin{equation*}
    \underline{q}(e,B)=\alpha_{\min}\1_{\{e\in B\}}+\beta_{\max}\1_{\{e\notin B\}},\quad  e\in E, B\subset E.
\end{equation*}
Thus, $\underline{\bfB}$ is a dynamical percolation with rates $\alpha_{\min},\beta_{\max}$. Analogously we obtain by choosing $\overline{r}_{\mathbf{up}_{e,F}}=\alpha_{\max}$ and $\overline{r}_{\mathbf{down}_{e,F}}=\beta_{\min}$ a dynamical percolation with rates $\alpha_{\max},\beta_{\min}$. Note that also
\begin{equation*}
    \begin{aligned}
        \underline{q}(e,B)&\leq q(e,B)\leq \overline{q}(e,B) \quad \text{ if } \quad e\notin B,\\
        \underline{q}(e,B)&\geq q(e,B)\geq \overline{q}(e,B) \quad \text{ if } \quad e\in B.
    \end{aligned}
\end{equation*}
Therefore, we can use \cite[Theorem~III.1.5]{liggett2012interacting} which implies that there exist a coupling such that if $\underline{\bfB}_0= \bfB_0= \overline{\bfB}_0$ then $\underline{\bfB}_t\subset \bfB_t\subset \overline{\bfB}_t$ for all $t>0$ a.s. Hence, by the graphical representation introduced in  Section~\ref{GraphRepresentationCPERE} it follows that if $(\underline{\bfC}_0,\underline{\bfB}_0)= (\bfC_0,\bfB_0) = (\overline{\bfC}_0,\overline{\bfB}_0)$, then $(\underline{\bfC}_t,\underline{\bfB}_t)\subset (\bfC_t,\bfB_t)\subset (\overline{\bfC}_t,\overline{\bfB}_t)$ a.s. for all $t>0$.
\end{proof}
Next we show Proposition~\ref{IndependenceGrowthCond} which states that  in the case of the CPDP  even if the chosen infection rate $\lambda$ does not satisfy the growth condition $c_1(\lambda,\rho)>\kappa^{-1}\rho$ 
the interior of the survival region does not depend on the initial configuration.
To be precise we will show a slightly stronger but more technical result which implies Proposition~\ref{IndependenceGrowthCond}. Let us define the parameter set
\begin{equation*}
	    \cS^{\pi}:=\{(\lambda,r,\alpha,\beta): \theta^{\pi}(\lambda,r,\alpha,\beta,\{x\})>0\}.
	\end{equation*}
Note that since the background is started stationary this set does not depend on the choice of $x\in V$. This follows by translation invariance since this yields that $\theta^{\pi}(\{x\})=\theta^{\pi}(\{y\})$ for all $x,y\in V$, and thus if there exists $x\in V$ such that $\theta^{\pi}(\{x\})>0$ then for any non-empty and finite $C\subset V$ it follows that $\theta^{\pi}(C)>0$.
\begin{proposition}\label{IndependenceGrowthCond2}
	Let 	$(\bfC,\bfB)$ be a CPDP. 
	\begin{enumerate}
	    \item[(i)] If there exists an $\varepsilon>0$ such that $(\lambda,r,\alpha+\varepsilon,\beta-\varepsilon)\in \cS^\pi$ then it follows that $(\lambda,r,\alpha,\beta)\in \cS(\{x\},\emptyset)$.
	    \item[(ii)] If there exists an $\varepsilon>0$ such that $(\lambda,r,\alpha+\varepsilon,\beta-\varepsilon)\in (0,\infty)^4\backslash \cS^\pi$ then it follows that $(\lambda,r,\alpha,\beta)\in (0,\infty)^4\backslash\cS(\{x\},E)$.
	\end{enumerate}
\end{proposition}
\begin{proof}
This proof is a modification of the proof of Proposition~\ref{IndependencySurvival}, where we use again a comparison of $\bfC$ with the two processes $\underline{\bfC}^{s}$ and $\overline{\bfC}^{s}$. Let $s>0$ be given, then recall that the lower bound $\underline{\bfC}^{s}$ ignores all transmission arrows in $[0,s]$ and the upper bound $\overline{\bfC}^{s}$ ignores all recovery symbols as well as the background since infection arrows are used in any case, whether or not the background is open or closed.

    Let us begin with showing (i). 
   Note that we omit the $\lambda$ and $r$ as variables in the following since they are constant throughout the whole proof. Now we fix an $\varepsilon>0$ such that $\theta(\alpha-\varepsilon,\beta+\varepsilon,\{x\})>0$ which exists by assumption.
    Furthermore, we see that 
	\begin{equation*}
		\Pw(e\in \bfB_t^{B})=\1_{\{e\in B\}}e^{-(\alpha+\beta)t}+ \frac{\alpha}{\alpha+\beta}(1-e^{-(\alpha+\beta)t}).
	\end{equation*}
	Thus, for $B=\emptyset$  and the time $s:=\frac{1}{\alpha+\beta}\log\big(\frac{\alpha}{\varepsilon}\big)$ it follows that
	\begin{equation*}
		\Pw(e\in \bfB_s^{\emptyset})=\frac{\alpha-\varepsilon}{\alpha+\beta}\quad\text{and}\quad\Pw(e\notin \bfB_s^{\emptyset})=\frac{\beta+\varepsilon}{\alpha+\beta}.
	\end{equation*}
	Since the events $\{e_1\in \bfB_s^{\emptyset}\}$ and $\{e_2\in \bfB_s^{\emptyset}\}$ are independent for any $e_1,e_2\in E$ we see that $\bfB_s^{\emptyset}\sim \pi_{\alpha-\varepsilon,\beta+\varepsilon}$.
	
	Recall from the proof of Proposition~\ref{IndependencySurvival} that $\underline{\underline{\bfC}}^s$ is a process which is constructed analogously to $\underline{\bfC}^s$ with the difference that on $[0,s]$ also no recovery symbols have an effect. Therefore, $\underline{\underline{\bfC}}^s$ is just a delayed CPERE. By construction it is clear that it is only possible for $\underline{\bfC}^{s,\{x\},B}$ to survive if until time $s$ the vertex $x$ is not hit by a recovery symbol, i.e.~let $T:=\inf\{t>0:(\mathbf{rec}_{x},t)\in\Xi^{\text{rec}}\}$, then $\underline{\bfC}^{s,\{x\},B}$ goes extinct a.s. on the event $\{T\leq s\}$. Note that $\underline{\underline{\bfC}}^{\{x\},B,s}=\underline{\bfC}^{\{x\},B,s}$ on $\{T> s\}$ and thus,
	\begin{align}\label{DPIndependencyInequality1}
		\begin{aligned}
			\Pw\big(\underline{\bfC}^{\{x\},\emptyset,s}_t\neq \emptyset \,\, \forall t\geq 0\big)=\Pw\big(\{\underline{\underline{\bfC}}^{\{x\},\emptyset,s}_t\neq \emptyset \,\, \forall t\geq s\}\cap \{T>s\}\big).
		\end{aligned}	
	\end{align}
	Furthermore we know that the event $\{T>s\}$ only depends on $\Xi^{\text{rec}}$ in the time interval $[0,s]$ while the survival of $\underline{\underline{\bfC}}$ is not influenced by recovery events in $[0,s]$. Thus,
	\begin{align}\label{DPIndependencyInequality2}
		\begin{aligned}
			\Pw\big(\{\underline{\underline{\bfC}}^{\{x\},\emptyset,s}_t\neq \emptyset \,\, \forall t\geq s\}\cap \{T>s\}\big)
			=&\Pw(T>s)	\Pw\big(\underline{\underline{\bfC}}^{\{x\},\emptyset,s}_t\neq \emptyset \,\, \forall t\geq s\big).
		\end{aligned}
	\end{align}
	By construction it follows that $(\underline{\underline{\bfC}}_t^{s})_{t\leq s}$ and $(\bfB^{B}_t)_{t\leq s}$ are independent.
	Now we can thin the open events and draw independently new closing events respectively with rate $\varepsilon$ on $(s,\infty)$. Since $\bfB_s\sim \pi_{\alpha-\varepsilon,\beta+\varepsilon}$ a coupling via the graphical representation yields
	\begin{equation}\label{DPIndependencyInequality3}
		\Pw_{\alpha,\beta}\big(\underline{\underline{\bfC}}^{\{x\},\emptyset,s}_t\neq \emptyset \,\, \forall t\geq s\big)\geq \theta^{\pi}(\alpha-\varepsilon,\beta+\varepsilon,\{x\}).
	\end{equation}
	Thus, we can use that $\underline{\bfC}^{\{x\},\emptyset,s}_t\subset \bfC^{\{x\},\emptyset}_t$ for all $t\geq 0$ and \eqref{DPIndependencyInequality1}-\eqref{DPIndependencyInequality3} to conclude that
	\begin{equation*}%\label{IndependencySurvivalEq1}
		\theta(\alpha,\beta,\{x\},\emptyset)\geq\Pw(T>s)\theta^{\pi}(\alpha-\varepsilon,\beta+\varepsilon,\{x\})>0
	\end{equation*}
	where the last inequality follows by assumption. This proves $(i)$.
	
	Next, we proceed with showing (ii). 
	Similarly to before we know that there exists an $\varepsilon>0$ such that $\theta^{\pi}(\alpha+\varepsilon,\beta-\varepsilon,\{x\})=0$.
	Analogously to  before we see that
	\begin{equation*}
		\Pw(e\notin \bfB_s^{B})=\1_{\{e\notin B\}}e^{-(\alpha+\beta)t}+ \frac{\beta}{\alpha+\beta}(1-e^{-(\alpha+\beta)t}).
	\end{equation*}
	Thus, fixing $\varepsilon$ and choosing $B=E$ we see for the time $s:=\frac{1}{\alpha+\beta}\log\big(\frac{\beta}{\varepsilon}\big)$ that
	\begin{equation*}
		\Pw(e\in \bfB_s^{E})=\frac{\alpha+\varepsilon}{\alpha+\beta}\quad\text{and}\quad\Pw(e\notin \bfB_s^{E})=\frac{\beta-\varepsilon}{\alpha+\beta}.
	\end{equation*}
	Since the events $\{e_1\in \bfB_s^{E}\}$ and $\{e_2\in \bfB_s^{E}\}$ are independent for any $e_1,e_2\in E$ it follows that $\bfB_s^{E}\sim \pi_{\alpha+\varepsilon,\beta-\varepsilon}$.
	By construction of $\overline{\bfC}$ we see that
	\begin{align*}
		\Pw(\bfC^{\{x\},E}_t\neq \emptyset \,\,\forall t\geq 0)\leq \Pw(\overline{\bfC}^{\{x\},E,s}_t\neq \emptyset \,\,\forall t\geq 0).
	\end{align*}
	We can again thin out the closed events and draw independently new open events with a rate $\varepsilon$. Since $\bfB^E_{s}\sim \pi_{\alpha+\varepsilon,\beta-\varepsilon}$ and $\widetilde{\bfC}^{\{x\}}_t=\overline{\bfC}^{\{x\},E,s}_t$ for all $t\in[0,s]$, where $\widetilde{\bfC}$ is the coupled contact process that ignores the background and recoveries, it follows that
	\begin{equation*}	
		\Pw_{\alpha,\beta}(\overline{\bfC}^{\{x\},E,s}_t\neq \emptyset \,\,\forall t\geq 0)
		\leq \E^{\{x\}}[\theta^{\pi}(\alpha+\varepsilon,\beta-\varepsilon,\widetilde{\bfC}_s)]=0,
	\end{equation*}
	where we used that by assumption $\theta^{\pi}(\alpha+\varepsilon,\beta-\varepsilon,C)=0$ for all finite $C$ and $|\widetilde{\bfC}^{\{x\}}_s|< \infty$ a.s.
\end{proof}
To prove Proposition~\ref{IndependenceGrowthCond} we still need the following topological statement.
 
\begin{lemma}\label{lem:TopologicalAuxLemma}
Let  $d\in \N$ and  $f:(0,\infty)^d\to [0,1]$ be a function, which is in every coordinate either monontone increasing or decreasing for arbitrary fixed values of the other coordinates. For the set $A:=\{x: f(x)>0\}$ it holds that $\inte(A)=\inte(\overline{A})$ and $\overline{\inte(A)}=\overline{A}$.
\end{lemma} 

\begin{proof}
    Without loss of generality we can assume that $f$ is monotone increasing in every coordinate since we can map the coordinate $x_i\mapsto \frac{1}{x_i}$ to transform decreasing into increasing. Note that this changes the parametrization of the set $A\subset \R^d$ but not its topological properties. As usual we denote by $\leq$ the order on $(0,\infty)$ and also the componentwise order on $(0,\infty)^d$.

    We first show that $\inte(A)=\inte(\overline{A})$. Clearly, we have $\inte(A)\subset \inte(\overline{A})$ and so it suffices to show that $x\notin \inte(A)$ implies $x\notin \inte(\overline{A})$. Since $x \notin \inte(A)$ we have for every $z\in(0,\infty)^d$ with $z_i< x_i$ for all $i\leq d$  that $f(z)=0$. Because if we had $f(z)>0$, it would follow that $f(y)>0$ for every $y\geq z$ by monotonicity. This would imply that there exists a $\delta>0$ 
    such that $f(y)>0$ for all $y\in B_{\delta}(x)$ 
    and so $x\in \inte(A)$, which is a contradiction.
    On the other hand we can conclude that there exists a $\delta'>0$ such that $B_{\delta'}(z)\subset A^{c}$ and thus $z \notin \overline{A}$. Furthermore, by this argument we can see that $B_{\epsilon}(x)\cap \big(\overline{A}\big)^{c}\neq \emptyset$ for every $\varepsilon>0$, and thus $x\notin \inte(\overline{A})$.

    Next we show that $\overline{\inte(A)}=\overline{A}$. Again, one inclusion is obvious and for the other inclusion $\overline{A} \subset \overline{\inte(A)}$ it suffices to show $A \subset \overline{\inte(A)}$.
    For every $x\in A$ by definition it holds that $f(x)>0$. By monotonicity we have for all $z \geq x$ that $f(z)>0$. Furthermore, for every $z\in(0,\infty)^d$ with $z_i> x_i$ for all $i\leq d$ there exists a $\delta>0$ such that for every $y\in B_{\delta}(z)$ it holds that $y_i> x_i$ for all $i\leq d$. This implies that $B_{\delta}(z)\subset A$, and therefore we can conclude that $z\in \inte(A)$. But this means that for every $\varepsilon>0$ it must hold that $B_\varepsilon(x)\cap \inte(A)\neq \emptyset$. By definition of the closure it follows that $x\in \overline{\inte(A)}$.
\end{proof}
Now we have all preliminary results which we need to show Proposition~\ref{IndependenceGrowthCond}, i.e.\ $\inte(\cS(C,E))=\inte(\cS(\{x\},\emptyset))$ and $\overline{\cS(C,E)}=\overline{\cS(\{x\},\emptyset)}$ for any finite and non-empty  $C\subset V$.

\begin{proof}[Proof of Proposition~\ref{IndependenceGrowthCond}]
	By definition we know that $\cS(\{x\},\emptyset)\subset\cS^{\pi}$. Furthermore, if $(\lambda,r,\alpha,\beta)\in \inte(\cS^{\pi})$, then in particular there must exists an $\varepsilon>0$ such that $(\lambda,r,\alpha+\delta,\beta-\delta)\in \cS^{\pi}$ for all $\delta\in(-\varepsilon,\varepsilon)$. Thus, we know by Proposition~\ref{IndependenceGrowthCond2} $(i)$ that $(\lambda,r,\alpha,\beta)\in \cS(\{x\},\emptyset)$ must hold. But this implies that $\inte(\cS(\{x\},\emptyset))=\inte(\cS^{\pi})$.
	
	Similarly, we know that $\cS^{\pi}\subset\cS(\{x\},E)$, and therefore we also know that
	\begin{equation*}
	    (0,\infty)^4\backslash\overline{\cS(\{x\},E)}\subset (0,\infty)^4\backslash\overline{\cS^{\pi}}.
	\end{equation*}
	Both of these sets are open, and thus if $(\lambda,r,\alpha,\beta)\in (0,\infty)^4\backslash\overline{\cS^{\pi}}$, then there must exist an $\varepsilon>0$ such that $(\lambda,r,\alpha+\delta,\beta-\delta)\in (0,\infty)^4\backslash\overline{\cS^{\pi}}$ for all $\delta\in(-\varepsilon,\varepsilon)$. Hence, by Proposition~\ref{IndependenceGrowthCond2} $(ii)$ it follows that $(\lambda,r,\alpha,\beta)\in (0,\infty)^4\backslash\overline{\cS(\{x\},E)}$ which then yields
	\begin{equation*}
	    (0,\infty)^4\backslash\overline{\cS(\{x\},E)}=(0,\infty)^4\backslash\overline{\cS^{\pi}},
	\end{equation*}
	and thus $\overline{\cS^{\pi}}=\overline{\cS(\{x\},E)}$.
	
	Now, for any $C$ finite and non-empty and any $B \subset E$ the function $(\lambda,r,\alpha,\beta) \mapsto \theta\big(\lambda,r,\alpha,\beta,C,B\big)$ satisfies the conditions of  Lemma~\ref{lem:TopologicalAuxLemma}  and since
	\begin{equation*}%\label{SurvivalRegionCPDP}
        \cS(C,B):=\{(\lambda,r,\alpha,\beta)\in (0,\infty)^4:\theta\big(\lambda,r,\alpha,\beta,C,B\big)>0\}
    \end{equation*}
   this then yields that 
      $\inte(\cS(C,B))=\inte(\overline{\cS(C,B)})$ and $\overline{\cS(C,B)}=\overline{\inte(\cS(C,B))}$. 
      Similarly, we obtain  $\inte(\cS^{\pi})=\inte(\overline{\cS^{\pi}})$ and $\overline{\cS^{\pi}}=\overline{\inte(\cS^{\pi})}$.
      Taken together we obtain from $\inte(\cS(\{x\},\emptyset))=\inte(\cS^{\pi})$ also
      $\overline{\cS(\{x\},\emptyset)}=\overline{\cS^{\pi}}$ and from $\overline{\cS^{\pi}}=\overline{\cS(\{x\},E)}$ also $\inte(\cS^{\pi})=\inte(\cS(\{x\},E))$.
      Finally, since we have  for any $C$ finite and non-empty and any $B \subset E$ that 
      \begin{equation*}
      \cS(\{x\},\emptyset)= \cS(C,\emptyset)\subset \cS(C,B) \subset \cS(C,E) = \cS(\{x\},E)
      \end{equation*}
      and the interior and the closure of the left and right hand side agree (and are equal to 
      $\inte(\cS^{\pi})$ and $\overline{\cS^{\pi}}$ respectively) and the claim follows.
\end{proof}

\subsection{CPDP on the $d$-dimensional integer lattice}\label{InvolvingBlockContruction}
In this subsection we focus solely on the CPDP on the $d$-dimensional lattice with nearest neighbour structure, i.e.~$V=\Z^d$ and $E=\{\{x,y\}\subset \Z^d: ||x-y||_1=1\}$, where $||\cdot||_1$ denotes the $1$-norm.

Our next goal is to adapt a block construction which was initially developed by Bezuidenhout and Grimmett \cite{bezuidenhout1990critical} for the classical contact process. This construction will allow us to show that the CPDP dies out at criticality and that complete convergence is satisfied. In essence, the construction allows us to couple the CPDP with a oriented percolation on a macroscopic grid in such a way that the percolation model survives if and only if the infection process of the CPDP survives.

The block construction of \cite{bezuidenhout1990critical} is quite sophisticated and technical. In order to adapt the original version to our case only minor changes are needed. One of the main differences is the additional background process $\bfB$. But it is not difficult to incorporate this additional feature. As in the original construction we just need to be careful to restart the process in an appropriate configuration at the beginning of a block. Thus, we only give a broad description and sketch the proofs of the major results. For a more detailed description of the procedure for the CPDP based on exposition in \cite[Part I.2]{liggett2013stochastic} we refer to \cite{seiler2021}. 
We should also mention that we are not the first to adapt these techniques to a variation of the contact process. 
In a setting similar to ours this has already been done, for example by Reminik \cite{remenik2008contact} and Steif and Warfheimer \cite{steif2007critical} for a contact process with varying recovery rates and by Deshayes \cite{deshayes2014contact} for a contact process with ageing.

The idea of the construction
is to formulate finite space-time conditions which are equivalent to survival of the process. Thus, for an arbitrary but fixed $L\in\N$ we first introduce a truncated version $(\prescript{}{L}\bfC,\prescript{}{L}\bfB)$ of the CPDP on a finite space-time box 
\begin{align*}
	V_L:=[-L,L]^d\cap \Z^d \text{ and } E_L:= \{e: e\cap V_L\in E\}.
\end{align*}
This process can again be defined via a graphical representation by only considering recovery symbols on vertices $x\in V_L$ 
and infection events which emanate from a vertex $x\in(-L,L)^d\cap \Z^d$. Therefore, we also only need to consider flip events influencing edges in $E_L$. For ease of notation we 
 abbreviate $[-n,n]^d\cap \Z^d$ by $[-n,n]^d$ and in particular write $\bfC^{[-n,n]^d,B}_t$ for $\bfC^{[-n,n]^d\cap \Z^d,B}_t$  in the following.

Now we are ready to formulate the above mentioned conditions on the finite space-time box $[-L,L]^d\times[0,T+1]$, where $T>0$. For that we need to consider the events
\begin{align}
	\cA_1=\cA_1(n,L,T):=\big\{\prescript{}{L+n}\bfC^{[-n,n]^{d},\emptyset}_{T+1}\supset & x+[-n,n]^{d} \text{ for some } x\in[0,L)^d \label{SpaceTimeCondition1}\big\},\\
	\cA_2=\cA_2(n,L,T):=\big\{\prescript{}{L+2n}\bfC^{[-n,n]^{d},\emptyset}_{t+1}\supset & x+[-n,n]^{d} \text{ for some } 0\leq t < T\notag\\
	&\phantom{AAAAAAA}\text{ and } x\in\{L+n\}\times[0,L)^{d-1}\big\}.\label{SpaceTimeCondition2}
\end{align}
In words, the event $\cA_1$ states that if we start the truncated CPDP in the initial configuration $([-n,n]^{d},\emptyset)$ we find a spatially shifted version of the box $[-n,n]^{d}$ at the top of a bigger space-time box $[-(L+n),L+n]^d\times[0,T+1]$. On the other hand, the event $\cA_2$ states that we  find a spatially shifted version of the box $[-n,n]^{d}$ at the "right" boundary (in direction of the first coordinate) of the bigger space-time box. In broad terms one could say that these events guarantee that throughout this big space-time box the infection survives at least as "strongly" as it started. We illustrate the cross section of these events (in the direction of the first coordinate axis) in \autoref{fig:ManySites}. The main difference to the original construction is that we also need to pay attention to the state of the background process. We consider it to be in the most unfavorable state from the perspective of the infection, which is the empty configuration $\emptyset$.
\begin{figure}[t]
	\centering
	\includegraphics[width=140mm]{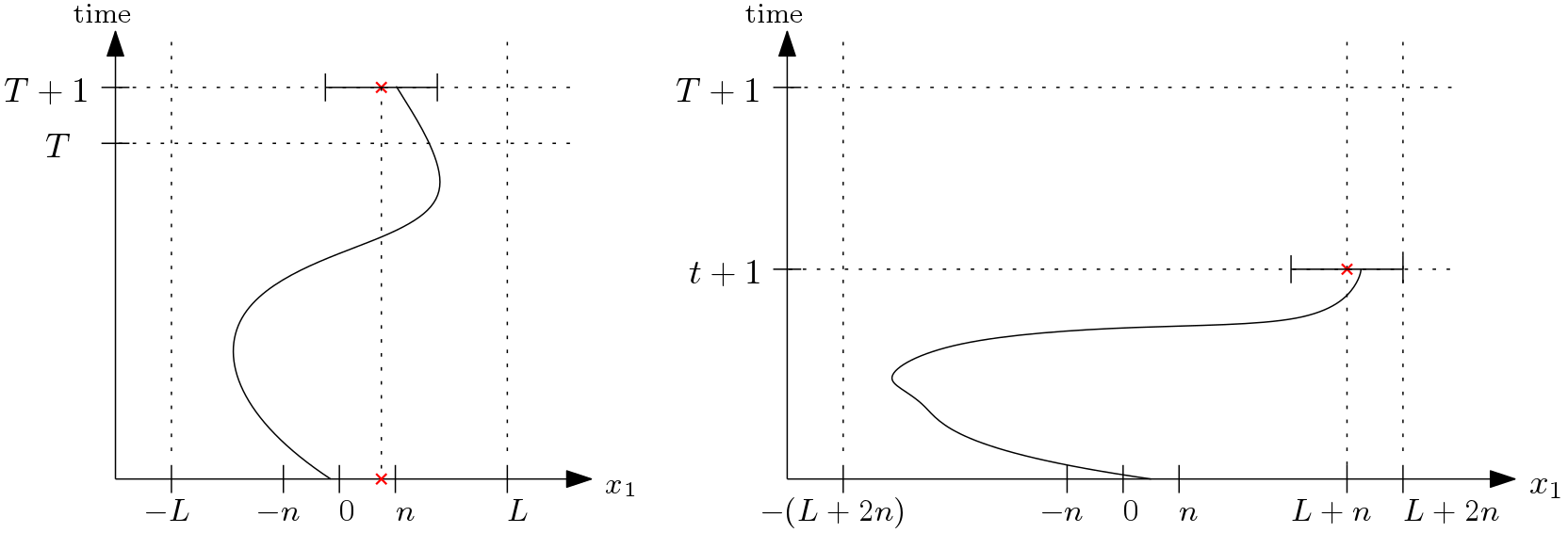}
	\caption{Illustration of the events in \eqref{SpaceTimeCondition1} and \eqref{SpaceTimeCondition2}}
	\label{fig:ManySites}
\end{figure}
\begin{remark}
    In case of the CPDP we know that the states of any pair of edges $e\neq e'\in E$ are independent of each other, and therefore it does not matter if we only consider the restricted background process $\prescript{}{L}\bfB$ or $\bfB$ to define the truncated infection process $\prescript{}{L}\bfC$.
    
    This is not the case if $\bfB$ is an arbitrary spin system. If we do not restrict the background $\bfB$, then shifts of the events $\cA_1$ and $\cA_2$ which are spatially disjoint with respect to the graphical representation are no longer independent of $\cA_1$ or $\cA_2$. But this is crucial for the comparison argument with an oriented percolation which we use later on. On the other hand if we want to first restrict the process to $\prescript{}{L}\bfB$ which is only defined on the box $[-L,L]^d$ we need to fix a boundary condition. Thus, the question is how to choose a good boundary condition? It is also not clear at all if a number of technical results used in the construction still hold true in this case.    
\end{remark}

The finite space-time condition which we impose states that we can choose the parameters $n,L$ and $T$ in such a way that these events happen with "high" probability.
\begin{condition}\label{SpaceTimeCondition}
	For all $\varepsilon>0$ there exist $n,L\geq 1$ and $T>0$ such that
	\begin{align*}
		\Pw(\cA_1)>1-\varepsilon \quad \text{ and } \quad \Pw(\cA_2)>1-\varepsilon.
	\end{align*}
\end{condition}
Now first one shows that if survival is possible then these conditions are satisfied.
For this we need some further notation. For $L\in \N$ and $T\geq 0$ we set
\begin{align*}
S(L,T)&:=\{(x,t)\in \Z^d\times [0,T]: ||x||_{\infty}=L\} \text{ and }\\
S_{+}(L,T)&:=\{(x,t)\in \Z^d\times [0,T]: x_1=L, x_i\geq 0 \text{ for } 2\leq i \leq d \}.
\end{align*}
The set $S(L,T)$ is the union of all lateral faces of the space-time box $[-L,L]^d\times [0,T]$ and $S_{+}(L,T)$ is the intersection of one particular lateral face with the first orthant.

Next we fix $C\subset (-L,L)^d\cap \Z^d$. We want to consider all points in $S(L,T)$ and respectively in $S_{+}(L,T)$, which can be reached from $C$ through an $\emptyset$-infection path, i.e.~an infection path which starts with the background in state $\emptyset$. Let us define by $N_{\emptyset}^{C}(L,T)$ the maximal number of points in any $D\subset S(L,T)\cap \prescript{}{L}\bfC^{C,\emptyset}$ and respectively by $N^C_{+,\emptyset}(L,T)$ the maximal number of points in any subset $D\subset S_{+}(L,T)\cap \prescript{}{L}\bfC^{C,\emptyset}$, where $D$ has the property that every two points $(x,t_1),(x, t_2)\in D$ with the same spatial coordinate  satisfy $|t_2-t_1|\geq 1$.
\begin{theorem}\label{CPDPSurvivalImplyCond}
	Suppose $\theta(\lambda,r,\alpha,\beta,\{\zero\},\emptyset)>0$, then Condition~\ref{SpaceTimeCondition} is satisfied.
\end{theorem}
\begin{proof} We only sketch this proof and refer for details to 
\cite[Theorem~6.1.7]{seiler2021}. The proof is basically split in two parts. First one shows that with high probability we have sufficiently many infected vertices at the top face and at lateral faces of the space-time box $S(L,T)$.
Lemma~\ref{SurvivalContinuityLemma2} is fundamental for this proof. This lemma states that for a given $0<\delta<1$ there exists an $n$ such that
\begin{equation*}%\label{LargeInitialSetHighSurvivalProbability}
	\Pw(\bfC^{[-n,n]^d,\emptyset}_t\neq\emptyset\,\,\forall t\geq 0)\geq 1-\delta^2.
\end{equation*}
In words this means that we can push the extinction probability below any given bound by choosing $n$ large enough. Next we use that for any $n,N\geq 1$
\begin{equation*}%\label{ChoiceOfL&T}
	\lim_{t\to \infty}\lim_{L\to \infty}\Pw(|\prescript{}{L}\bfC^{[-n,n]^d,\emptyset}_t|\geq 2^dN)=\Pw(\bfC^{[-n,n]^d,\emptyset}_t\neq \emptyset \,\, \forall t\geq 0)
\end{equation*}
(see \cite[Proposition~6.1.3]{seiler2021}), which means that if a contact process survives, then it does so with infinitely many vertices being infected. This fact can be used to construct two strictly increasing sequences $(T_k)_{k\geq 0}$ and $(L_k)_{k\geq 0}$ such that $T_k,L_k\uparrow \infty$ and
\begin{equation}\label{EqualityForTheTwoSequences1}
	\Pw(|\prescript{}{L_k}\bfC^{[-n,n]^d,\emptyset}_{T_k}|> 2^dN)=1-\delta
\end{equation}
for all $k\geq 0$ and with \cite[Lemma~6.1.5]{seiler2021} one can conclude that for some $M \geq 1$ also
\begin{equation}\label{EqualityForTheTwoSequences2}		
	\Pw\big(N^{[-n,n]^d}_{\emptyset}(L_k,T_k)\leq  Md2^d\big)< 2\delta
\end{equation}
for all $k\geq 0$. Note that we chose the constant $2^dN$ since by positive correlation one can conclude that
\begin{align*}
	\Pw\big(|\prescript{}{L}\bfC^{[-n,n]^d,\emptyset}_T\cap [0,L)^d|\leq N\big)&\leq \big(\Pw\big(|\prescript{}{L}\bfC^{[-n,n]^d,\emptyset}_T|\leq 2^dN)\big)^{2^{-d}},\\
		\Pw\big(|N^{[-n,n]^d}_{+,\emptyset}(L,T)|\leq M\big)&\leq \big(\Pw\big(|N^{[-n,n]^d}_{\emptyset}(L,T)|\leq Md2^d\big)\big)^{d^{-1}2^{-d}}
\end{align*}
(see \cite[Proposition~6.1.4 \& 6.1.6]{seiler2021}). This can be used together with \eqref{EqualityForTheTwoSequences1} and \eqref{EqualityForTheTwoSequences2} to conclude that
	\begin{align}
	&\Pw(|\prescript{}{L}\bfC^{[-n,n]^d,\emptyset }_T\cap [0,L)^d|> N)\geq 1-\big(\Pw(|\prescript{}{L}\bfC^{[-n,n]^d,\emptyset}_T|\leq 2^dN)\big)^{2^{-d}}=1-\delta^{2^{-d}}, \label{CrucialLowerBounds1}\\
	&\Pw(|N^{[-n,n]^d}_{+,\emptyset}\!(L,T)|> M)\geq 1-\big(\Pw(|N^{[-n,n]^d}_{\emptyset}\!(L,T)|\leq Md2^d)\big)^{\frac{1}{d2^{d}}}>1-( 2\delta)^{\frac{1}{d2^{d}}}\label{CrucialLowerBounds2}.
\end{align}
Let us focus on the first part of Condition~\ref{SpaceTimeCondition}. We now know by \eqref{CrucialLowerBounds1} that with high probability we have more than $N$ infected vertices at the top of the space-time box. Now the key idea is that if $N$ is large enough at least one of theses vertices $x$ will infect the spatially shifted box $x+[-n,n]^d$ after a time step of length $1$. Of course these events are highly correlated. Our way of dealing with this is to choose $N$ large enough such that we still find a sufficiently large number of vertices which are additionally far enough apart such that we can disregard this correlation:  
Given $\delta$ and $n$ choose an $N'$ such that
\begin{equation}\label{ManyPointsInfectSurely}
	\Big(1-\Pw\big(\prescript{}{n+1}\bfC^{\{\zero\},\emptyset}_{1}\supset [-n,n]^d\big)\Big)^{N'}<\delta.
\end{equation}
Furthermore, we then choose $N$ large enough such that for any configuration $A\subset \Z^d$ with $|A|> N$ there exists a $D\subset A$ with $|D|> N'$ and $||x-y||_{\infty}\geq 2n+1$ for all $x,y \in D$ with  $x\neq y$. In words, $N$ needs to be large enough such that any set $A$ which contains $N$ or more vertices will contain at least $N'$ vertices that are all a distance of $2n+1$ apart.
Let us now consider for such a subset $D$ and $T>0$ the event
\begin{align*}
	W^T_D=\{&\exists x\in D \text{ such that there are } \emptyset\text{-infection paths from } (x,T)  \text{ to every }\\& (y,T+1) \text{ with } y \in (x+[-n,n]^d)
	\text{ that stay in } (x+[-n,n]^d)\times(T,T+1]\}.
\end{align*}
Since we restricted the $\emptyset-$infection paths which lead from $(x,T)$ to every point in $(x+[-n,n]^d )\times \{T+1\}$ to stay in $(x+[-n,n]^d)\times (T,T+1]$ and we assumed that every two vertices $x,y\in D$ are $2n+1$ apart we get that
\begin{equation}\label{technicalInequality}
	\Pw\big(W^T_{D}\big)>1-\Big(1-\Pw(\prescript{}{n+1}\bfC^{\{\zero\},\emptyset}_{1}\supset [-n,n]^d)\Big)^{N'}>1-\delta.
\end{equation}
Note that $W^T_{D}$ only uses the graphical representation between $T$ and $T+1$, and thus it is independent of $\cF_T$, i.e.~$\Pw\big(W^T_{D}|\cF_T\big)=	\Pw\big(W^T_{D}\big)$. It is not difficult to see that
\begin{align}\label{technicalInclusion}
	&\{|\prescript{}{L}\bfC^{[-n,n]^d,\emptyset }_T\cap [0,L)^d|> N,\prescript{}{L}\bfC^{[-n,n]^d,\emptyset }_T\cap [0,L)^d=A\}\cap W^T_{D(A)}\notag\\
	\subset &\{\prescript{}{L+n}\bfC^{[-n,n]^d,\emptyset }_{T+1}\supseteq x+[-n,n]^d \text{ for some } x\in[0,L)^d\}.
\end{align}
where $D(A)$ is a subset of $A$ containing at least $N'$ elements, which are all spaced a distance $2n+1$ apart and also that
\begin{align}\label{technicalInclusion2}
	\begin{aligned}
		&\bigcup_{A\subseteq [0,L)^d}\{|\prescript{}{L}\bfC^{[-n,n]^d,\emptyset}_T\cap [0,L)^d|> N,\prescript{}{L}\bfC^{[-n,n]^d,\emptyset}_T\cap [0,L)^d=A\}\\
		=&\{|\prescript{}{L}\bfC^{[-n,n]^d,\emptyset}_T\cap [0,L)^d|> N\}.
	\end{aligned}
\end{align}
Thus, \eqref{CrucialLowerBounds1} and \eqref{technicalInequality}-\eqref{technicalInclusion2} yield together that
\begin{align*}
	&\Pw\big( \prescript{}{L+n}\bfC^{[-n,n]^d,\emptyset}_{T+1}\supset x+[-n,n]^d \text{ for some } x\in[0,L)^d\big)\\ \geq &\Pw\big(|\prescript{}{L}\bfC^{[-n,n]^d,\emptyset}_T\cap [0,L)^d|> N\big) (1-\delta)\geq (1-\delta^{2^{-d}})(1-\delta).
\end{align*}
Choosing $\delta$ accordingly yields the first part of Condition~\eqref{SpaceTimeCondition}. 

The second inequality can be concluded in a similar fashion since \eqref{CrucialLowerBounds2} and translation invariance yield that we get with high probability more than $M$ infected space-time points $(x,t)$ on an arbitrary but fixed intersection of a lateral face of $S(L,T)$ and an arbitrary orthant. Recall that $S_{+}(L,T)$ was the intersection of the lateral face in direction of the first coordinate $x_1$ and the first orthant.

Similarly to before we can choose $M$ large enough such that for any finite set $F\subset\Z^d\times \R_+$ which contains more than $M$ space-time points such that if the spatial coordinate of two points agree then their time coordinates are at least a distance of $1$ apart, we find a subset which contains at least $M'$ many space-time points which are all $2n+1$ apart in spatial distance. Again with high probability at least one point $(x,t)$ infects via an $\emptyset-$infection path the spatially shifted space-time box $x+[0,2n]\times [-n,n]^{d-1}\times \{t+1\}$. By a similar strategy as before we can deal with the dependency and get
\begin{align*}
	\Pw(\prescript{}{L+2n}\bfC^{[-n,n]^d,\emptyset}_{t+1}&\supseteq x+[-n,n]^d \text{ for some } 0\leq t < T\text{ and } x\in\{L+n\}\times[0,L)^{d-1})\\
	&>(1-\delta)(1-( 2\delta)^{d^{-1}2^{-d}}).
\end{align*}
Again choose $\delta$ accordingly such that we get the second inequality of Condition~\eqref{SpaceTimeCondition}.
\end{proof}

Now it remains to show that Condition~\ref{SpaceTimeCondition} implies survival of the CPDP. The strategy is to use Condition~\ref{SpaceTimeCondition} to define so-called "good blocks" and with that to construct an oriented percolation model on a macroscopic lattice which is coupled to the CPDP in the sense that if the percolation model survives it implies that also the CPDP survives.

The first step is to combine \eqref{SpaceTimeCondition1} and \eqref{SpaceTimeCondition2} into one, since this is more convenient for the construction. We consider the event
\begin{align}\label{SpaceTimeCondition3}
	\begin{aligned}
		\cA_3=\cA_3(n,L,T):=\big\{\prescript{}{2L+2n}\bfC^{[-n,n]^{d},\emptyset}_{t}&\supset x+[-n,n]^{d} \text{ for some } T\leq t < 2T \\
		&\text{ and } x\in[L+n,2L+n]\times[0,2L)^{d-1}\big\}.
	\end{aligned}	
\end{align}
Similarly to before we illustrate in \autoref{fig:SpaceTime} the cross section in direction of the first coordinate of the event in \eqref{SpaceTimeCondition3}.
\begin{figure}[ht]
	\centering
	\includegraphics[width=110mm]{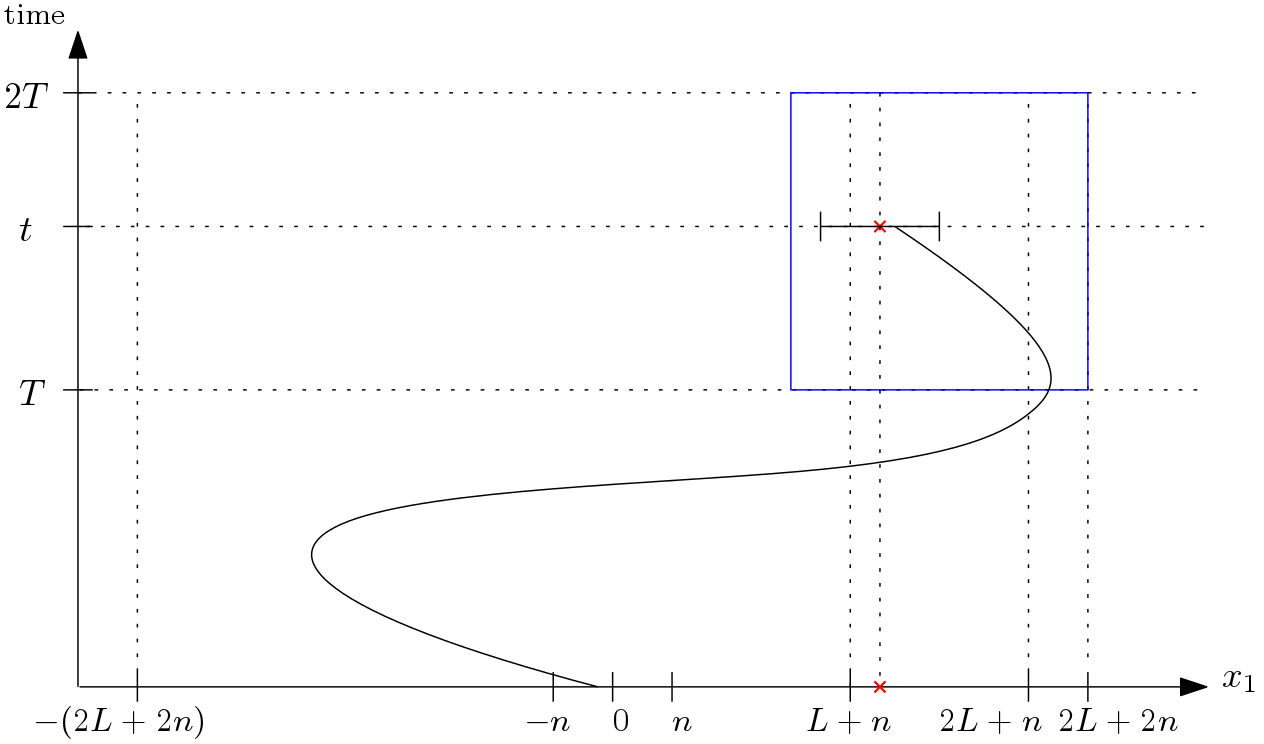}
	\caption{Illustration of the event $\cA_3$ defined in \eqref{SpaceTimeCondition3}. The blue space-time box shows the area where the infected space box of length $2n$ will be contained.}
	\label{fig:SpaceTime}
\end{figure}
This event states that we start with a space box $[-n,n]^d$ of infected vertices in the worst possible background configuration, then we find again such an infected space box at some later time shifted at least by $L+n$ and at most by $2L+n$ to the right along the first spatial coordinate.

Let us emphasize again that it is important to consider the background $\bfB$ to be started in the empty configuration $\emptyset$ since one of the key ideas is, roughly speaking, that we will stack appropriately shifted versions of $\cA_3$ so that at the transition from one block to the next we need to ensure that the background process is again restarted in the same configuration.
\begin{proposition}\label{KeyProperty1}
	Suppose Condition \ref{SpaceTimeCondition} holds. Then for every $\varepsilon>0$ there are choices of $n,L,T$ such that $\Pw(\cA_3)>1-\varepsilon$.
\end{proposition}
\begin{proof}
	We briefly describe the broad idea, illustrated also in Figure~\ref{fig:SpaceTime2}, and refer to \cite[Proposition~6.2.1]{seiler2021} for the proof details. 
	Broadly speaking, if for the CPDP the event $\cA_2$ occurs then there exists a point $(x,\tau)$
	with $1\leq \tau\leq T+1$ and $x\in\{L+n\}\times[0,L)^{d-1}$ such that
	\begin{equation*}
	    \prescript{}{L+2n}\bfC^{[-n,n]^{d},\emptyset}_{\tau}\supset  x+[-n,n]^{d}.
	\end{equation*}
	 Then  	 we restart the CPDP at time $\tau$ in the state $(x+[-n,n]^{d},\emptyset)$. If for this restarted process an appropriately spatially shifted version of the event $\cA_1$ occurs then it can be shown that this implies that for the original CPDP the event $\cA_3$ must have occurred. Now, due to 
	 Condition \ref{SpaceTimeCondition}  both of these events happen with high probability.
\end{proof}
\begin{figure}[t]
	\centering
	\includegraphics[width=110mm]{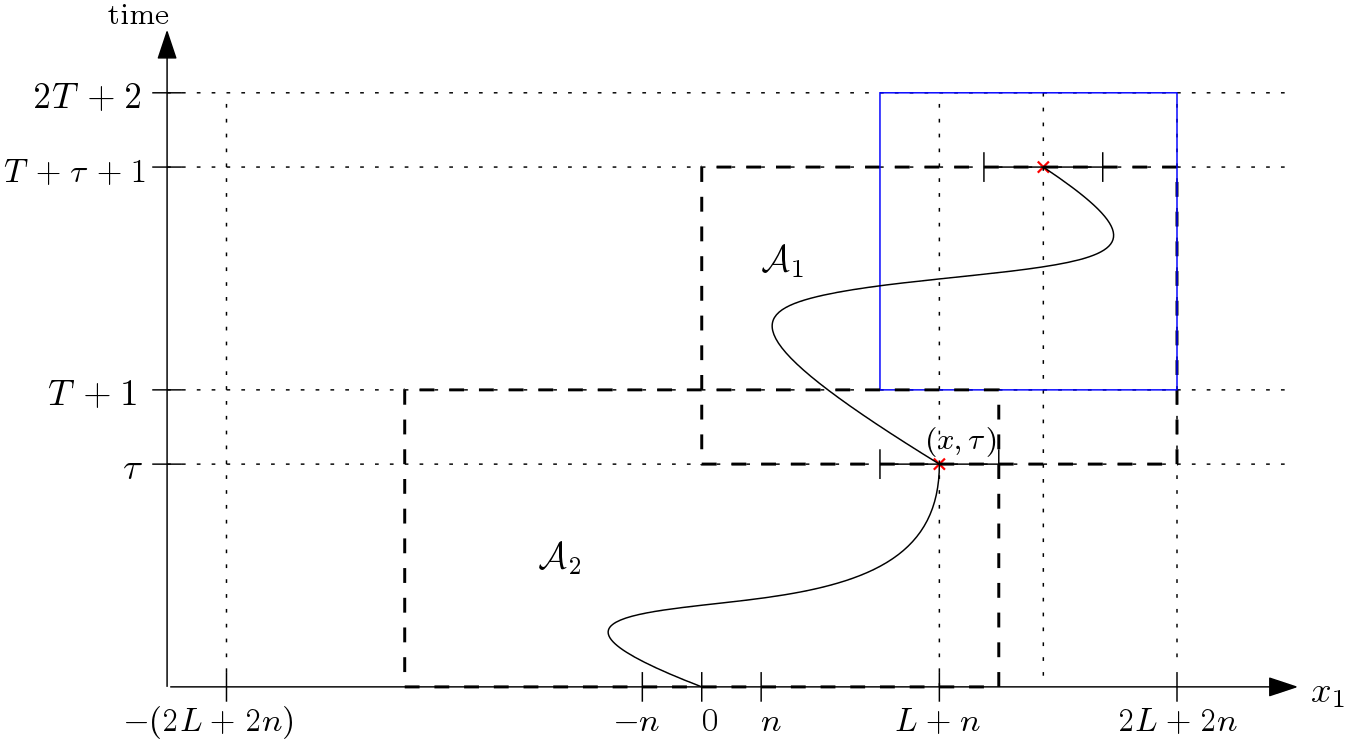}
	\caption{Here it is illustrated how $\cA_3$ is constructed by first using $\cA_2$ and then $\cA_1$, where we restart the process with only a copy of $[-n,n]^d$ being infected and the background in the state $\emptyset$.}
	\label{fig:SpaceTime2}
\end{figure}
Now we use the event $\cA_3$ and Proposition~\ref{KeyProperty1}, to successively define what in our case good blocks are and show that they occur with sufficiently high probability. Let us set
\begin{align*}
	\cD_{j,k}:=[-(1-2j)a,(1+2j)a]\times [-a,a]^{d-1}\times[5kb,(5k+1)b],
\end{align*}
where $j,k\in\Z$ and $a,b>0$.

\begin{proposition}\label{ConstructionOfBlocks}
	Suppose Condition \ref{SpaceTimeCondition} holds. Then for every $\varepsilon>0$ there are choices of $n,a,b$ with $n<a$ such that if $(x,s)\in \cD_{j,k}$,
	\begin{align*}
		\Pw\big(&\exists (y,t) \in \cD_{j+1,k+1} \text{ s.t. there are }\notag 
		\emptyset\text{-infection} 		\text{ paths that stay in } \\&
		([-5a,5a] +2ja) \times [-5a,5a]^{d-1}\times[0,6b] 	\text{ and lead from } \\ & (x,s)+([-n,n]^{d}\times \{0\}) \text{ to every point in } (y,t)+([-n,n]^{d}\times \{0\}) \big)>1-\varepsilon.
	\end{align*}
\end{proposition}
\begin{proof}
    The idea here is that one can apply Proposition~\ref{KeyProperty1} multiple times since we can stack appropriately spatially shifted and if necessary reflected versions of the event $\cA_3$ repeatedly to move the centre $(x,s)$ of the initially fully infected hypercube $x+[-n,n]^d$, where $x\in[-(2L+n),2L+n]^d$, in five to ten steps to a new centre $(y,t)$ with $y\in[2L+n,3(2L+n)]\times [-(2L+n),(2L+n)]^{d-1}$. Note that in every step we use the strong Markov property to restart the background process in the empty configuration. 
    
	We visualized the procedure in Figure~\ref{fig:BuildingBlocks}. In Figure~\ref{fig:BuildingBlocks} (a) we illustrated how we are able, with an appropriate spatial shift, to move the center of the hypercube in direction of a given coordinate axis such that it is contained in $[2L+n,3(2L+n)]$. Furthermore, in Figure~\ref{fig:BuildingBlocks} (b) it is illustrated how we ensure that via reflection over the remaining coordinate axes we keep the remaining coordinates in $[-(2L+n),(2L+n)]^{d-1}$. See \cite[Proposition~6.2.2]{seiler2021} for the detailed proof.
\end{proof}
\begin{figure}[t]
	\centering 
	\subfigure[\textbf{First coordinate:} Start in $(x,s)$ use \eqref{SpaceTimeCondition3} to find infected box around $(z,r)$ restart at this point and use successively \eqref{SpaceTimeCondition3}.]{\label{fig:BuildingBlocks:a}\includegraphics[width=76.5mm]{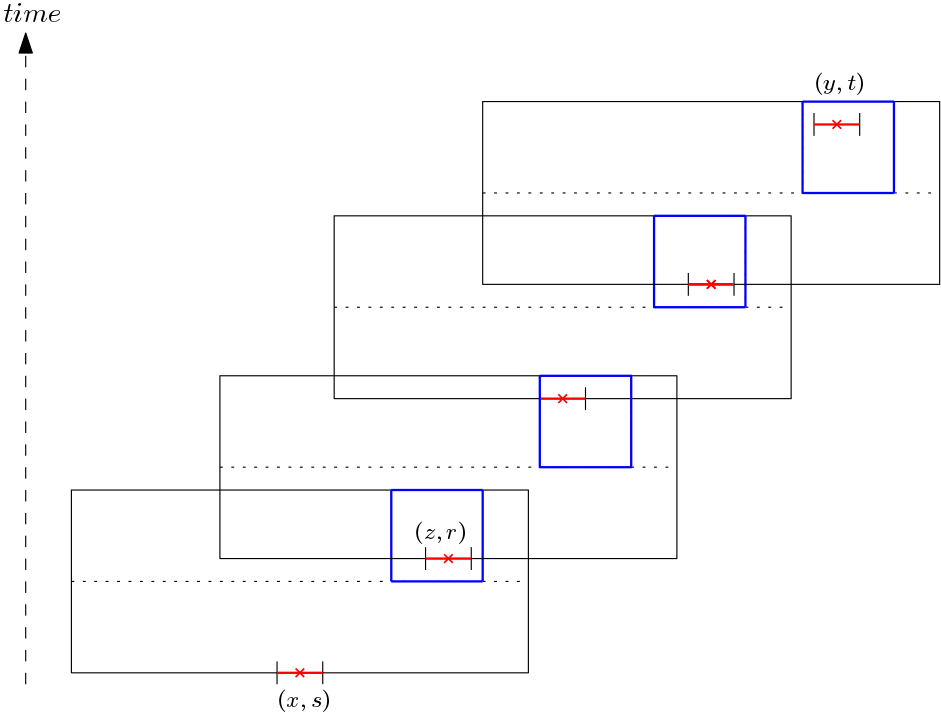}}\hfill
	\subfigure[\textbf{All other coordinates:} Assume $x_i=0$, while using successively \eqref{SpaceTimeCondition3} reflect along the coordinate plane if $i$-th coordinate changes its sign. Note that after achieving $x_1\in\text{[$a,3a$]} $ we apply this strategy to the first coordinate as well until $t\in \text{[$5b,6b$]}$  ]{\label{fig:BuildingBlocks:b}\includegraphics[width=58.5mm]{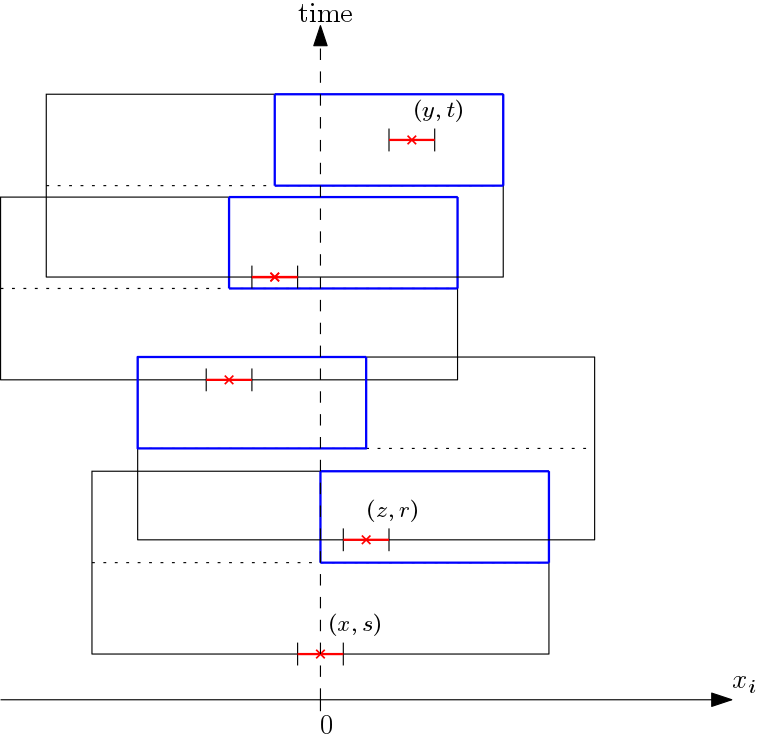}}
	\caption{Visualization of the construction in Proposition \ref{ConstructionOfBlocks}}
	\label{fig:BuildingBlocks}
\end{figure}
Now we define good block events so that we can construct  a suitable coupled oriented percolation model on the ``macroscopic" lattice $\{(j,k)\in \Z\times \N_0: j+k \text{ even}\}$. We identify the points $(j,k)$ with the space-time boxes
\begin{equation*}
	\cS_{j,k}:=[a(12j-1),a(12j+1)]\times[-a,a]^{d-1} \times [30kb,(30k+1)b]=\cD_{6j,6k}.
\end{equation*}
Broadly speaking, we declare $(j,k)$ to be open if we find an appropriate space-time shift of $[-n, n]^{d}\times \{0\}$ in $\cS_{j,k}$, which is completely infected. For $a,b>0$ as in Proposition~\ref{ConstructionOfBlocks} let $w(j,k):=((12ja,0,\dots,0),30kb)\in \Z^{d}\times \N_0$ and set
\begin{equation*}
\cM^{\pm}(j,k):=\Big(\bigcup_{l=0}^{6} \,\,([-5a,5a]\pm 2la)\times[-5a,5a]^{d-1}\times[5l,(5l+1)b]\Big) +w(j,k).
\end{equation*}
See the solid boxes in Figure~\ref{fig:Macro} for an illustration. At last we now formulate the events which we call "good blocks" $\cB^{\pm}=\cB^{\pm}(j,k,(x,s))$, where $j,k\in \Z$ and $(x,s)\in \cS_{j,k}$. 
We set
\begin{align}\label{MacroscopBuildBlock}
	\cB^{\pm}:=\left\{\begin{aligned}
	&\exists (y,t) \in \cS_{j\pm1,k+1}
	\text{ and there are } \emptyset-\text{infection}
	\text{ paths }\\
	&\text{ contained in }\cM^{\pm}(j,k)\text{ which lead from } (x+[-n,n]^{d})\times \{s\} \\
	&\text{ to every point in } (y+[-n,n]^{d})\times \{t\}.
	\end{aligned}\right\}.
\end{align}
For these events, similarly to Proposition $\ref{ConstructionOfBlocks}$, the following lemma holds.
\begin{lemma}\label{BuildBlock}
	Suppose Condition~\ref{SpaceTimeCondition} holds. Then for every $\varepsilon>0$ there are choices of $n,a,b$ with $n<a$ such that if $(x,s)\in \cS_{j,k}$ then $\Pw(\cB^{\pm})>1-\varepsilon$ where $(j,k)\in \Z\times\N_0$.
\end{lemma}
\begin{proof}
	This is a direct consequence of Proposition~$\ref{ConstructionOfBlocks}$ in the sense that we can again stack spatially shifted versions of the events formulated in Proposition~$\ref{ConstructionOfBlocks}$ in an appropriate manner to achieve this result.
	Note that  for $\cB^{-}$ the statement of Proposition~\ref{ConstructionOfBlocks} for a reflected version over the first coordinate axis of the considered event is needed. This also holds true, i.e.~we reflect the whole construction in the direction of the first coordinate at $(2ja,0,\dots,0)\in \Z^d$  such that at the end $(y,t)\in \cD_{j-1,k+1}$.
	For more details see \cite[Lemma~6.2.4]{seiler2021}. 
\end{proof}
Note that the boxes $\cB^{\pm}$ only depend on a finite sector of the graphical representation and only overlap with the adjacent boxes (see \autoref{fig:Macro}).
At first this last step seems a bit redundant, since we could very well work with the events defined in Proposition \ref{ConstructionOfBlocks}, but with this additional step we made the dependency between the respective events clearer. Now we are ready to prove the main theorem of this section.
\begin{theorem}\label{ComparisonWithOrientedPercolation}
	Suppose Condition \ref{SpaceTimeCondition} holds. Then for every $q<1$ there are choices of $n,a,b$ such that if the initial configurations $W_0\subset 2\Z$ and $\bfC_0=C$ satisfy 
	\begin{align}\label{CompCondition1}
		j\in W_0 \Rightarrow C\supset x+[-n,n]^d \text{ for some } x\in[a(12j-1),a(12j+1)]\times[-a,a]^{d-1}
	\end{align}
	then $(\bfC_t,\bfB_t)_{t\geq 0}$ can be coupled with an oriented site percolation $(W_k)_{k\geq 0}$ with parameter $q$ such that
	\begin{equation}\label{CompCondition2}
		j\in W_k \Rightarrow \bfC_t\supset x+[-n,n]^d \text{ for some } (x,t)\in \cS_{j,k}
	\end{equation}
	In particular this implies that the CPDP survives.
\end{theorem}
\begin{figure}[t]
	\centering
	\includegraphics[width=140mm]{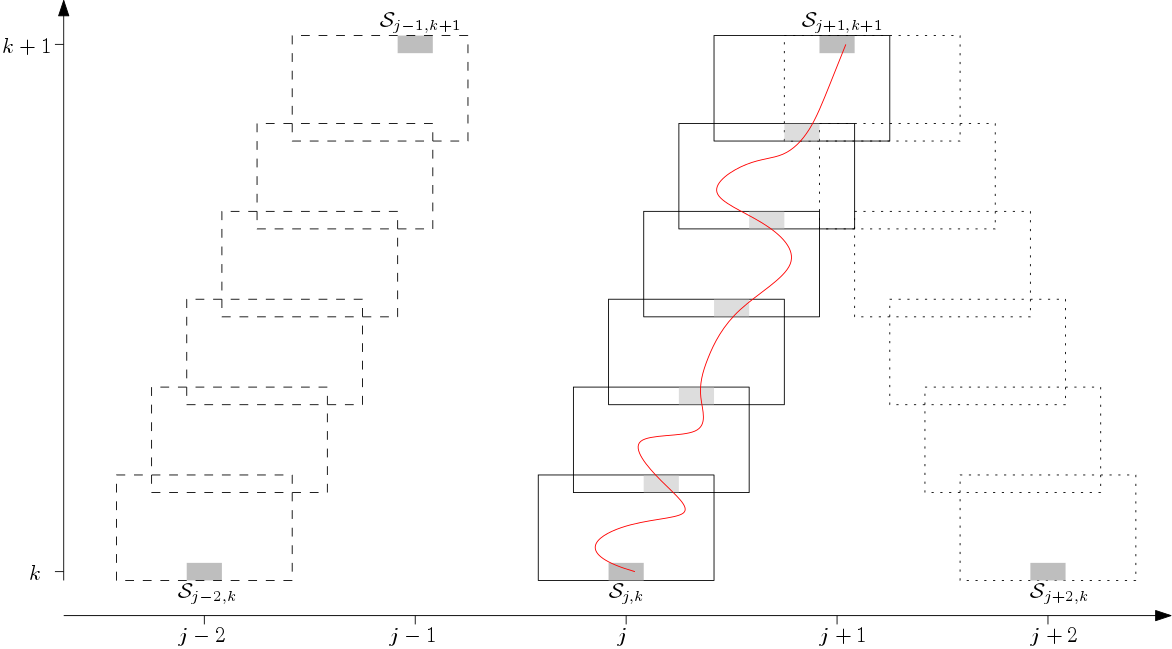}
	\caption{Here we see a visualization of the space-time boxes $\cB^{\pm}$ (defined in \eqref{MacroscopBuildBlock}), where the solid line visualizes the box $\cB^{+}(j,k,\cdot)$. We also see that $\cB^{+}(j,k,\cdot)$ only overlaps with $\cB^{-}(j,k,\cdot)$ and $\cB^{-}(j+2,k,\cdot)$, where the dotted lines visualizes $\cB^{-}(j+2,k,\cdot)$ and the dashed $\cB^{+}(j-2,k,\cdot)$.}
	\label{fig:Macro}
\end{figure}
\begin{proof}
    The idea is that we construct our percolation model recursively with the help of Lemma \ref{BuildBlock}. Thus, let for an arbitrary $\varepsilon>0$ the numbers $n,a,b$ be chosen as in the Lemma \ref{BuildBlock}.
	Note that since the events we use are not independent we need to use a comparison of independent and locally dependent Bernoulli random variables to obtain an independent oriented site percolation in a second step.
	
	We now construct random variables  $(X_{j}(k),Y_{j}(k))$ with $k\geq 0$ and $j\in \Z$: The variables $X_{j}(k)$ will either be $1$ if there exists an $(x,s)\in S_{j,k}$ such that $(x,s)+([-n,n]^d\times\{0\})$ is infected and otherwise $0$. Additionally if such a point exists we set $Y_{j}(k)=(x,s)$ and if not $Y_{j}(k)=\dagger$, where $\dagger$ is a designated state such that the state space of these random variables is $\{0,1\}\times \big(\Z^d\times [0,\infty)\big)\cup\{\dagger\}$.
	
	Without loss of generality we  assume that $W_0=\{0\}$. By assumption \eqref{CompCondition1} there exists an $x_0$ such that $(x_0,0)\in \cS_{0,0}$ and $ x_0+[-n,n]^d$ is initially infected. We set $(X_{0}(0),Y_{0}(0))=(1,(x_0,0))$ and $(X_{j}(0),Y_{j}(k))=(0,\dagger)$ for all $j\neq 0$. Now with respect to $k$ we recursively construct these random variables. Suppose that $(X_{j}(k),Y_{j}(k))_{j\in \Z}$ are defined for all $k\leq m$, then we proceed with the step $m\to m+1$.
	\begin{enumerate}
		\item If $X_{j-1}(m)=0$ and $X_{j+1}(m)=0$ then we set $(X_{j}(m+1),Y_{j}(m+1))=(0,\dagger)$
		\item We set $X_{j}(m+1)=1$ if either $X_{j-1}(m)=1$ and the event $\cB^{+}(j-1,m,Y_{j-1}(m))$ occurs or $X_{j+1}(m)=1$ and $\cB^{-}(j+1,m,Y_{j+1}(m))$ occurs.	Again the events $\cB^{+}(j-1,m,Y_{j-1}(m))$ and $\cB^{-}(j+1,m,Y_{j+1}(m))$ only guarantee existence of a point $(y,t)\in S_{j,m+1}$ such that $(z,r)+([-n,n]^d\times\{0\})$ is completely infected, but there might exist more than one. We let $Y_{j}(m+1)$ be the smallest space-time point $(y,t)$ in the sense that we take the earliest with respect to time and, if that does not yield a unique point, we minimize according to an arbitrary but beforehand specified order on $\Z^d$. 
	\end{enumerate}
	\vspace{-1mm}
	By this construction for fixed $k\geq 0$ the set $\{j: X_j(k)=1\}$ obviously satisfies \eqref{CompCondition2}. Next let $\cG_m$ be the $\sigma$-algebra generated from all $(X_j(k),Y_j(k))_{j\in \Z}$ with $k\leq m$. By the choice of $n,a,b$ made at the beginning of the proof we see that
	\begin{align*}
	\Pw(X_{j}(m+1)=1|\cG_m)>1-\varepsilon \quad \text{ on } \quad \{X_{j}(m)=1 \text{ or } X_{j-1}(m)=1 \}.
	\end{align*}
	We use now a standard result concerning $k$-dependent Bernoulli random variables. Since $\cB^{\pm}$ only overlap with their adjacent boxes, by construction the $(X_{j}(m+1))_{j\in \Z}$ are conditional on $\cG_m$ a $1$-dependent family of Bernoulli variables. By \cite[Theorem~B26]{liggett2013stochastic} we find a family of independent Bernoulli variables such that we can define an oriented site percolation $(W_k)_{k\geq 0}$  with parameter $q:=(1-\varepsilon^{-3})^2$ which satisfies \eqref{CompCondition1} and \eqref{CompCondition2}. Since $\varepsilon$ was arbitrary this finishes the proof.
\end{proof}

\subsection{Consequences of the percolation comparison}\label{ConsequenceOfComparison}
In this section we can finally reap the benefits of all the work we have done so far in Subsection~\ref{InvolvingBlockContruction} for the CPDP on $\Z^d$. First we prove that at criticality, survival is not possible and as direct consequence we gain continuity of the survival probability.  
Then, we use Theorem~\ref{ComparisonWithOrientedPercolation} to show that for the CPDP the two conditions \eqref{ConvergenceCond1.1} and \eqref{ConvergenceCond2.1} are satisfied if the survival probability is positive such that by Theorem~\ref{CompleteConvergenceComplete} it follows that complete convergence for the CPDP holds. 
\subsubsection{Extinction at criticality and continuity of the survival probability}
Here, we show extinction at criticality (Theorem~\ref{NoSurvivalAtCriticality}) and continuity of the survival probability (Proposition~\ref{JointContiuityOfCPDP}). 
Recall that we defined in \eqref{SurvivalRegionCPDP} the survival region as \begin{equation*}
    \cS:=\{(\lambda,r,\alpha,\beta)\in (0,\infty)^2:\theta\big(\lambda,r,\alpha,\beta,\{\zero\},\emptyset\big)>0\}.
\end{equation*}
In case of the CPDP we have two additional parameters $\alpha$ and $\beta$ for which we can easily deduce similar monotonicity and continuity properties as for the infection and recovery rate $\lambda$ and $r$ in Section~\ref{BasicProperties}.
\begin{lemma}[Monotonicity with respect to the background]\label{MonotonicityOfCPDPLemma}
	Let $(\bfC,\bfB)$ be a CPDP with parameters $\lambda,r,\alpha,\beta$. Let $\widehat{\alpha}\geq \alpha$, then there exists a CPDP $(\widehat{\bfC},\widehat{\bfB})$ with parameter $\lambda,r,\widehat{\alpha},\beta$ and the same initial configuration such that $\bfC_t\subseteq  \widehat{\bfC}_t$ and $\bfB_t\subseteq  \widehat{\bfB}_t$ for all $t\geq0$. In words $\bfC$ is monotone increasing in $\alpha$. On the other hand $\bfC$ is monotone decreasing in $\beta$.
\end{lemma}
\begin{proof}
	This follows with an analogous coupling as in the proof of Lemma~\ref{MonotonicityAdditivityLemma}. Since if we consider $\widehat{\alpha}\geq \alpha$
	then we can define $(\widehat{\bfC},\widehat{\bfB})$ via $\Xi_{\lambda,r}$ and add additional independent opening events with rate $(\widehat{\alpha}-\alpha)$. Thus we get a CPDP $(\widehat{\bfC},\widehat{\bfB})$ which is coupled to $(\bfC,\bfB)$ such that $\bfC_t\subset \widehat{\bfC}_t$ and $\bfB_t\subset \widehat{\bfB}_t$ for all $t\geq0$. The monotonicity in $\beta$ follows analogously.
\end{proof}
\begin{remark}
	Obviously $\pi_{\alpha,\beta}\preceq \pi_{\widehat{\alpha},\beta}$ if $\alpha\leq \widehat{\alpha}$. Thus, if we consider the stationary case, i.e. $\bfC_0=C\subset V$ and $\bfB_0\sim \pi_{\alpha,\beta}$, then there exists an CPDP $(\widehat{\bfC},\widehat{\bfB})$ with parameter $\lambda,r,\widehat{\alpha},\beta$ and $\bfC_0=C\subset V$ and $\widehat{\bfB}_0\sim \pi_{\widehat{\alpha},\beta}$ such that $\bfC_t\subseteq  \widehat{\bfC}_t$ and $\bfB_t\subseteq  \widehat{\bfB}_t$ for all $t\geq0$. This follows by first coupling the initial state of the background with Strassen's Theorem such that $\bfB_0\subseteq  \widehat{\bfB}_0$ and then using Lemma~\ref{MonotonicityOfCPDPLemma}.
\end{remark}
\begin{lemma}[Continuity for finite times and finite initial infections]\label{ContinuityLemma2}
	Let $(\bfC,\bfB)$ be a CPDP with initial configuration $C\subset  V$ with $|C|<\infty$. Also let $\cA\subset D_{\cP(V)}([0,t])$ for $t\geq 0$.
	\begin{enumerate}
		\item[(i)] The maps $
		\alpha\mapsto \Pw^{(C,B)}_{\lambda,r,\alpha,\beta}\big((\bfC_s)_{s\leq t}\in\cA\big) \text{ and } \beta\mapsto \Pw^{(C,B)}_{\lambda,r,\alpha,\beta}\big((\bfC_s)_{s\leq t}\in\cA\big)$
		are continuous. 
		\item[(ii)] If $\bfB_{0}\sim \pi_{\alpha,\beta}$, then $\alpha \mapsto \Pw^{(C,\pi_{\alpha,\beta})}_{\lambda,r,\alpha,\beta}\big((\bfC_s)_{s\leq t}\in\cA\big) \text{ and } \beta \mapsto \Pw^{(C,\pi_{\alpha,\beta})}_{\lambda,r,\alpha,\beta}\big((\bfC_s)_{s\leq t}\in\cA\big)$
		are continuous.
	\end{enumerate}
\end{lemma}
\begin{proof}
	\begin{enumerate}
		\item[(i)]  The proof for $\alpha$ and $\beta$ is similar to the proof of Lemma~\ref{ContinuityLemma}.
		To show the continuity in $\alpha$ we can again construct a CPDP $(\widehat{\bfC},\widehat{\bfB})$ with rates $\tilde{\alpha}$ and $\beta$ with $\tilde{\alpha}> \alpha$ such that these processes are coupled in the way that $\bfC_0=\widehat{\bfC}_0$, $\bfB_0=\widehat{\bfB}_0$, $\bfC_t\subseteq\widehat{\bfC}_t$ and $\bfB_t\subseteq\widehat{\bfB}_t$ for all $t\geq 0$. Then, it again suffices to show that as $\widehat{\alpha}\to \alpha$ it follows that $\Pw(\bfC_s\neq \widehat{\bfC}_s \text{ for some } s\leq t)\to 0$ to conclude right continuity. Left continuity follows via an analogous procedure. The continuity in $\beta$ follows analogously.
		\item[(ii)] Here the difference to (i) is that the the invariant law $\pi$ is the initial state of the background and thus we additionally need to ensure that the initial state of the process $\widehat{\bfB}$ also converges. This is not difficult to show since all edges are independent. \qedhere
	\end{enumerate}
\end{proof}
We are finally ready to show that survival is impossible at criticality for the CPDP on $\Z^d$.  
\begin{proof}[Proof of Theorem~\ref{NoSurvivalAtCriticality}]
	We first note that the good blocks $\cB^{\pm}$ defined in \eqref{MacroscopBuildBlock} only depend on a bounded section of the graphical representation, so by Lemma~\ref{ContinuityLemma} and Lemma~\ref{ContinuityLemma2} we get that $\Pw_{ \lambda,r,\alpha,\beta}(\cB^{\pm})$ is continuous seen as a function of any of the four parameters. Let us take as usual the infection parameter $\lambda$ as an example.
	By Proposition \ref{ConstructionOfBlocks} we know that for every $\varepsilon>0$ we can find $a,b,n$ such that $\Pw_{ \lambda}(\cB^{\pm})>1-\varepsilon$ Then because of continuity there must exist a $\lambda'<\lambda$ such that $\Pw_{ \lambda'}(\cB^{\pm})>1-\varepsilon$ as well and hence by Theorem \ref{ComparisonWithOrientedPercolation} it follows that the CPDP also survives with infection rate $\lambda'$. This proves the claim.
\end{proof}
Recall that we call a function $f:\R^d\subset U\to \R$ separately continuous if it is continuous in each coordinate separately. In comparison to that one calls $f$ jointly continuous if it is continuous with respect to the Euclidean topology on $\R^d$.
\begin{proposition}\label{ContinuityCPDP}
	Let $C\subset V$ with $C$ finite and non-empty and $B\subset E$. 
	\begin{enumerate}
		\item[(i)] The survival probability $\theta(\lambda,r,\alpha,\beta,C,B)$ is separately right continuous seen as a function in $(\lambda,r,\alpha,\beta)$ on $(0,\infty)^4$.
		\item[(ii)] The survival probability $\theta(\lambda,r,\alpha,\beta,C,B)$ is separately left continuous seen as a function in $(\lambda,r,\alpha,\beta)$ on $\inte(\cS)$.
	\end{enumerate}
\end{proposition}
\begin{proof}
	We already shown right and left continuity in $\lambda$ and $r$ on the respective parameter sets  in Proposition~\ref{RightContinuityOfSurv} and \ref{LeftContinuityOfSurv}. Right and left continuity in $\alpha$ and $\beta$ can be shown with the same approach.
\end{proof}
We have seen that Corollary~\ref{NoSurvivalAtCriticality} states that the infection process $\bfC$ cannot survive at criticality. As a consequence of this fact we can conclude that the survival probability is jointly continuous with respect to its parameters $(\lambda,r,\alpha,\beta)$. 
\begin{proof}[Proof of Proposition~\ref{JointContiuityOfCPDP}]
	Proposition~\ref{ContinuityCPDP} shows that the survival probability is separately left continuous seen as a function in the four parameters $(\lambda,r, \alpha,\beta)$ on $\inte(\cS)$ and is separately right continuous on $(0,\infty)^4$. Now let us again prove continuity for $\lambda\mapsto\theta(\lambda,r,\alpha,\beta,C,B)$. The proof is analogous for the remaining three parameter. By Proposition~\ref{ContinuityCPDP} it is clear that the function is everywhere continuous expect at criticality. Now obviously in case of $\lambda$ the left limit at criticality exists, since we come from the subcritical parameter region where the survival probability is constant $0$. By Theorem \ref{NoSurvivalAtCriticality} we know that the CPDP a.s. goes extinct at criticality, which means that the survival probability is $0$. But with that we have shown that the left limit and the right limit at the critical value are the same since $\lambda\mapsto\theta(\lambda,r,\alpha,\beta,C,B)$ is right continuous on $(0,\infty)$ by Proposition~\ref{ContinuityCPDP} and thus, the function is continuous.
	
	Now we know that the survival probability is separately continuous seen as a function of the four parameters. But the function is also monotone in each coordinate, so we can use \cite[Proposition~2]{kruse1969joint}, which states that if a function is continuous and monotone in each coordinate, then it is jointly continuous.
\end{proof}
\subsubsection{Complete Convergence of the CPDP}
In this section we show complete convergence of the CPDP on $\Z^d$ and thus Theorem~\ref{CPDPCompleteConvergenceCPERE}. We start by showing that the second condition \eqref{ConvergenceCond2.1} holds.
\begin{proposition}\label{SecondCondition}
Let $(\bfC,\bfB)$ be the CPDP on $\Z^d$
	and $(\lambda,r,\alpha,\beta)\in \cS$. Then for every $x\in \Z^d$,
	\begin{align*}
	\lim_{n\to \infty}\liminf_{t\to\infty}\Pw_{\lambda,r,\alpha,\beta}^{(\B_n(x),\emptyset)}(\bfC_{t}\cap \B_n(x)\neq \emptyset)=1.
	\end{align*}
\end{proposition}
\begin{proof}
	By translation invariance it suffices to prove the claim for $x=\zero$.
	For $d\geq 2$ the claim follows analogously as in the second part of the proof of \cite[Theorem 2.27]{liggett2013stochastic}. Hence, we only need to consider $d=1$.
	
	Again by Theorem \ref{ComparisonWithOrientedPercolation} for every $0<q<1$ there exist $n,a,b$ such that an oriented site percolation $(W_k)_{k\geq	0}$ with parameter $q$ exists, which satisfies \eqref{CompCondition1} and \eqref{CompCondition2}. Now let us consider the set $D_m=(-15am-2,15am+1)$. By construction of the oriented site percolation in the proof of Theorem \ref{ComparisonWithOrientedPercolation} (see \autoref{fig:Macro} for a visualization) it follows that for $m>0$,
	\begin{align}\label{help3}
	\liminf_{t\to\infty}\Pw^{(D_m,\emptyset)}(\bfC_{t}\cap D_m	\neq\emptyset) \geq\liminf_{k\to\infty}\Pw^{\{-m,\dots, m\}}(W_k\cap \{-m,\dots, m\}\neq \emptyset),
	\end{align}
	since the infection is always contained in the blocks $\cB^{\pm}$. Now by \cite[Theorem 2]{durrett1987stochastic} it follow that the right-hand side in \eqref{help3} converges to $1$ as $m\to \infty$. 
\end{proof}
Now it is left to prove that \eqref{ConvergenceCond1.1} holds true. We split the proof of this condition in two parts. First we show with Theorem~\ref{ComparisonWithOrientedPercolation} that a positive survival probability already implies that the probability that a single vertex is infinitely often infected is positive as well. 
\begin{proposition}\label{FirstCondition1} Let $(\bfC,\bfB)$ be a CPERE on $\Z^d$ and $(\lambda,r,\alpha,\beta)\in \cS$, then $\Pw^{(C,B)}_{\lambda,r,\alpha,\beta}(x\in \bfC_t \text{ i.o.})>0$ for all $x\in V$ and all non-empty $C\subset V$ and $B\subset E$.
\end{proposition}
\begin{proof}
We briefly summarize the proof strategy which  is again analogous to that of the classical contact process with the difference 
that one needs to choose the background appropriately when the strong Markov property is used to restart the process. For a detailed proof see \cite[Proposition~6.3.6]{seiler2021}. 

First, one observes that it suffices to show
	$\Pw^{(\zero,\emptyset)}_{\lambda,r,\alpha,\beta}(\zero\in \bfC_t \text{ i.o.})>0.$
Then the rough idea is that by Theorem \ref{ComparisonWithOrientedPercolation} we can choose an oriented site percolation $(W_k)_{k\geq0}$ which satisfies \eqref{CompCondition1} and \eqref{CompCondition2} with parameter $0<q<1$ large enough such that $\Pw^{\{0\}}(0\in W_{2k} \text{ i.o.} )>0$, i.e.~$0$ is hit infinitely often. That the latter is true for large enough $q$ is a well-known fact for oriented percolation. But this means that by \eqref{CompCondition1} and \eqref{CompCondition2} there is a positive probability for the event
\begin{equation*}
 	\bfC_t^{[-n,n]^d,\emptyset}\supset x+[-n,n]^d \cap \Z^d \text{ for some } (x,t)\in \cS_{0,k}
\end{equation*}
for infinitely many $k$. This means that with positive probability we have infinitely many times $t_k$ such that at this time a vertex $x_k$ exist such that the box $x+[-n,n]^d\cap \Z^d$ is fully infected. Furthermore we know that there exists a uniform bound on the distance of all $x_k$ to $\zero$ since all $(t_k,x_k)$ are contained in $\cS_{0,k}$. Thus, using a generalized version of Borel-Cantelli's Lemma we obtain the claim.
\end{proof}
Next we show (for a general CPERE on $\Z^d$) that if we have a positive probability that a single vertex is infinitely often infected we can already conclude that \eqref{ConvergenceCond1.1} holds. 
\begin{proposition}\label{FirstCondition2}
	Let $(\bfC,\bfB)$ be a CPERE on $\Z^d$.
	Suppose that $\Pw^{(C,B)}_{\lambda,r}(x\in \bfC_t \text{ i.o.})>0$ for all $x\in V$, all non-empty $C\subset V$ and all $B\subset E$. Then we have
	\begin{equation*}
		\Pw^{(C,B)}_{\lambda,r}(x\in \bfC_t \text{ i.o.})=\theta(\lambda,r,C,B).
	\end{equation*}
\end{proposition}
\begin{proof}
Again there are only minor adjustments needed for the proof compared to how this is proven for the classical contact process. Therefore, we only present a proof sketch here. For the details we refer to \cite[Proposition~6.3.6]{seiler2021}.
	
	First, we observe that $\{\zero \in \bfC_t \text{ i.o.}\}\subset \{\bfC_t\neq\emptyset \,\, \forall t\geq 0\}$. Thus, to show the claim we need to show the converse inclusion. We set $A:=\{\zero\in \bfC_t \text{ i.o.}\}$, then one can show by using monotonicity and the strong Markov property the inequality 
	\begin{equation*}%\label{KeyInequality}
		\Pw(A|\cF_s)=\Pw^{(\bfC_s,\bfB_s)}(A)\geq \Pw^{(\{x\},\emptyset)}(\zero\in \bfC_t \text{ for some } t\geq 0)\Pw^{(\{\zero\},\emptyset)}(A)\1_{\{x\in \bfC_s\}}.
	\end{equation*}
	Since $\Pw^{(\{\zero\},\emptyset)}(x\in\bfC_t)>0$ for any $x\in \Z^d$ it follows by symmetry properties of $\Z^d$ and translation invariance that
	\begin{equation*}
        \Pw^{(\{x\},\emptyset)}(\zero\in \bfC_t \text{ for some } t\geq 0)=\Pw^{(\{\zero\},\emptyset)}(x\in \bfC_t \text{ for some } t\geq 0)\geq \Pw^{(\{\zero\},\emptyset)}(A),
    \end{equation*}
    and thus we get that
    \begin{equation*}
		\Pw(A|\cF_s)\geq \big(\Pw^{(\{\zero\},\emptyset)}(A)\big)^2\1_{\{x\in \bfC_s\}}.
	\end{equation*}
    Now by applying the martingale convergence theorem it follows that $\Pw(A|\cF_s)\to \1_A$ as $s\to \infty$ since $A$ is an element of the tail $\sigma$-algebra. But this implies that 
    \begin{equation*}
        \{\bfC_t\neq \emptyset \,\, \forall t\geq 0 \}\subset \{\zero\in \bfC_t \text{ i.o.}\}
    \end{equation*}
    almost surely. Thus, it follows that the sets must be almost sure equal, which provides the claim.
	\qedhere
\end{proof}
Finally we are able prove that complete convergence holds for the CPDP on the whole parameter set $(0,\infty)^{4}$.
\begin{proof}[Proof of Theorem~\ref{CPDPCompleteConvergence}]
	Suppose that $(\lambda,r,\alpha,\beta)\in \cS$. Then by  Proposition~\ref{SecondCondition}, Proposition~\ref{FirstCondition1} and Proposition~\ref{FirstCondition2} we know that \eqref{ConvergenceCond1.1} and \eqref{ConvergenceCond2.1} are satisfied, and thus by Theorem~\ref{CompleteConvergenceComplete} it follows that
	\begin{equation*}
	(\bfC^{C,B}_t,\bfB^{B}_t)\Rightarrow (1-\theta(C,B))(\delta_{\emptyset}\otimes\pi)+\theta(C,B)\overline{\nu} \quad \text{ as } t\to\infty.
	\end{equation*}
	On the other hand if $(\lambda,r,\alpha,\beta)\in \cS^c$, then by  Proposition~\ref{EqualityOfCriticalValues} it follows that $\overline{\nu}=\delta_{\emptyset}\otimes\pi$. Thus, it follows that $(\bfC^{V,E}_t,\bfB^{E}_t)\Rightarrow \delta_{\emptyset}\otimes\pi$ as $t\to\infty$.	By monotonicity shown in Lemma~\ref{MonotonicityAdditivityLemma} we then know that $(\bfC^{C,B}_t,\bfB^{B}_t)\Rightarrow \delta_{\emptyset}\otimes\pi$ as $t\to \infty$ for all $C\subset V$ and $B\subset E$, which proves the claim.
\end{proof}
We conclude this section by showing that for a general CPERE on $\Z^d$ 
complete convergence holds on a suitable subset of its survival region,
namely on the survival region of a suitably chosen CPDP, which lies "below" the CPERE. We use the subscript DP in order to distinguish between a CPERE and a CPDP.
\begin{proof}[Proof of Theorem~\ref{CPDPCompleteConvergenceCPERE}]
	Let $(\bfC,\bfB)$ be a CPERE
	whose background process has spin rate $q(\cdot,\cdot)$. Recall from \eqref{MaximalAndMinimalRates1} the rates $\alpha_{\min}:=\min_{F\subset \cN^{L}_e }q(e,F)$ and $\beta_{\max}:=\max_{F\subset \cN^{L}_e }q(e,F\cup \{e\})$. 
	By Proposition~\ref{ComparisonCPEREandCPDP} there exists a CPDP $(\underline{\bfC},\underline{\bfB})$ with rates $\alpha_{\min}$, $\beta_{\max}$  and the same initial configuration as $(\bfC,\bfB)$, i.e.~$\bfC_0=\underline{\bfC}_0$ and $\bfB_0=\underline{\bfB}_0$, such that $\underline{\bfC}_t\subset \bfC_t$ and $\underline{\bfB}_t\subset \bfB_t$ for all $t\geq 0$. This implies that
	\begin{equation}\label{StrongSurvCPERE}
		\Pw(x\in \underline{\bfC}_t \text{ i.o.})\leq \Pw(x\in \bfC_t \text{ i.o.})
	\end{equation}
	By assumption $\theta_{\text{DP}}(\lambda,r,\alpha_{\min},\beta_{\max},\{\zero\},\emptyset)>0$, and thus Proposition~\ref{FirstCondition1} and \eqref{StrongSurvCPERE} imply that $\Pw^{(C,B)}_{\lambda,r}(x\in \bfC_t \text{ i.o.})>0$ for any finite and non-empty set $C\subset\Z^d$ and any $B\subset E$. Furthermore, by Proposition~\ref{FirstCondition2} it follows that the first condition \eqref{ConvergenceCond1.1} holds. Due to the fact that  $\underline{\bfC}_t\subset \bfC_t$ and $\underline{\bfB}_t\subset \bfB_t$ for all $t\geq 0$ and Proposition~\ref{SecondCondition} one also obtains that \eqref{ConvergenceCond2.1} is satisfied. Since we assumed that $(i)$-$(iii)$ of Assumption~\ref{AssumptionBackground} are satisfied  Theorem~\ref{CompleteConvergenceComplete} now implies that if $\theta_{\text{DP}}(\lambda,r,\alpha_{\min},\beta_{\max},\{\zero\},\emptyset)>0$, then 
	\begin{equation*}
		(\bfC^{C,B},\bfB^{B})\Rightarrow(1-\theta(\lambda,r,C,B))(\delta_{\emptyset}\otimes\pi)+\theta(\lambda,r,C,B)\overline{\nu}
	\end{equation*}
	for all $C\subset V$ and all $B\subset E$.
\end{proof}
\begin{remark}
\label{rem:complconvextensionproof}
    Proposition~\ref{FirstCondition2} can  easily be extended to a broader class of graphs. We only need to assume that the graph $G$ is distance transitive, which is a stronger version of transitivity, see Remark \ref{ComparsisonRemark1}, since this property together with translation invariance provides us with the following symmetry
    \begin{equation*}
        \Pw^{(\{x\},\emptyset)}(\zero\in \bfC_t \text{ for some } t\geq 0)=\Pw^{(\{\zero\},\emptyset)}(x\in \bfC_t \text{ for some } t\geq 0),
    \end{equation*}
    which one needs in Proposition~\ref{FirstCondition2} to show $\Pw^{(C,B)}_{\lambda,r}(x\in \bfC_t \text{ i.o.})=\theta(\lambda,r,C,B)$.
    Therefore, as in the proof of Theorem~\ref{CPDPCompleteConvergenceCPERE} we can compare a CPERE on a distance transitive graph $G$ with a CPDP on the $1$-dimensional integer lattice to conclude complete convergence for a certain choice of parameters, as described in Remark~\ref{ComparsisonRemark1}.
\end{remark}

\vspace{1em}
\noindent
\textbf{Acknowledgments.} We would like to thank the anonymous referees for carefully reading the manuscript and for giving us many useful suggestions. This helped us greatly to improve the manuscript. We would like to especially thank one of the referees for pointing out an error in the proof of Proposition~\ref{IndependenceGrowthCond}. 
Fixing the error even allowed us to slightly strengthen the statement.

%\printbibliography
\end{document}